\documentclass[reqno,oneside,11pt]{amsart}

\usepackage[dvipsnames,svgnames,x11names,hyperref]{xcolor}

\usepackage[cmtip]{xy}
\xyoption{ps}
\xyoption{color}\UseCrayolaColors
\xyoption{dvips}
\xyoption{all}
\usepackage{lscape} 
\usepackage{enumerate}
\usepackage{txfonts} 
\usepackage{ctable}

\theoremstyle{plain}
\newtheorem{maintheorem}{Theorem} 
\newtheorem{theorem}{Theorem}[section]
\newtheorem{corollary}[theorem]{Corollary}
\newtheorem{proposition}[theorem]{Proposition}
\newtheorem{lemma}[theorem]{Lemma}
\newtheorem{minoration}[theorem]{Lower bound}
\newtheorem*{IH}{Induction Hypothesis}

\theoremstyle{definition}
\newtheorem{mainquestion}{Question} 
\newtheorem{definition}[theorem]{Definition}
\newtheorem{question}[theorem]{Question}
\newtheorem{example}[theorem]{Example}
\newtheorem{remark}[theorem]{Remark}

\newtheorem{caselist}[theorem]{List of Cases}

\newcommand{\mygraph}[1]{\xybox{\xygraph{#1}}} 

\newcommand{\C}{\mathbb{C}}
\newcommand{\R}{\mathbb{R}}

\newcommand{\Z}{\mathbb{Z}}
\newcommand{\N}{\mathbb{N}}
\newcommand{\A}{{\mathbb{A}}}
\newcommand{\p}{\mathbb{P}}
\renewcommand{\H}{\mathbb{H}}

\renewcommand{\a}{{\mathcal{A}}}
\newcommand{\T}{\mathcal{T}}
\renewcommand{\P}{\mathcal{P}}
\renewcommand{\LL}{\mathcal{L}}
\newcommand{\Sq}{\mathcal{S}}

\DeclareMathOperator{\PSL}{PSL}
\DeclareMathOperator{\GL}{GL}
\DeclareMathOperator{\SL}{SL}
\DeclareMathOperator{\OO}{O}
\DeclareMathOperator{\SO}{SO}

\DeclareMathOperator{\V}{V}

\newcommand{\Comp}{\mathcal{C}}

\renewcommand{\phi}{\varphi}
\newcommand{\id}{\text{\rm id}}
\newcommand{\RightArrowIsomorphism}{\xrightarrow{\,\smash{\raisebox{-0.65ex}{\ensuremath{\scriptstyle\sim}}}\,}}

\DeclareMathOperator{\Ker}{Ker}
\DeclareMathOperator{\Aut}{Aut}
\DeclareMathOperator{\Bir}{Bir}
\DeclareMathOperator{\Stab}{Stab}
\DeclareMathOperator{\SStab}{SStab}

\DeclareMathOperator{\degmax}{degmax}
\DeclareMathOperator{\degsum}{degsum}

\DeclareMathOperator{\Min}{Min}
\DeclareMathOperator{\CAT}{CAT}
\DeclareMathOperator{\Tame}{Tame}
\DeclareMathOperator{\STame}{STame}
\DeclareMathOperator{\Isom}{Isom}
\DeclareMathOperator{\Vect}{Vect}
\DeclareMathOperator{\Spec}{Spec}
\DeclareMathOperator{\ged}{gdeg}

\DeclareMathOperator{\Jac}{Jac}
\DeclareMathOperator{\jj}{j}
\DeclareMathOperator{\Nor}{N}
\DeclareMathOperator{\ev}{ev}

\newcommand{\TSL}{\Tame(\SL_2)}
\newcommand{\STSL}{\STame(\SL_2)}
\newcommand{\TQ}{\Tame_q(\C^4)}
\newcommand{\STQ}{\STame_q(\C^4)}
\newcommand{\Autq}{\Aut_q(\C^4)}
\newcommand{\Et}{E^{12}}
\newcommand{\Eb}{E_{34}}
\newcommand{\Er}{E^2_4}
\newcommand{\El}{E^1_3}
\newcommand{\wght}{{\textbf{\rm w}}}
\renewcommand{\hom}[1]{{#1}^\wght}

\newcommand{\Rgen}{\textit{R}_{\text{\rm gen}}}

\def\polhk#1{\setbox0=\hbox{#1}{\ooalign{\hidewidth \lower1.5ex\hbox{`}\hidewidth\crcr\unhbox0}}}

\newcommand{\smat}[4]{\left( \begin{smallmatrix} #1\,&#2\\ #3\,&#4 \end{smallmatrix} \right)}
\newcommand{\mat}[4]{\left( \begin{matrix} #1\,&#2\\ #3\,&#4 \end{matrix} \right)}
\newcommand{\sbmat}[4]{\left[\begin{smallmatrix} #1\,&#2\\ #3\,&#4 \end{smallmatrix} \right]}
\newcommand{\bmat}[4]{\left[\begin{matrix} #1\,&#2\\ #3\,&#4 \end{matrix} \right]}
\newcommand{\blob}{\rule[.2ex]{0.8ex}{0.8ex}}

\newcommand*\scolvec[3][]{\left(\begin{smallmatrix}\ifx\relax#1\relax\else#1\\\fi#2\\#3\end{smallmatrix} \right)}
\newcommand*\colvec[3][]{\left(\begin{matrix}\ifx\relax#1\relax\else#1\\\fi#2\\#3\end{matrix} \right)}

\newcolumntype{L}{>{$}l<{$}}
\newcolumntype{C}{>{$}c<{$}}

\usepackage[pdfauthor={C. Bisi \& J.-P. Furter \& S. Lamy}, colorlinks, linktocpage, citecolor = Navy, linkcolor = Navy]{hyperref}
\usepackage[all]{hypcap}        

\title[The tame automorphism group of an affine quadric threefold]{The tame automorphism group of an affine quadric threefold acting on a square complex}
\author{Cinzia Bisi, Jean-Philippe Furter and St\'ephane Lamy}
\thanks{This research was partially supported by ANR Grant ``BirPol''  ANR-11-JS01-004-01, by the Research Network Program GDRE-GRIFGA and by GNSAGA-INDAM, Rome, Italy.}
\address{Dipartimento Matematica ed Informatica, Universit\'a di Ferrara, Via Machiavelli n.35, 44121 Ferrara, Italy}
\email{bsicnz@unife.it}
\address{Laboratoire MIA, Universit\'e de La Rochelle, Avenue Michel Cr\'epeau,
 17000 La Rochelle, France}
\email{jpfurter@univ-lr.fr}
\address{Institut de Math\'ematiques de Toulouse, Universit\'e Paul Sabatier, 118 route de Narbonne, 31062 Toulouse Cedex 9, France}
\email{slamy@math.univ-toulouse.fr}

\begin{document}

\begin{abstract}
We study the group $\TSL$ of tame automorphisms of a smooth affine 3-dimensional quadric, which we can view as the underlying variety  of $\SL_2(\C)$.
We construct a square complex on which the group admits a natural cocompact action, and we prove that the complex is $\CAT(0)$ and hyperbolic.
We propose two applications of this construction: 
We show that any finite subgroup in $\TSL$ is linearizable, and that $\TSL$ satisfies the Tits alternative.
\end{abstract}

\maketitle

\setcounter{tocdepth}{2}
\begin{footnotesize}\tableofcontents\end{footnotesize}

\section*{Introduction}

The structure of transformation groups of rational surfaces is  quite well understood. By contrast, the higher dimensional case is still essentially a \textit{terra incognita}. This paper is an attempt to explore some aspects of transformation groups of rational 3-folds.

The ultimate goal would be to understand the structure of the whole Cremona group $\Bir(\p^3)$. Since this seems quite a formidable task, it is natural to break down the study by looking at some natural subgroups of $\Bir (\p^3)$, with the hope that this gives an idea of the properties to expect in general. We now list a few of these subgroups, in order to give a feeling about where our modest subgroup $\TSL$ fits into the bigger picture. A first natural subgroup is the  monomial group $\GL_3( \Z)$, where a matrix $(a_{ij})$ is identified to a birational map of $\C^3$ by taking 
$$(x,y,z) \dasharrow (x^{a_{11}}y^{a_{12}}z^{a_{13}},x^{a_{21}}y^{a_{22}}z^{a_{23}}, x^{a_{31}}y^{a_{32}}z^{a_{33}}).$$
Another natural subgroup is the group of polynomial automorphisms of $\C^3$. 
These two examples seem at first glance quite different in nature, nevertheless it turns out that both are contained in the subgroup $\Bir_0(\p^3)$ of birational transformations of genus 0, which are characterized by the fact that they admit a resolution by blowing-up points and rational curves (see \cite{Fr, Lgenus}). 
On the other hand, it is known (see \cite{Pan}) that given a smooth curve $C$ of arbitrary genus, there exists an element $f$ of $\Bir(\p^3)$ with the property that any resolution of $f$ must involve the blow-up of a curve isomorphic to $C$. 
So we must be aware that even if a full understanding of the group $\Aut(\C^3)$ still seems far out of reach, this group $\Aut(\C^3)$ is such a small subgroup of $\Bir(\p^3)$ that it might turn out not to be a good representative of the wealth of properties of the whole group $\Bir(\p^3)$. 
\begin{figure}[ht]
$$
\xymatrix@R=1pc{
& & \GL_3(\Z) \ar@{}[dl]_(0.4){\rotatebox{45}{\footnotesize $\supset$}}\\
\Bir (\p^3) &  \Bir_0 (\p^3) \ar@{}[l]|{\supset} & \Aut(\C^3) \ar@{}[l]|{\supset} & \Tame(\C^3) \ar@{}[l]|{\supset}\\
& & \Aut(\SL_2) \ar@{}[ul]^(0.4){\rotatebox{-45}{\footnotesize $\supset$}} &  \TSL \ar@{}[l]|{\supset}
}
$$
\caption{A few subgroups of $\Bir(\p^3)$.}\label{fig:birp3}
\end{figure}

The group $\Aut(\C^3)$ is just a special instance of the following construction: given a rational affine 3-fold $V$, the group $\Aut(V)$ can be identified with a subgroup of $\Bir(\p^3)$. 
Apart from $V = \C^3$, another interesting example is when $V\subseteq \C^4$ is an affine quadric 3-fold, say $V$ is the underlying variety of $\SL_2$.
In this case the group $\Aut(V)$ still seems quite redoubtably difficult to study.
We are lead to make a further restriction and to consider only the smaller group of \emph{tame} automorphisms, either in the context of $V =\C^3$ or $ \SL_2$.

The definition of the tame subgroup for $\Aut(\C^n)$ is classical. Let us recall it in dimension 3. 
The tame subgroup $\Tame(\C^3)$ is the subgroup of $\Aut(\C^3)$ generated by the affine group $A_3 = \GL_3 \ltimes \C^3$ and by elementary automorphisms of the form $(x,y,z) \mapsto (x+P(y,z),y,z)$. 
A natural analogue in the case of an affine quadric 3-fold was given recently in \cite{LV}. This is the group $\Tame(\SL_2)$, which will be the main object of our study in this paper.

When we consider the 2-dimensional analogues of the groups in Figure \ref{fig:birp3}, we obtain in particular the Cremona group $\Bir(\p^2)$, the monomial group $\GL_2(\Z)$ and the group of polynomial automorphisms $\Aut(\C^2)$. 
A remarkable feature of these groups is that they all admit natural actions on some sort of hyperbolic spaces. 
For instance the group $\SL_2(\Z)$ acts on the hyperbolic half-plane $\H^2$, since $\PSL_2(\Z) \subseteq \PSL_2(\R) \simeq \Isom_+(\H^2)$. 
But $\SL_2(\Z)$ also acts on the Bass-Serre tree associated with the structure of amalgamated product $\SL_2(\Z) \simeq \Z/4 *_{\Z/2} \Z/6$. 
A tree, or the hyperbolic plane $\H^2$, are both archetypal examples of spaces which are hyperbolic in the sense of Gromov.
The group $\Aut(\C^2)$ also admits a structure of amalgamated product. 
This is the classical theorem of Jung and van der Kulk, which states that $\Aut(\C^2) = A_2 *_{A_2 \cap E_2} E_2$, where $A_2$ and $E_2$ are respectively the subgroups of affine and triangular automorphisms.
So $\Aut(\C^2)$ also admits an action on a Bass-Serre tree, and we know since the work of Danilov and Gizatullin \cite{DanGiz} that the same is true for many other affine rational surfaces.
Finally, it was recently realized that the whole group $\Bir(\p^2)$ also acts on a hyperbolic space, via a completely different construction: By simultaneously considering all possible blow-ups over $\p^2$, it is possible to produce an infinite dimensional analogue of $\H^2$ on which the Cremona group acts by isometries (see \cite{Can, CL}). 
 
With these facts in mind, given a 3-dimensional transformation group it is natural to look for an action of this group on some spaces with non-positive curvature, in a sense to be made precise.
Considering the case of monomial maps, we have a natural action of $\SL_3(\Z)$  
on the symmetric space $\SL_3(\R)/\SO_3(\R)$, see \cite[II.10]{BH}. 
The later space is a basic example of a $\CAT(0)$ symmetric space.
Recall that a $\CAT(0)$ space is a geodesic metric space where all triangles are thinner than their comparison triangles in the Euclidean plane.
We take this as a hint that $\Bir(\p^3)$ or some of its subgroups should act on spaces of non-positive curvature.
At the moment it is not clear how to imitate the construction by inductive limits of blow-ups to obtain a space say with the $\CAT(0)$ property, so we try to generalize instead the more combinatorial approach of the action on a Bass-Serre tree.
The group $\Tame(\C^3)$ does not possess an obvious structure of amalgamated product over two of its subgroups.
We should mention here that it was recently observed in \cite{wright} that the group $\Tame(\C^3)$ can be described as the amalgamation of three of its subgroups along pairwise intersections; in fact a similar structure exists for the Cremona group $\Bir(\p^2)$ as was noted again by Wright a long time ago (see \cite{W}).
Such an algebraic structure yields an action of $\Tame(\C^3)$ on a natural simply connected 2-dimensional simplicial complex.
However it is still not clear if this complex has good geometric properties, and so it is not immediate to answer the following:

\begin{mainquestion}\label{que:hyp}
Is there a natural action of $\Tame(\C^3)$ on some hyperbolic and/or $\CAT(0)$  space?
\end{mainquestion}

Of course, this question is rather vague.
In our mind an action on some hyperbolic space would qualify as a ``good answer'' to Question \ref{que:hyp} if it allows to answer the following questions, which we consider to be basic tests about our understanding of the group:

\begin{mainquestion} \label{que:linearizable}
Is any finite subgroup in $\Tame(\C^3)$ linearizable?
\end{mainquestion}

\begin{mainquestion} \label{que:tits}
Does $\Tame(\C^3)$ satisfy the Tits alternative?
\end{mainquestion}

To put this into a historical context, let us review briefly the similar questions in dimension 2. 
The fact that any finite subgroup in $\Aut(\C^2)$ is linearizable is classical (see for instance \cite{Kambayashi, Furushima}).
The Tits alternative for $\Aut(\C^2)$ 
and $\Bir(\p^2)$ were proved respectively in \cite{Ltits} and \cite{Can}, and the proofs involve the actions on the hyperbolic spaces previously mentioned.

Now we come to the group $\TSL$. We define it as the restriction to $\SL_2$ of the subgroup $\TQ$ of $\Aut(\C^4)$ generated by $\OO_4$ and $\Er$, where  $\OO_4$ is the complex orthogonal group associated with the quadratic form given by the determinant $q=x_1x_4-x_2x_3$, and $$\Er= \left\lbrace \mat{x_1}{x_2}{x_3}{x_4} \mapsto \mat{x_1}{x_2 + x_1P(x_1,x_3)}{x_3}{x_4+x_3P(x_1,x_3)}; \;P \in \C[x_1,x_3] \right\rbrace .$$
One possible generalization of simplicial trees are $\CAT(0)$ cube complexes (see \cite{Wi}).
We  briefly explain how we construct a square complex on which this group acts cocompactly (but certainly not properly!).
Each element of $\TSL$ can be written as $f=\smat{f_1}{f_2}{f_3}{f_4}$.
Modulo some identifications that we will make precise in Section \ref{sec:complexes}, we associate vertices to each component $f_i$, to each row or column $(f_1,f_2)$, $(f_3,f_4)$, $(f_1,f_3)$, $(f_2,f_4)$ and to the whole automorphism $f$. On the other hand, (undirected) edges correspond to inclusion of a component inside a row or column, or of a row or column inside an automorphism.
This yields a graph, on which we glue squares to fill each loop of four edges (see Figure \ref{fig:bigsquare}), to finally obtain a square complex $\Comp$.  
   
In this paper we answer analogues of Questions \ref{que:hyp} to \ref{que:tits} in the context of the group $\TSL$.
The main ingredient in our proofs is a natural action by isometries on the complex $\Comp$, which admits good geometric properties:

\begin{maintheorem} \label{thm:mainHyp}
The square complex $\Comp$ is $\CAT(0)$ and hyperbolic.
\end{maintheorem}

As a sample of possible applications of such a construction we obtain:
 
\begin{maintheorem} \label{thm:mainLin}
Any finite subgroup in $\Tame(\SL_2)$ is linearizable, that is, conjugate to a subgroup of the orthogonal group $\OO_4$.
\end{maintheorem}

\begin{maintheorem} \label{thm:mainTits}
The group $\Tame(\SL_2)$ satisfies the Tits alternative, that is, for any subgroup $G \subseteq \Tame(\SL_2)$ we have:
\begin{enumerate}
\item either $G$ contains a solvable subgroup of finite index;
\item or $G$ contains a free subgroup of rank 2.
\end{enumerate}
\end{maintheorem}

The paper is organized as follows.
In Section \ref{sec:tameSL2} we gather some definitions and facts about the groups $\TSL$ and $\OO_4$.
The square complex is constructed in Section \ref{sec:complexes}, and some of its basic properties are established. 
Then in Section \ref{sec:complexes_bis} we study its geometry: links of vertices, non-positive curvature, simple connectedness, hyperbolicity. In particular, we obtain a proof of Theorem \ref{thm:mainHyp}.
In studying the geometry of the square complex one realizes that some simplicial trees naturally appear as substructures, for instances in the link of some vertices, or as hyperplanes of the complex. 
At the algebraic level this translates into the existence of many amalgamated product structures for subgroups of $\TSL$.
In Section \ref{sec:structures} we study in details some of these products, which are reminiscent of Russian nesting dolls (see Figure \ref{fig:nesting}). 
Groups acting on $\CAT(0)$ spaces satisfy nice properties: for instance any such finite group admits a fixed point. 
In Section \ref{sec:theorems} we exploit such geometric properties to give a proof of Theorems \ref{thm:mainLin} and \ref{thm:mainTits}.
In Section \ref{sec:complements} we give some examples of elliptic, parabolic and loxodromic subgroups, which appear in the proof of the Tits alternative. We also briefly discuss the case of $\Tame(\C^3)$, comment on the recent related work \cite{wright}, and propose some open questions.  
Finally we gather in an annex some reworked results from \cite{LV} about the theory of elementary reductions on the groups $\TSL$ and $\TQ$. 

\section{Preliminaries} \label{sec:tameSL2}

We identify $\mathbb{C}^4$ with the space of $ 2 \times 2$ complex matrices.
So a polynomial automorphism $f$ of $\C^4$ is denoted by
$$ f \colon \mat{x_1}{x_2}{x_3}{x_4} \mapsto  \mat{f_1}{f_2}{f_3}{f_4},$$
where $f_i \in \C[x_1,x_2,x_3,x_4]$ for $1 \le i \le 4$, 
or simply by $f = \smat{f_1}{f_2}{f_3}{f_4}$.
We choose to work with the smooth affine quadric given by the equation $q=1$, where  $q=x_1x_4-x_2x_3$ is the determinant:
$$ \SL_2= \left\lbrace  \mat{x_1}{x_2}{x_3}{x_4};\; x_1x_4-x_2x_3=1 \right\rbrace.$$  
We insist that we use this point of view for notational convenience, but we are interested only in the underlying variety of $\SL_2$.
In particular $\Aut(\SL_2)$ is the group of automorphism of $\SL_2$ as an affine variety, and not as an algebraic group.

We denote by $\Autq$ the subgroup of $\Aut(\C^4)$ of automorphisms preserving the quadratic form $q$:
\[\Autq = \{ f \in \Aut (\C^4); \; q \circ f =q \}.\]
We will often denote an element $f \in \Autq$ in an abbreviated form such as $f = \smat{f_1}{f_2}{f_3}{\dots}$: Here the dots should be replaced by the unique polynomial $f_4$ such that $f_1f_4 - f_2f_3 = x_1x_4-x_2x_3$.
We call $\TQ$ the subgroup of $\Autq$ generated by $\OO_4$ and $\Er$, where $\OO_4 = \Autq \cap \GL_4$ is the complex orthogonal group associated with $q$, and $\Er$ is the group defined as
\[ \Er = \left\lbrace \mat{x_1}{x_2 + x_1P(x_1,x_3)}{x_3}{x_4+x_3P(x_1,x_3)}; \;P \in \C[x_1,x_3] \right\rbrace.\]
We denote by $\rho\colon \Autq \to \Aut(\SL_2)$ the natural restriction map, and we define the \textbf{tame group} of $\SL_2$, denoted by $\TSL$, to be the image of $\TQ$ by $\rho$.
We also define $\STQ$ as the subgroup of index 2 in $\TQ$ of automorphisms with linear part in $\SO_4$, and the \textbf{special tame group} $\STSL = \rho(\STQ)$.

\begin{remark}
The morphism $\rho$ is clearly injective in restriction to $\OO_4$ and to $\Er$: This justifies that we will consider $\OO_4$ and $\Er$ as subgroups of $\TSL$.
On the other hand it is less clear if $\rho$ induces an isomorphism between $\TQ$ and $\TSL$: It turns out to be true, but we shall need quite a lot of machinery before being in position to prove it (see Proposition \ref{pro:TQ=TSL}).  
Nevertheless by abuse of notation if $f = \smat{f_1}{f_2}{f_3}{f_4}$ is an element of $\TQ$ we will also consider $f$ as an element of $\TSL$, the morphism $\rho$ being implicit. 
See also Section \ref{sec:rho} for other questions around the restriction morphism $\rho$.
\end{remark}

The Klein four-group $\V_4$  will be considered as the following subgroup of  $\OO_4$:
\[\V_4 = \left\lbrace \id,\; \mat{x_4}{x_2}{x_3}{x_1},\; \mat{x_1}{x_3}{x_2}{x_4},\; \mat{x_4}{x_3}{x_2}{x_1} \right\rbrace .\]
In particular $\V_4$ contains the transpose automorphism $\tau = \smat{x_1}{x_3}{x_2}{x_4}$.

\subsection{Tame\texorpdfstring{$\mathbf{(\text{SL}_2)}$}{(SL2)}}

We now review some results which are essentially contained in \cite{LV}.
However, we adopt some slightly different notations and definitions.
For the convenience of the reader, we give self-contained proofs of all needed results in an annex.

We define a degree function on $\C[x_1,x_2,x_3,x_4]$ with value in $\N^4 \cup \{-\infty\}$ by taking
\begin{align*}
\deg_{\C^4} x_1 &= (2,1,1,0) &\deg_{\C^4} x_2 = (1,2,0,1) \\
\deg_{\C^4} x_3 &= (1,0,2,1) &\deg_{\C^4} x_4 = (0,1,1,2) 
\end{align*}
and by convention $\deg_{\C^4} 0 = -\infty$.
We use the graded lexicographic order on $\N^4$ to compare degrees.
We obtain a degree function on the algebra $\C[\SL_2] = \C[x_1,x_2,x_3,x_4]/(q-1)$ by setting
$$\deg p = \min\{ \deg_{\C^4} r;\; r \equiv p \mod (q-1) \}.$$
We define two notions of degree for an automorphism $f = \smat{f_1}{f_2}{f_3}{f_4} \in \TSL$:
\begin{align*}
\degsum f &= \sum_{1 \, \leq \, i \, \leq \, 4} \deg f_i;\\
\degmax f &= \max_{1 \, \leq \, i \, \leq \, 4} \deg f_i.
\end{align*}

\begin{lemma} \label{lem:degmax}
Let $f$ be an element in $\TSL$.
\begin{enumerate}
\item For any $u \in \OO_4$, we have $\degmax f =\degmax  u \circ f$.
\item We have $f \in \OO_4$ if and only if $\degmax f = (2,1,1,0)$.
\end{enumerate}
\end{lemma}

\begin{proof}
\begin{enumerate}
\item Clearly $\degmax  u \circ f \leq \degmax f$, and similarly we get $ \degmax f = \degmax u^{-1} \circ (u \circ f) \leq \degmax u \circ f$.
\item This follows from the fact that if $P \in \C[x_1,x_2,x_3,x_4]$ with $\deg_{\C^4} P = (i,j,k,l)$, then the ordinary degree of $P$ is the average $\frac14 (i+j+k+l)$. \qedhere
\end{enumerate}
\end{proof}

The degree $\degsum$ was the one used in \cite{LV}, with a different choice of weights with value in $\N^3$.
Because of the nice properties in Lemma \ref{lem:degmax} we prefer to use $\degmax$, together with the above choice of weights.
The choice to use a degree function with value in $\N^4$ is mainly for aesthetic reasons, on the other hand the property that the ordinary degree is recovered by taking mean was the main impulse to change the initial choice.
From now on we will never use $\degsum$, and we simply denote $\deg = \degmax$.
 
An \textbf{elementary automorphism} (resp. a \textbf{generalized elementary automorphism}) is an element $e \in \TSL$ of the form
\[e = u \mat{x_1}{x_2 + x_1P(x_1,x_3)}{x_3}{x_4+x_3P(x_1,x_3)} u^{-1}\]
where $P \in \C[x_1,x_3]$, $u \in \V_4$ (resp. $u \in \OO_4$). 
Note that any  elementary automorphism belongs to (at least) one of the four subgroups
$\Et$, $\Eb$, $\El$, $\Er$ of $\TSL$ respectively defined as the set of elements of the form
\begin{multline*}
\mat{x_1+x_3Q(x_3,x_4)}{x_2+x_4Q(x_3,x_4)}{x_3}{x_4}, \; \mat{x_1}{x_2}{x_3+x_1Q(x_1,x_2)}{x_4+x_2Q(x_1,x_2)}, \\
\mat{x_1+x_2Q(x_2,x_4)}{x_2}{x_3+x_4Q(x_2,x_4)}{x_4}, \; \mat{x_1}{x_2+x_1Q(x_1,x_3)}{x_3}{x_4+x_3Q(x_1,x_3)}  ,
\end{multline*}
where $Q$ is any polynomial in two indeterminates.

We say that $f \in \TSL$ admits an \textbf{elementary reduction} if there exists an elementary automorphism $e$ such that $\deg e\circ f < \deg f$.
In \cite{LV}, the definition of an elementary automorphism is slightly different. 
However all these changes -- new weights, new degree, new elementary reduction -- do not affect the formulation of the main theorem; in fact it simplifies the proof:

\begin{theorem}[see Theorem \ref{thm:main}]\label{th:main theorem of LV}
Any non-linear element of $\TSL$ admits an elementary reduction.
\end{theorem}

Since the graded lexicographic order of $\N^4$ is a well-ordering, Theorem \ref{th:main theorem of LV} implies that any element $f$ of $\TSL$ admits a finite sequence of elementary reductions
\[ f \to e_1 \circ f \to e_2 \circ e_1 \circ f \to \cdots \to e_n  \circ \cdots \circ e_1 \circ f\]
such that the last automorphism is an element of $\OO_4$.

An important technical ingredient of the proof is the following lemma, which tells that under an elementary reduction, the degree of both affected components decreases strictly.

\begin{lemma}[see Lemma \ref{lem:degree of each component drops}] \label{lem:both degree drop}
Let $f = \smat{f_1}{f_2}{f_3}{f_4} \in \TSL$.
If $e \in \El$ and 
$$e\circ f =  \mat{f'_1}{f_2}{f'_3}{f_4},$$ 
then
$$\deg e \circ f  \sphericalangle \deg f \Longleftrightarrow \deg f'_1 \sphericalangle \deg f_1  \Longleftrightarrow   \deg f'_3 \sphericalangle \deg f_3$$
for any relation $\sphericalangle$ among $<$, $>$, $\le$, $\ge$ and $=$.
\end{lemma}

A useful immediate corollary is:

\begin{corollary} \label{cor:linktype1}
Let $f = \smat{f_1}{f_2}{f_3}{f_4} \in \STSL$ be an automorphism such that $f_1 = x_1$.
Then $f$ is a composition of elementary automorphisms preserving $x_1$. 
In particular, $f_2$ and $f_3$ do not depend on $x_4$ and we can view $(f_2,f_3)$ as defining an element of the subgroup of $\Aut_{\C[x_1]}\C[x_1][x_2,x_3]$ generated by $(x_3,x_2)$ and automorphisms of the form $(ax_2 + x_1P(x_1,x_3), a^{-1}x_3)$. In particular, if $f_3=x_3$, there exists some polynomial $P$ such that $f_2=x_2 + x_1 P(x_1,x_3)$.
\end{corollary}

\begin{remark} \label{rem:justification of Er}
We obtain the following justification for the definition of the group $\Er$: 
Any automorphism $f=\smat{f_1}{f_2}{f_3}{f_4}$ in $\TSL$ such that $f_1=x_1$ and $f_3=x_3$ belongs to $\Er$.
\end{remark}

\begin{lemma}[see Lemma \ref{lem:minoration does not apply to both}] \label{lem:12inLV}
Let $f \in \TSL$, and assume there exist two elementary automorphisms  
\begin{align*}
e &=\mat{x_1 + x_3Q(x_3,x_4)}{x_2+x_4Q(x_3,x_4)}{x_3}{x_4} \in \Et, \\
e' &= \mat{x_1 + x_2P(x_2,x_4)}{x_2}{x_3+x_4P(x_2,x_4)}{x_4} \in \El,
\end{align*}
such that $\deg e \circ f \le \deg f$ and $\deg e' \circ f < \deg f$.\\
Then we are in one of the following cases:
\begin{enumerate}
\item $Q = Q(x_4) \in \C[x_4]$;
\item $P = P(x_4) \in \C[x_4]$;
\item There exists $R(x_4)\in \C[x_4]$ such that $\deg (f_2 + f_4R(f_4) ) < \deg f_2$;
\item There exists $R(x_4)\in \C[x_4]$ such that $\deg ( f_3 + f_4R(f_4) ) < \deg f_3$.
\end{enumerate}
\end{lemma}

\subsection{Orthogonal group} \label{sec:O4}

\subsubsection{Definitions} \label{sec:definitions about O4}

Recall that we denote by $\OO_4$ the orthogonal group of $\C^4$ associated with the quadratic form $q = x_1x_4 - x_2x_3$.
We have $\OO_4= \langle \SO_4, \tau \rangle$, where $\tau = \smat{x_1}{x_3}{x_2}{x_4}$ denotes the involution given by the transposition.
The $2\mathbin{:}1$ morphism of groups  
\begin{align*}
\SL_2 \times \SL_2 &\to  \SO_4 \\
(A,B)\quad &\mapsto  A \cdot \mat{x_1}{x_2}{x_3}{x_4} \cdot B^t
\end{align*}
is the universal cover of $\SO_4$. 
Here the product $A \cdot \smat{x_1}{x_2}{x_3}{x_4} \cdot B^t$ actually denotes the usual product of matrices. 
However, if $f= \smat{f_1}{f_2}{f_3}{f_4}$ and $g= \smat{g_1}{g_2}{g_3}{g_4}$ are elements of $\OO_4$, their composition is
$$f \circ g  =  \mat{f_1 \circ g}{f_2 \circ g}{f_3 \circ g}{f_4 \circ g} \in \OO_4$$
which must not be confused with the product of the $2 \times 2$ matrices 
$\smat{f_1}{f_2}{f_3}{f_4}$ and $\smat{g_1}{g_2}{g_3}{g_4}$!
(see also Remark \ref{rem:remark on notation} below).

\subsubsection{Dual quadratic form} \label{sec:dual}

We now study the totally isotropic spaces of a quadratic form on the dual of $\C^4$ in order to understand the geometry of the group $\OO_4$. 

In this section we set $V =\C^4$ and we denote by $V^*$ the dual of $V$. 
We denote respectively by $e_1,e_2,e_3,e_4$ and $x_1,x_2,x_3,x_4$ the canonical basis of $V$ and the dual basis of $V^*$. 
Since $q(x)= x_1x_4-x_2x_3$ is a non degenerate quadratic form on $V$, there is a non degenerate quadratic form $q^*$ on $V^*$ corresponding to $q.$ Moreover, any endomorphism $f$ of $V$ belongs to the orthogonal group $\OO (V,q)$ if and only if its transpose $f^t$ belongs to the orthogonal group  $\OO (V^*,q^*)$. 
In other words, we have $q \circ f =q$ if and only if $q^* \circ f^t = q^*$.
Since the matrix of $q$ in the canonical basis is
$A= \frac{1}{2}
\left(
\begin{smallmatrix}
0 &  0 &  0 & 1\\
0 &  0 & -1 \rule{1mm}{0mm} & 0 \\
0 & -1 \rule{1mm}{0mm}  & 0  & 0 \\
1 &  0 & 0  & 0
\end{smallmatrix}
\right)$, then, the matrix of $q^*$ in the dual basis is $A^{-1}=2
\left(
\begin{smallmatrix}
0 &  0 &  0 & 1\\
0 &  0 & -1\rule{1mm}{0mm} & 0 \\
0 & -1\rule{1mm}{0mm} & 0  & 0 \\
1 &  0 & 0  & 0
\end{smallmatrix}
\right)
$.
We denote by $\langle \cdot \,, \, \cdot \rangle$ the bilinear pairing $V^* \times V^* \to \C$ associated with $\frac{1}{4}q^*$ (so that its matrix in the dual basis is $\frac{1}{4}A^{-1} = A$).

\begin{remark} \label{rem:remark on notation}
In this paper, each element of $\OO_4$ is denoted in a rather unusual way as a $2 \times 2$ matrix of the form  $f= \smat{f_1}{f_2}{f_3}{f_4}$, where each $f_i =\sum_jf_{i,j}x_j$, $f_{i,j} \in \C$, is an element of $V^*$. 
The corresponding more familiar $4 \times 4$ matrix is $M:=(f_{i,j})_{1 \, \leq \, i,j \, \leq \, 4} \in M_4(\C)$ and it satisfies the usual equality $M^tAM=A$.
\end{remark}

\begin{lemma}\label{lem:description of the orthogonal group}
Consider $f= \smat{f_1}{f_2}{f_3}{f_4}$, where the elements $f_k$ belong to $V^*$. Then, the following assertions are equivalent:
\begin{enumerate}
\item  $f \in \OO_4$;
\item $\langle f_i, f_j \rangle =  \langle x_i, x_j \rangle$ for all $i,j\in \{1,2,3,4 \} $.
\end{enumerate}
\end{lemma}

\begin{proof}
Observe first that $f^t(x_i) = f_i(x_1, \dots, x_4)$ for $i = 1, \dots, 4$.
Then, we have seen that $f \in \OO_4$ if and only if $f^t$ belongs to the orthogonal group $\OO (V^*, \frac{1}{4}q^*)$, i.e. if and only if for any $x,y \in V^*$, we have $\langle f^t(x), f^t(y) \rangle  =  \langle x,y \rangle$. 
This last equality is satisfied for all $x,y \in V^*$ if and only if it is satisfied for any $x,y \in \{x_1,x_2,x_3,x_4 \}$.
\end{proof}

Recall that a subspace $W \subseteq V^*$ is \textbf{totally isotropic} (with respect to $q^*$) if for all $x,y \in W$, $\langle x, y\rangle = 0$.
 
\begin{lemma}
Let $f_1,f_2$ be linearly independent elements of $V^*$. The following assertions are equivalent:
\begin{enumerate}
\item $\Vect(f_1,f_2)$ is totally isotropic ;
\item There exist $f_3,f_4 \in V^*$ such that  $\smat{f_1}{f_2}{f_3}{f_4} \in \OO_4$.
\end{enumerate}
\end{lemma}

\begin{proof}
If  $\smat{f_1}{f_2}{f_3}{f_4} \in \OO_4$, then by Lemma \ref{lem:description of the orthogonal group} for any $i,j \in \{ 1,2 \}$ we have $\langle f_i, f_j \rangle = \langle x_i, x_j \rangle = 0$.

Conversely, if $\langle f_i, f_j \rangle = \langle x_i, x_j \rangle =  0$ for any $i,j \in \{ 1,2 \}$, by Witt's Theorem (see e.g. \cite[p. 58]{S}) we can extend the map $x_1 \mapsto f_1$, $x_2 \mapsto f_2$ as an isometry $V^* \to V^*$. 
Then denoting by $f_3,f_4$ the images of $x_3, x_4$, we have $\langle f_i, f_j \rangle = \langle x_i, x_j \rangle$ for all $i,j \in \{1,2,3,4 \}$. We conclude by  Lemma \ref{lem:description of the orthogonal group}.
\end{proof}

If $\smat{f_1}{f_2}{f_3}{f_4} \in \OO_4$, the planes $\Vect(f_1,f_2)$,  $\Vect(f_3,f_4)$, $\Vect(f_1,f_3)$ and  $\Vect(f_2,f_4)$ are totally isotropic. 
Moreover the following decompositions hold: 
$$V^*= \Vect(f_1,f_2) \oplus \Vect(f_3,f_4) \quad\text{ and }\quad V^*= \Vect(f_1,f_3) \oplus \Vect(f_2,f_4).$$
We have the following reciprocal result.

\begin{lemma}\label{lem:uniqueauto}
Let $W$ and $W'$ be two totally isotropic planes of $V^*$ such that $V^*= W \oplus W'$. Then for any  basis $(f_1,f_2)$ of $W$, there exists a unique basis $(f_3,f_4)$ of $W'$ such that  $\smat{f_1}{f_2}{f_3}{f_4} \in \OO_4$.
\end{lemma}

\begin{proof}
\noindent \underline{Existence.}
By Witt's Theorem, we may assume that $f_1=x_1$ and $f_2=x_2$.
Let $f_3,f_4$ be a basis of $W'$. 
If we express them in the basis $x_1,x_2,x_3,x_4$, we get 
$f_3= a_1 x_1 + a_2 x_2 + a_3 x_3 + a_4 x_4$ and $f_4= b_1 x_1 + b_2 x_2 + b_3 x_3 + b_4 x_4$.
Since $x_1,x_2,f_3,f_4$ is a basis of $V^*$, we get $\det \smat{a_3}{a_4}{b_3}{b_4} \neq 0$.
Therefore, up to replacing $f_3$ and $f_4$ by some linear combinations, we may assume that $ \smat{a_3}{a_4}{b_3}{b_4}= \smat{1}{0}{0}{1}$, i.e. $f_3=  a_1 x_1 + a_2 x_2 + x_3$ and $f_4 = b_1 x_1 + b_2 x_2 + x_4$.

Since $\langle f_3,f_3 \rangle =  - a_2$ and $\langle f_4,f_4 \rangle =  b_1$, we get $a_2=b_1 = 0$. 
Finally,  $\langle f_3,f_4 \rangle = \frac{1}{2} (a_1- b_2)$, so that $a_1 = b_2$, $f_3 =x_3 + a_1 x_1$ and  $f_4 = x_4+ a_1 x_2$. 

Now it is clear that $\smat{x_1}{x_2}{x_3 + a_1 x_1}{x_4+a_1x_2} \in \OO_4$.

\noindent \underline{Unicity.} Let $(f_3,f_4)$ and $(\tilde{f}_3,\tilde{f}_4)$ be two bases of $W'$ such that  $\smat{f_1}{f_2}{f_3}{f_4}$  and $\smat{f_1}{f_2}{\tilde{f}_3}{\tilde{f}_4}$ belong to $\OO_4$. 
From $f_1f_4-f_2f_3= f_1\tilde{f}_4 -f_2 \tilde{f}_3$, we get $f_2( \tilde{f}_3 -f_3) =f_1(\tilde{f}_4 -f_4)$. Since $f_1$ and $f_2$ are coprime, we get the existence of a polynomial $\lambda$ such that $ \tilde{f}_3 -f_3 = \lambda f_1$ and $\tilde{f}_4 -f_4 = \lambda f_2$. 
Since the $f_i$ and $\tilde f_j$ are linear forms, we see that $\lambda$ is a constant.
This proves that $ \tilde{f}_3 -f_3$ and $ \tilde{f}_4 -f_4$ are elements in $W \cap W' =\{ 0 \}$, and we obtain $(f_3,f_4)=(\tilde{f}_3,\tilde{f}_4)$.
\end{proof}

\begin{lemma}\label{lem:isotropicsecantplanes}
For any nonzero isotropic vector $f_1$ of $V^*$, there exist exactly two  totally isotropic planes of $V^*$ containing $f_1$.  Furthermore, they are of the form $\Vect(f_1,f_2)$ and $\Vect(f_1,f_3)$, where $\smat{f_1}{f_2}{f_3}{\dots}$ is an element of $\OO_4$.
\end{lemma}

\begin{proof}
By Witt's theorem, we may assume that $f_1=x_1$. 
Any totally isotropic subspace $W$ in $V^*$ containing $x_1$ is included into $x_1^{\perp}= \Vect(x_1,x_2,x_3)$. 
Therefore, there exist $a_2, a_3 \in \C$ such that
$W= \Vect(x_1, a_2 x_2 + a_3 x_3)$. 
Finally, since $q^*(a_2 x_2 + a_3 x_3) = - 4 a_2 a_3 = 0$ (recall that $q^*(u) = 4 \langle u,u \rangle$ for any $u \in V^*$), $W$ is equal to $\Vect(x_1, x_2)$ or  $\Vect(x_1, x_3)$.
\end{proof}

\begin{lemma}\label{lem:pair of isotropic planes}
Let $W$ and $W'$ be two totally isotropic planes of $V^*$.
Then there exists $f \in \OO_4$ such that $f(W) = \Vect(x_3,x_4)$ and $f(W')$ is one of the following three possibilities:
\begin{enumerate}
\item $f(W') = \Vect(x_3,x_4)$;
\item $f(W') = \Vect(x_1,x_2)$;
\item $f(W') = \Vect(x_2,x_4)$.
\end{enumerate}
\end{lemma}

\begin{proof}
By Witt's theorem there exists  $f \in \OO_4$ such that $f(W) = \Vect(x_3,x_4)$.
If $W' = W$ we are in Case (1), and if $W \cap W' = \{0\}$ then we can apply Lemma \ref{lem:uniqueauto} to get Case (2).
Assume now that $W \cap W'$ is a line.
Again by Witt's theorem we can assume that $W \cap W' = \Vect(x_4)$, and then we conclude by Lemma \ref{lem:isotropicsecantplanes} that we are in Case (3).
\end{proof}

We can reinterpret the last two lemmas in geometric terms.

\begin{remark} \label{rem:geometry of isotropic cone}
The isotropic cone of $q^*$ is given by $a_1a_4 - a_2a_3 = 0$, where $f = a_1x_1 + \dots + a_4x_4 \in V^*$.
In particular this is a cone over a smooth quadric surface $S$ in $\p ( V^*) \simeq \p^3$.
Totally isotropic planes correspond to cones over a line in $S$, but $S$ is isomorphic to $\p^1 \times \p^1$, and lines in $S$ correspond to horizontal or vertical ruling.
From this point of view Lemma \ref{lem:isotropicsecantplanes} is just the obvious geometric fact that any point in $S$ belongs to exactly two lines, one vertical and the other horizontal. 
Similarly, Lemma \ref{lem:pair of isotropic planes} is the fact that $\OO_4$ acts transitively on pairs of disjoint lines, and on pairs of intersecting lines.
\end{remark}

\begin{corollary}\label{cor:pair of elementary automorphisms}
Let $e,e'$ be two generalized elementary automorphisms. Then, up to conjugation by an element of $\OO_4$, we may assume that $e'\in\El$ and that $e$ belongs to either $\El$, $\Er$ or $\Et$. 
\end{corollary}

\begin{proof}
Each generalized  elementary automorphism $e$ fixes pointwise (at least) a totally isotropic plane of $V^*$ (note that $e$ acts naturally on $\C[x_1,x_2,x_3,x_4]$). Observe furthermore that the plane $\Vect(x_3,x_4)$ is fixed if and only if $e$ belongs to $\Et$. Therefore, the result follows from Lemma \ref{lem:pair of isotropic planes}.
\end{proof}

In the next definition, the quadric $S$ is identified to $\p^1 \times \p^1$ via the isomorphism $\p^1 \times \p^1 \to S$ sending $((\alpha :\beta),(\gamma : \delta))$ to $ \C \, (\alpha \gamma  x_1+ \beta \gamma  x_2 + \alpha \delta x_3 + \beta \delta x_4)$.

\begin{definition}
A totally isotropic plane of $V^*$ is said to be horizontal (resp. vertical), if it corresponds to a horizontal (resp. vertical) line of $\p^1 \times \p^1$.
\end{definition}
 
The map sending $(a:b) \in \p^1$ to $\Vect(ax_1 +b x_3, ax_2 +b x_4)$ (resp.  $\Vect(ax_1 +b x_2, ax_3 +b x_4)$)  is a parametrization of the horizontal (resp. vertical) totally isotropic planes of $V^*$. Let $f$ be any element of $\OO_4$ and let $\Vect (u,v)$ be any totally isotropic plane of $V^*$. The group $\OO_4$ acts on the set of totally isotropic planes via the following formula
\[ f . \Vect (u,v) = \Vect (u \circ f^{-1}, v \circ f^{-1}).\]

\begin{lemma}\label{lem:action of O_4 with respect to horizontality}
Any element of $\SO_4$ sends a horizontal totally isotropic plane to a horizontal totally isotropic  plane, and a vertical totally isotropic plane to a vertical totally isotropic plane. Any element of $\OO_4 \smallsetminus \SO_4$ exchanges the horizontal and the vertical totally isotropic planes.
\end{lemma}

\begin{proof}
The set of totally isotropic planes of $V^*$ is parametrized by the disjoint union of two copies of $\p^1$. The group $\SO_4$ being connected, it must preserve each  $\p^1$. The element $\tau$ of $\OO_4 \smallsetminus \SO_4$ exchanges the horizontal totally isotropic plane $\Vect (x_1,x_2)$ and the vertical totally isotropic  plane $\Vect (x_1,x_3)$. The result follows.
\end{proof}

\begin{remark}
Let $\Delta := \{ (x,x), \; x \in \p^1 \}$ denote the diagonal of $\p^1\times \p^1$. Identify the set of horizontal totally isotropic planes to $\p^1$.
Remark that the map 
\begin{align*}
\SO _4 &\to (\p^1 \times \p^1) \setminus \Delta \\
f=\mat{f_1}{f_2}{f_3}{f_4} &\mapsto ( \Vect(f_1,f_2), \Vect(f_3,f_4) )
\end{align*}
is a fiber bundle, whose fiber is isomorphic to $\GL _2$. 
Indeed, by Lemma \ref{lem:uniqueauto}, any element  $g=\smat{g_1}{g_2}{g_3}{g_4}$ of $\SO _4$  satisfying $\Vect(g_1,g_2) =\Vect(f_1,f_2)$ and $\Vect(g_3,g_4)=\Vect(f_3,f_4)$ is uniquely determined by the basis $(g_1,g_2)$ of $\Vect(f_1,f_2)$.

In the same way, we obtain a fiber bundle  
\begin{align*}
\OO_4 \smallsetminus \SO_4 &\to (\p^1 \times \p^1) \setminus \Delta,\\
\mat{f_1}{f_2}{f_3}{f_4} &\mapsto ( \Vect(f_1,f_3), \Vect(f_2,f_4) ).
\end{align*}

\end{remark}

\section{Square Complex} \label{sec:complexes}

We now define  a square complex $\Comp$, which will be our main tool in the study of $\TSL$, and we state some of its basic properties.

\subsection{Definitions} \label{sec:definition of complex}

A function $f_1 \in \C[\SL_2] = \C[x_1,x_2,x_3,x_4]/(q-1)$ is said to be a {\bf component} if it can be completed to an element  $f = \smat{f_1}{f_2}{f_3}{f_4}$ of $\TSL$.
The vertices of our 2-dimensional complex are defined in terms of orbits of tuples of components, as we now explain.
For any element $f = \smat{f_1}{f_2}{f_3}{f_4}$ of $\TSL$, we define the three vertices $[f_1], [f_1,f_2]$ and $\sbmat{f_1}{f_2}{f_3}{f_4}$ as the following sets:
\begin{itemize}
\item $[f_1]:= \C^* \cdot f_1 = \{ af_1 ;\; a \in \C^* \}$; 
\item $[f_1,f_2]:= \GL_2 \cdot (f_1,f_2) = \left\lbrace  (af_1 + bf_2 ,\, cf_1 + df_2) ;\; \mat{a}{b}{c}{d} \in \GL_2 \right\rbrace$;
\item $\bmat{f_1}{f_2}{f_3}{f_4} = \OO_4 \cdot f.$
\end{itemize}
Each bracket $[f_1]$ (resp. $[f_1,f_2]$, resp. $\sbmat{f_1}{f_2}{f_3}{f_4}$) denotes an orbit under the left action of the group $\C^*$ (resp. $\GL_2$, resp. $\OO_4$).
Vertices of the form  $[f_1]$ (resp. $[f_1,f_2]$, resp. $\sbmat{f_1}{f_2}{f_3}{f_4}$) are said to be of {\bf type 1} (resp. {\bf 2}, resp. {\bf 3}). Remark that our notation distinguishes between:
\begin{itemize}
\item $\mat{f_1}{f_2}{f_3}{f_4}$ which denotes an element of $\TSL$;
\item $\bmat{f_1}{f_2}{f_3}{f_4}$ which denotes a vertex of type 3.
\end{itemize}
The set of the \textbf{vertices} of the complex $\Comp$ is the disjoint union of the three types of vertices that we have just defined.

We now define the edges of $\Comp$, which reflect the inclusion of a component inside a row or column, or of a row or column inside an automorphism.
Precisely the set of the \textbf{edges} is the disjoint union of the following two types of edges:
\begin{itemize}
\item Edges that link a vertex $[f_1]$ of type 1 with a vertex $[f_1,f_2]$ of type 2; 
\item Edges that link a vertex $[f_1,f_2]$ of type 2 with a vertex  $\sbmat{f_1}{f_2}{f_3}{f_4}$ of type 3.
\end{itemize}

The set of the \textbf{squares} of $\Comp$  consists in filling the loop of four edges associated with the classes 
$[f_1]$, $[f_1,f_2]$, $[f_1,f_3]$ and $\sbmat{f_1}{f_2}{f_3}{f_4}$ for any $f = \smat{f_1}{f_2}{f_3}{f_4} \in \TSL$ (see Figure \ref{fig:square}).
The square associated with the classes $[x_1]$, $[x_1,x_2]$, $[x_1,x_3]$ and $\sbmat{x_1}{x_2}{x_3}{x_4}$ will be called the \textbf{standard square}. 

Observe that to an automorphism $f = \smat{f_1}{f_2}{f_3}{f_4}$ we can associate (by applying the above definitions to $\sigma \circ f$ with $\sigma$ in the Klein group $\V_4$):
\begin{itemize}
\item Four vertices of type 1: $[f_1]$, $[f_2]$, $[f_3]$ and $[f_4]$;
\item Four vertices of type 2: $[f_1,f_2]$, $[f_1,f_3]$, $[f_2,f_4]$ and $[f_3,f_4]$;
\item One vertex of type 3: $[f]$.
\item Twelve edges and four squares (see Figure \ref{fig:bigsquare}).
\end{itemize}  
We call such a figure the \textbf{big square} associated with $f$.
For any integers $m,n \ge 1$, we call $\mathbf{m \times n}$ \textbf{grid} any subcomplex of $\Comp$ isometric to a rectangle of $\R^2$ of size $m \times n$.
So a big square is a particular type of $2 \times 2$ grid.

We adopt the following convention for the pictures (see for instance Figures \ref{fig:square}, \ref{fig:bigsquare} and \ref{fig:afewsquares}):
Vertices of type 1 are depicted with a $\circ$,
vertices of type 2  are depicted with a $\bullet$,
vertices of type 3 are depicted with a $\blob$. 
\begin{figure}[ht]
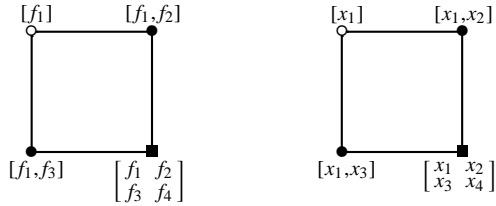

$$
\mygraph{
!{<0cm,0cm>;<1.6cm,0cm>:<0cm,1.6cm>::}
!~-{@{-}@[|(2)]}
!{(-1,1)}*{\circ}="f1"
!{(-0.957,1)}="f1right"
!{(-1,0.965)}="f1down"
!{(0,1)}*-{\bullet}="f1f2"
!{(-1,0)}*-{\bullet}="f1f3"
!{(0,0)}*-{\blob}="f1f2f3f4"
"f1right"-^<{[f_1]}^>{[f_1,f_2]}"f1f2"
"f1f3"-"f1down"
"f1f2f3f4"-"f1f2" "f1f2f3f4"-^<(.05){\sbmat{f_1}{f_2}{f_3}{f_4}}^>(.95){[f_1,f_3]}"f1f3"
}
\qquad \qquad
\mygraph{
!{<0cm,0cm>;<1.6cm,0cm>:<0cm,1.6cm>::}
!~-{@{-}@[|(2)]}
!{(-1,1)}*{\circ}="f1"
!{(-0.957,1)}="f1right"
!{(-1,0.965)}="f1down"
!{(0,1)}*-{\bullet}="f1f2"
!{(-1,0)}*-{\bullet}="f1f3"
!{(0,0)}*-{\blob}="f1f2f3f4"
"f1right"-^<{[x_1]}^>{[x_1,x_2]}"f1f2"
"f1f3"-"f1down"
"f1f2f3f4"-"f1f2" "f1f2f3f4"-^<(.05){\sbmat{x_1}{x_2}{x_3}{x_4}}^>(.95){[x_1,x_3]}"f1f3"
}
$$
\caption{Generic square \& standard square.} \label{fig:square}
\end{figure}

\begin{figure}[ht]
$$\mygraph{
!{<0cm,0cm>;<1.7cm,0cm>:<0cm,1.7cm>::}
!~-{@{-}@[|(2)]}
!{(-1,1)}*{\circ}="f1"
!{(-0.958,1)}="f1right"
!{(-1,0.967)}="f1down"
!{(0,1)}*-{\bullet}="f1f2"
!{(-1,0)}*-{\bullet}="f1f3"
!{(1,1)}*{\circ}="f2"
!{(0.958,1)}="f2left"
!{(1,0.968)}="f2down"
!{(1,0)}*-{\bullet}="f2f4"
!{(1,-1)}*{\circ}="f4"
!{(0.958,-1)}="f4left"
!{(1,-0.947)}="f4up"
!{(0,-1)}*-{\bullet}="f3f4"
!{(-1,-1)}*{\circ}="f3"
!{(-0.962,-1)}="f3right"
!{(-1,-0.947)}="f3up"
!{(0,0)}*-{\blob}="f1f2f3f4"
"f1right"-^<{[f_1]}^>{[f_1,f_2]}"f1f2"-^>{[f_2]}"f2left"
"f2down"-"f2f4"-^<{[f_2,f_4]}"f4up"
"f4left"-^<{[f_4]}"f3f4"-^<{[f_3,f_4]}^>{[f_3]}"f3right"
"f3up"-^>(.95){[f_1,f_3]}"f1f3"
"f1f3"-"f1down"
"f1f2f3f4"-"f1f2" "f1f2f3f4"-"f2f4" "f1f2f3f4"-"f3f4"
 "f1f2f3f4"-^<(.47){[f] = \sbmat{f_1}{f_2}{f_3}{f_4}}"f1f3"
}
\quad 
\mygraph{
!{<0cm,0cm>;<1.7cm,0cm>:<0cm,1.7cm>::}
!~-{@{-}@[|(2)]}
!{(-1,1)}*{\circ}="f1"
!{(-0.958,1)}="f1right"
!{(-1,0.967)}="f1down"
!{(0,1)}*-{\bullet}="f1f2"
!{(-1,0)}*-{\bullet}="f1f3"
!{(1,1)}*{\circ}="f2"
!{(0.958,1)}="f2left"
!{(1,0.968)}="f2down"
!{(1,0)}*-{\bullet}="f2f4"
!{(1,-1)}*{\circ}="f4"
!{(0.958,-1)}="f4left"
!{(1,-0.947)}="f4up"
!{(0,-1)}*-{\bullet}="f3f4"
!{(-1,-1)}*{\circ}="f3"
!{(-0.958,-1)}="f3right"
!{(-1,-0.948)}="f3up"
!{(0,0)}*-{\blob}="f1f2f3f4"
"f1right"-^<{[x_1]}^>{[x_1,x_2]}"f1f2"-^>{[x_2]}"f2left"
"f2down"-"f2f4"-^<{[x_2,x_4]}"f4up"
"f4left"-^<{[x_4]}"f3f4"-^<{[x_3,x_4]}^>{[x_3]}"f3right"
"f3up"-^>(.95){[x_1,x_3]}"f1f3"
"f1f3"-"f1down"
"f1f2f3f4"-"f1f2" "f1f2f3f4"-"f2f4" "f1f2f3f4"-"f3f4"
 "f1f2f3f4"-^<(.17){[\id]}"f1f3"
}$$
\caption{Generic big square \& standard big square.}
\label{fig:bigsquare}
\end{figure}

\begin{figure}[ht]
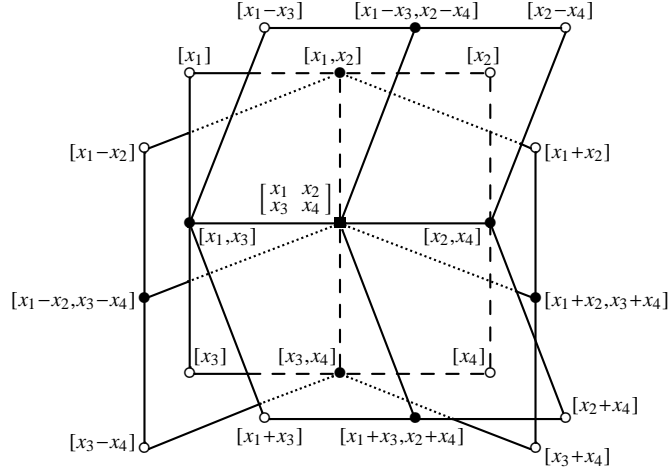

$$\mygraph{
!{<0cm,0cm>;<2cm,0cm>:<0cm,2cm>::}
!~-{@{-}@[|(2)]}
!{(-1,1)}*{\circ}="x1"
!{(-1,.97)}="x1down"
!{(-.968,1)}="x1right"
!{(-.61,1)}="x1fake"
!{(0,1)}*{\bullet}="x1x2"
!{(-1,0)}*-{\bullet}="x1x3"
!{(1,1)}*{\circ}="x2"
!{(1,.97)}="x2down"
!{(1,0)}*-{\bullet}="x2x4"
!{(1,-1)}*{\circ}="x4"
!{(0,-1)}*{\bullet}="x3x4"
!{(-1,-1)}*{\circ}="x3"
!{(-.968,-1)}="x3right"
!{(-.61,-1)}="x3fake"
!{(0,0)}*{\blob}="x1x2x3x4"
!{(-.5,1.3)}*{\circ}="x1-x3"
!{(-.465,1.3)}="x1-x3right"
!{(-.5,1.27)}="x1-x3down"
!{(.5,1.3)}*-{\bullet}="x1-x3x2-x4"
!{(1.5,1.3)}*{\circ}="x2-x4"
!{(1.465,1.3)}="x2-x4left"
!{(1.5,1.27)}="x2-x4down"
!{(-.5,-1.3)}*{\circ}="x1+x3"
!{(-.467,-1.3)}="x1+x3right"
!{(.5,-1.3)}*-{\bullet}="x1+x3x2+x4"
!{(.5,-1.3)}="x1+x3x2+x4up"
!{(1.5,-1.3)}*{\circ}="x2+x4"
!{(1.465,-1.3)}="x2+x4left"
!{(-1.3,.5)}*{\circ}="x1-x2"
!{(-1.3,.47)}="x1-x2down"
!{(-1.269,.5)}="x1-x2right"
!{(-1.01,.6)}="x1-x2fake"
!{(-1.3,-.5)}*-{\bullet}="x1-x2x3-x4"
!{(-1.3,-.5)}="x1-x2x3-x4right"
!{(-1.01,-.387)}="x1-x2x3-x4fake"
!{(-1.3,-1.5)}*{\circ}="x3-x4"
!{(-1.27,-1.5)}="x3-x4right"
!{(-.53,-1.2)}="x3-x4fake"
!{(1.3,.5)}*{\circ}="x1+x2"
!{(1.268,.5)}="x1+x2left"
!{(1.3,.473)}="x1+x2down"
!{(1.2,.53)}="x1+x2fake"
!{(1.3,-.5)}*-{\bullet}="x1+x2x3+x4"
!{(1.3,-.5)}="x1+x2x3+x4left"
!{(1.17,-.45)}="x1+x2x3+x4fake"
!{(1.3,-1.5)}*{\circ}="x3+x4"
!{(1.267,-1.5)}="x3+x4left"
!{(.75,-1.3)}="x3+x4fake"
"x1right"-^<{[x_1]}"x1fake"-@{--}^>{[x_1,x_2]}"x1x2"-@{--}^>{[x_2]}"x2"
"x2down"-@{--}"x2x4"-@{--}_<(.1){[x_2,x_4]}"x4"-@{--}_<(.1){[x_4]}"x3x4"-@{--}_<(0.3){[x_3,x_4]}"x3fake"-_>(.7){[x_3]}"x3right"
"x3"-_>(.9){[x_1,x_3]}"x1x3"
"x1x3"-"x1down"
"x1x2x3x4"-@{--}"x1x2" "x1x2x3x4"-"x2x4" "x1x2x3x4"-@{--}"x3x4"
 "x1x2x3x4"-_<(.27){\sbmat{x_1}{x_2}{x_3}{x_4}}"x1x3"
"x1x3"-"x1-x3down"
"x1-x3right"-^<{[x_1-x_3]}^>{[x_1-x_3,x_2-x_4]}"x1-x3x2-x4"-^>{[x_2-x_4]}"x2-x4left"
"x2-x4down"-"x2x4" "x1x2x3x4"-"x1-x3x2-x4"
"x1x3"-"x1+x3"
"x1+x3right"-_<{[x_1+x_3]}_>(.9){[x_1+x_3,x_2+x_4]}"x1+x3x2+x4"-"x2+x4left"
"x2+x4"-_<{[x_2+x_4]}"x2x4" "x1x2x3x4"-"x1+x3x2+x4up"
"x1x2"-@{.}"x1-x2fake"-"x1-x2right"
"x1-x2down"-_<{[x_1-x_2]}_>{[x_1-x_2,x_3-x_4]}"x1-x2x3-x4"-_>{[x_3-x_4]}"x3-x4"
"x3-x4right"-"x3-x4fake"-@{.}"x3x4" 
"x1x2x3x4"-@{.}"x1-x2x3-x4fake"-"x1-x2x3-x4right"
"x1x2"-@{.}"x1+x2fake"-"x1+x2left"
"x1+x2down"-^<{[x_1+x_2]}^>{[x_1+x_2,x_3+x_4]}"x1+x2x3+x4"-^>(1.1){[x_3+x_4]}"x3+x4"
"x3+x4left"-"x3+x4fake"-@{.}"x3x4" 
"x1x2x3x4"-@{.}"x1+x2x3+x4fake"-"x1+x2x3+x4left"
}$$
\caption{A few other squares...}
\label{fig:afewsquares}
\end{figure}

The group $\TSL$ acts naturally  on the complex $\Comp$. 
The action on the vertices of type 1, of type 2 and of type 3 is given respectively by the following three formulas:
\begin{align*}
g\cdot [f_1] &:= [f_1 \circ g^{-1}];\\
g\cdot [f_1,f_2] &:= [f_1 \circ g^{-1},f_2 \circ g^{-1}];\\
g\cdot [f] &:= [f \circ g^{-1}].
\end{align*}
It is an action by isometries, where $\Comp$ is endowed with the natural metric obtained by identifying each square to an 
euclidean square with edges of length 1.

\subsection{Transitivity and stabilizers}

We show that the action of $\TSL$ is transitive on many natural subsets of $\Comp$, and we also compute some related stabilizers.

\begin{lemma} \label{lem: 2a and 2b}
The action of $\TSL$ is transitive on vertices of type 1, 2 and 3 respectively.
The action of $\STSL$ is transitive on vertices of type 1 and 3 respectively, but admits two distinct orbits of vertices of type 2.

\end{lemma}

\begin{proof}
Let $v_1$ (resp. $v_2$, $v_3$) be a vertex of type 1 (resp. 2, 3). 
There exists $f = \smat{f_1}{f_2}{f_3}{f_4} \in \TSL$ such that $v_1 = [f_1]$ (resp. $v_2 = [f_1,f_2]$, $v_3 = [f]$).
Then $[x_1] = [f_1 \circ f^{-1}] = f\cdot [f_1]$ (resp. $[x_1,x_2] = f\cdot [f_1,f_2]$, $[\id] = f\cdot [f]$).
If $f$ is not in $\STSL$ then $g = \tau \circ f = \smat{f_1}{f_3}{f_2}{f_4}$ is in $\STSL$.
We also have $[x_1] = g\cdot [f_1]$ and $[\id] = [\tau] = g \cdot [f]$, but $g \cdot [f_1,f_2] = [x_1,x_3]$.

It remains to prove that $[x_1,x_3]$ and $[x_1,x_2]$ are not in the same orbit under the action of $\STSL$.
Assume that $g \in \TSL$ sends $[x_1,x_3]$ on $[x_1,x_2]$, and let $h \in \OO_4$ be the linear part of $g$.
We still have $h \cdot [x_1,x_3] = [x_1,x_2]$, and by Lemma \ref{lem:action of O_4 with respect to horizontality} we deduce that $h \in \OO_4 \smallsetminus \SO_4$, hence $g \in \TSL \smallsetminus \STSL$. 
\end{proof}

\begin{definition} \label{def:horizontal and vertical edges}
(1) We say that a vertex of type 2 is \textbf{horizontal} (resp. \textbf{vertical}) if it lies in the same orbit as $[x_1, x_2]$ (resp. $[x_1,x_3]$) under the action of $\STSL$.

(2) We  say that an edge is \textbf{horizontal} (resp. \textbf{vertical}) if it lies in the same orbit as the edges between  $[x_1]$ and $[x_1, x_2]$ or between $[x_1, x_3]$ and $[ \id ]$ (resp. between $[x_1]$ and $[x_1,x_3]$ or between $[x_1, x_2]$ and $[ \id ]$) under the action of $\STSL$.
\end{definition}
 
We will study in \S \ref{sec:stab x1} the structure of the stabilizer $\Stab([x_1])$. In particular we will show that it admits a structure of amalgamated product.

Of course by definition the stabilizer of the vertex $[\id]$ of type 3  is the group $\OO_4$.

\begin{lemma} \label{lem:stab x1x3}
The stabilizer in $\TSL$ of the vertex $[x_1, x_3]$ of type 2  is the semi-direct product $\Stab([x_1, x_3]) =\Er \rtimes \GL_2$, where
\begin{multline*} 
\GL_2= \left\lbrace\mat{ax_1 + bx_3}{a'x_2+b'x_4}{cx_1 + dx_3}{c'x_2+d'x_4}; \;a,b,c,d,a',b',c',d' \in \C, \right.\\
\left. \mat{d'}{-b'}{-c'}{a'}\mat{a}{b}{c}{d}=\mat{1}{0}{0}{1} \right\rbrace.
\end{multline*}
\end{lemma}

\begin{proof}
Let $g =\smat{g_1}{g_2}{g_3}{g_4}\in \Stab([x_1,x_3])$.
We have $[g_1, g_3] = g^{-1} \cdot [x_1,x_3] = [x_1,x_3]$.
Hence $g_1, g_3$ are linear polynomials in $x_1, x_3$ that define an automorphism of $\Vect(x_1,x_3)$, in other words we can view $g_1, g_3$ as an element of $\GL_2$.
By composing $g$ with a linear automorphism of the form 
$$\mat{ax_1 + bx_3}{a'x_2+b'x_4}{cx_1 + dx_3}{c'x_2+d'x_4},$$ 
we can assume $g_1 = x_1$, $g_3 = x_3$.
Then, the result follows from Remark \ref{rem:justification of Er}.
\end{proof}

We now turn to the action of $\TSL$ and $\STSL$ on edges.

\begin{lemma}
The action of $\TSL$ is transitive respectively on edges between vertices of type 1 and 2, and on edges between vertices of type 2 and 3.
The action of $\STSL$ on edges admits four orbits, corresponding to the four edges of the standard square.  
\end{lemma}

\begin{proof}
If there is an edge between $v_1$ a vertex of type 1 and $v_2$ a vertex of type 2, then there exists $f = \smat{f_1}{f_2}{f_3}{f_4} \in \TSL$ such that $v_1 = [f_1]$ and $v_2 = [f_1, f_2]$.
Then $f \cdot v_1 = [x_1]$ and $f \cdot v_2 = [x_1,x_2]$. 

Similarly if there is an edge between $v_3$ a vertex of type 3 and $v_2$ a vertex of type 2, then there exists $f = \smat{f_1}{f_2}{f_3}{f_4} \in \TSL$ such that $v_3 = [f]$ and $v_2 = [f_1, f_2]$.
Then $f \cdot v_3 = [\id]$ and $f \cdot v_2 = [x_1,x_2]$. 

In both cases, if $f \not\in \STSL$, we change $f$ by $g = \tau \circ f$ and we obtain $g\cdot v_1 = [x_1]$, $g \cdot v_2 = [x_1, x_3]$, $g \cdot v_3 = [\id]$.
\end{proof}

\begin{lemma} \label{lem:action on edges}
(1) The stabilizer  of the edge between $[x_1]$ and $[x_1, x_3]$ is the semi-direct product
\begin{align*}
\Er \rtimes \left\lbrace\mat{ax_1}{d^{-1}x_2}{dx_3 +cx_1  }{a^{-1}x_4 + ca^{-1}d^{-1}x_2}; \; a,c,d \in \C, \; ad \neq 0 \right\rbrace.
\end{align*}
(2) The stabilizer of the edge between $[x_1, x_2]$ and $[\id]$ is the following subgroup of $\SO_4$:
\[\left\lbrace  A \cdot \mat{x_1}{x_2}{x_3}{x_4} \cdot B^t; \; A,B \in \SL_2, \; A \text{ is lower triangular}  \right\rbrace.\]
\end{lemma}

\begin{proof}
(1) This follows from  Lemma \ref{lem:stab x1x3}.

(2) Recall that $\Stab ( [ \id])= \OO_4$. By Lemma \ref{lem: 2a and 2b}, we have $\Stab ( [x_1,x_2]) \subseteq \STSL$. Therefore, the stabilizer of  the edge between $[x_1, x_2]$ and $[\id]$ is included into $\SO_4$. By \ref{sec:definitions about O4}, any element of $\SO_4$ is of the form
\[ f=A \cdot \mat{x_1}{x_2}{x_3}{x_4} \cdot B^t, \quad \text{where }A,B \in \SL_2.\]
A direct computation shows that $f$ belongs to $\Stab([x_1, x_2],[\id])$ if and only if $A$ is lower triangular.
\end{proof}

\begin{lemma} \label{lem:action on v3 next to v2}
Let $v_2 = [f_1, f_2]$ be a vertex of type 2, and $\P$ be the path of length 2 through the vertices $[f_1]$, $[f_1, f_2]$, $[f_2]$. Then :
\begin{enumerate}
\item The pointwise stabilizer $\Stab \P$ is isomorphic to
\[\Er \rtimes \left\lbrace\mat{ax_1}{b^{-1}x_2}{bx_3}{a^{-1}x_4}; \; a,b \in \C^* \right\rbrace.\]
\item The group $\Stab \P$ acts transitively on the set of vertices of type 3 at distance 1 from $v_2$.
\item If $[f]$, $[g]$ are two vertices of type 3 at distance 1 from $v_2$, then there exists a generalized elementary automorphism $h$ such that $[g] = [h \circ f]$.
\end{enumerate}
\end{lemma}

\begin{proof}
Without loss in generality we can assume $f_1 = x_1$, $f_2 = x_3$.
Then (1) follows from Lemma \ref{lem:action on edges}.
By definition of the complex, if $v_3$ is at distance $1$ from $v_2 = [x_1, x_3]$, then $v_3 = [e]$ with $e = \smat{x_1}{e_2}{x_3}{e_4} \in \TSL$.
By Remark \ref{rem:justification of Er} we get $e \in \Er$ and (2) follows.
Now if $[f]$, $[g]$ are two vertices of type 3 at distance 1 from $v_2 = [x_1,x_3]$, we can assume $[f] = [\id]$ and $[g] = [e]$ for some $e \in \Er$.  Thus there exist $a,b \in \OO_4$ such that $g = a e$ and $f = b$. 
Then 
\[[g] = [a e] = [be ] = [beb^{-1} f]\]
and $h = beb^{-1}$ is a generalized elementary automorphism.
\end{proof}

\begin{lemma} \label{lem:action on squares}
The group $\TSL$ acts transitively on squares.
The pointwise stabilizer of the standard square is the following  subgroup of $\SO_4$:
\begin{align*}
S &= \left\lbrace   \mat{a}{0}{b}{a^{-1}} \cdot \mat{\rule{0mm}{2mm}x_1}{x_2}{\rule{0mm}{2mm}x_3}{x_4} \cdot \mat{a'}{b'}{0}{a'^{-1}}; \;a,a'\in \C^*, \; b,b' \in \C   \right\rbrace \\
&= \left\{ \mat{ax_1}{b(x_2 + cx_1)}{b^{-1}(x_3 +d x_1)}{\dots}; \; a,b,c,d \in \C, \; ab \neq 0 \right\}.
\end{align*}
\end{lemma}

\begin{proof}
By definition, a square corresponds to vertices $v_1 = [f_1]$, $v_2 = [f_1,f_2]$, $v_3 = [f]$ and $v_2' = [f_1,f_3]$ where $f = \smat{f_1}{f_2}{f_3}{f_4} \in \TSL$. Then $f\cdot v_1 = [x_1]$, $f\cdot v_2 = [x_1,x_2]$, $f \cdot v_3 = [\id]$ and $f \cdot v_2' = [x_1,x_3]$. The computation of the stabilizer of the standard square is left to the reader.
\end{proof}

\begin{remark} \label{rem:square containing id}
The squares containing $[\id]$ are naturally parametrized by $\p^1 \times \p^1$, i.e. by points of the quadric $S$ in Remark~\ref{rem:geometry of isotropic cone}: see Figure~\ref{fig:square corresponding to ([alpha:beta],[gamma:delta])}.
In the same vein, one can remark that the set of vertices of type 1 connected to an arbitrary vertex  $v_2 = [f_1, f_2]$ of type 2 is parametrized by $\p^1$; explicitely they are of the form $[af_1+bf_2]$.
\end{remark}

\begin{figure}[ht]
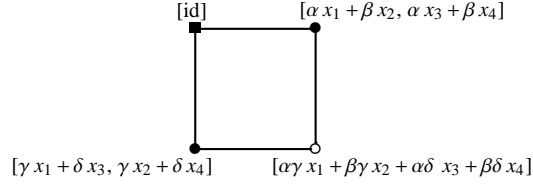

$$
\mygraph{
!{<0cm,0cm>;<1.6cm,0cm>:<0cm,1.6cm>::}
!~-{@{-}@[|(2)]}
!{(-1,1)}*-{\blob}="f1"
!{(0,1)}*-{\bullet}="f1f2"
!{(-1,0)}*-{\bullet}="f1f3"
!{(0,0)}*{\circ}="f1f2f3f4"
!{(-.045,0)}="f1f2f3f4left"
"f1"-^<{[\id]}^>(1.7){[ \alpha \, x_1 \, +\,  \beta \,  x_2, \; \alpha \,  x_3 \, + \,  \beta \,  x_4]}"f1f2"
"f1f3"-"f1"
"f1f2f3f4"-"f1f2" "f1f2f3f4left"-^<(-0.8){[\alpha \gamma \,  x_1 \, + \,  \beta \gamma  \, x_2 \, + \, \alpha \delta \,  \, x_3 \, + \,  \beta \delta  \, x_4]}^>(1.7){[ \gamma \,  x_1 \, + \, \delta \,  x_3, \;  \gamma \, x_2 \, + \, \delta \,  x_4]}"f1f3"
}
$$
\caption{The  square containing $[\id]$ corresponding to \hspace{15cm} $((\alpha:\beta),(\gamma:\delta))\in\p^1\times\p^1$.}
\label{fig:square corresponding to ([alpha:beta],[gamma:delta])}
\end{figure}

We have seen that any element $f$ of $\TSL$ defines a big square centered at $[f]$ (see Figure \ref{fig:bigsquare}). 
We have the following converse result:

\begin{lemma} \label{lem: Any big square is associated with an element of TSL}
Any $2 \times 2$ grid centered at a vertex of type 3 is the big square associated with some element  of $\TSL$. 
\end{lemma}

\begin{proof}
By Lemma \ref{lem:action on squares}, we may reduce to the case where the $2 \times 2$ grid contains the standard square. 
By Remark \ref{rem:square containing id}, there exist elements $(a:b)$ and $(a':b')$ in $\p^1$ such that the grid is as depicted on Figure \ref{fig:A big square containing the standard square}.

\begin{figure}[ht]
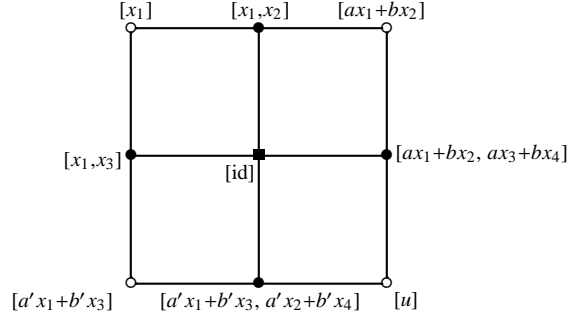

$$
\mygraph{
!{<0cm,0cm>;<1.7cm,0cm>:<0cm,1.7cm>::}
!~-{@{-}@[|(2)]}
!{(-1,1)}*{\circ}="f1"
!{(-.96,1)}="f1right"
!{(-1,.965)}="f1down"
!{(0,1)}*-{\bullet}="f1f2"
!{(-1,0)}*-{\bullet}="f1f3"
!{(1,1)}*{\circ}="f2"
!{(.96,1)}="f2left"
!{(1,.965)}="f2down"
!{(1,0)}*-{\bullet}="f2f4"
!{(1,-1)}*{\circ}="f4"
!{(.96,-1)}="f4left"
!{(0,-1)}*-{\bullet}="f3f4"
!{(-1,-1)}*{\circ}="f3"
!{(-.96,-1)}="f3right"
!{(0,0)}*-{\blob}="f1f2f3f4"
"f1right"-^<{[x_1]}^>{[x_1,x_2]}"f1f2"-^>{[ax_1+bx_2]}"f2left"
"f2down"-"f2f4"-^<{[ax_1+bx_2,\;ax_3+bx_4]}"f4"
"f4left"-^<(-0.2){[u]}"f3f4"-^<{[a'x_1 +b' x_3, \; a'x_2 +b'x_4]}^>(1.6){[a'x_1 +b'x_3]}"f3right"
"f3"-^>(.95){[x_1,x_3]}"f1f3"-"f1down"
"f1f2f3f4"-"f1f2" "f1f2f3f4"-"f2f4" "f1f2f3f4"-"f3f4"
 "f1f2f3f4"-^<(.17){[\id]}"f1f3"
}$$
\caption{A $2 \times 2$ grid containing the standard square.}
\label{fig:A big square containing the standard square}
\end{figure}
Note that $u=a' (ax_1+bx_2)+ b'(ax_3+bx_4) = a (a'x_1+b'x_3)+b(a'x_2 + b'x_4)$. Since the vertices $[ax_1+bx_2]$ and $[a'x_1 +b'x_3]$  are distinct from $[x_1]$, we have $bb' \neq 0$. We may therefore assume that $bb'=1$. If we set $f_1=x_1$, $f_2=ax_1+bx_2$, $f_3= a'x_1+b'x_3$, $f_4=u$, we have $f_1f_4-f_2f_3= bb' (x_1x_4-x_2x_3) = x_1x_4-x_2x_3$, so that $f=\smat{f_1}{f_2}{f_3}{f_4} \in \OO_4$. Finally, our $2 \times 2$ grid is the big square associated with $f$.
\end{proof}

\begin{corollary} \label{lem:transitive on big squares}
The action of $\TSL$ on the set of $2 \times 2$ grids centered at a vertex of type 3 is transitive.
\end{corollary}

\begin{proof}
By Lemma \ref{lem: Any big square is associated with an element of TSL}, any $2 \times 2$ grid centered at a vertex of type 3 is associated with an element $f$ of $\TSL$. Therefore, by applying $f$ to this big square, we obtain the standard big square.
\end{proof}

The following lemma is obvious.

\begin{lemma}
The (point by point) stabilizer of the standard big square is the group 
$$\left\lbrace \mat{ax_1}{bx_2}{b^{-1}x_3}{a^{-1}x_4};\; a,b \in \C^* \right\rbrace.$$
\end{lemma}

\subsection{Isometries} \label{sec:minimizing set}

If $f$ is an isometry of a $\CAT(0)$ space $X$, we define $\Min(f)$ to be the set of points realizing the infimum $\inf d(x,f(x))$.
The set $\Min(f)$ is a closed convex subset of $X$ (see \cite[p. 229]{BH}).
If $X$ is a $\CAT(0)$ cube complex of finite dimension, then for any $f \in \Isom(X)$, the set $\Min(f)$ is non empty (\cite[II.6, 6.6.(2), p. 231]{BH}).

We say that $f$ is \textbf{elliptic} if $\inf d(x,f(x)) = 0$ (there exists a fixed point for $f$), and that $f$ is \textbf{hyperbolic} otherwise. The number $\ell(f) = \inf d(x,f(x))$ is called the \textbf{translation length} of $f$.
Note that in the elliptic case, $\Min(f)$ is the fixed locus of $f$.

In a $\CAT(0)$ space, an isometry is elliptic if and only if one of its orbits is bounded, or equivalently if any of its orbits is bounded (see \cite[Proposition II.6.7]{BH}). Recall also that for any isometry $f$, $\ell (f^k)= |k| \times \ell(f)$ for each integer $k$.

For subgroups, we introduce a similar terminology.
Let $X$ be a $\CAT(0)$ cube complex, and denote by $X(\infty)$ the natural boundary of $X$ (see \cite[Chapter II.8]{BH}).
Let $\Gamma \subseteq \Isom(X)$ be a subgroup of isometries acting without inversion on edges. 
\begin{itemize}
\item $\Gamma$ is \textbf{elliptic} if there exists a vertex $v \in X$ that is fixed by all elements in $\Gamma$;
\item $\Gamma$ is \textbf{parabolic} if all elements of $\Gamma$ are elliptic, there is no global fixed vertex in $X$ and there is a fixed point in $X(\infty)$;
\item $\Gamma$ is \textbf{loxodromic} if $\Gamma$ contains at least one hyperbolic isometry and there is a fixed pair of points in $X(\infty)$.
\end{itemize}
We will also use the following less standard terminology: We say that an isometry $f$ is \textbf{hyperelliptic} if $f$ is elliptic with $\Min(f)$ unbounded.   
Here is a simple criterion to produce hyperelliptic elements.

\begin{lemma}  \label{lem: criterion of hyperellipticity}
Any elliptic isometry of a $\CAT (0)$ space commuting with a hyperbolic isometry is hyperelliptic.
\end{lemma}

\begin{proof}
Assume that $f$ is such an elliptic isometry commuting with an hyperbolic isometry $g$. By \cite[II.6.2]{BH}, the set $\Min(f)$ is globally invariant by $g$. Since $g$ is hyperbolic, this set is unbounded. 
\end{proof}

The following criterion is useful in identifying hyperbolic isometries.

\begin{lemma}\label{lem:hyperbolic_elements}
Let $X$ be a $\CAT(0)$ space, $x \in X$ a point, and $g \in \Isom(X)$. Then $x \in \Min(g)$ if and only if $g(x)$ is the middle point of $x$ and $g^2(x)$.
\end{lemma}

\begin{proof}
If $x \in \Min(g)$, it is clear that $g(x)$ is the middle point of $x$ and $g^2(x)$. 
Conversely, assume that $g(x)$ is the middle point of $x$ and $g^2(x)$. 
We may assume furthermore that $x$ is different from $g(x)$. The orbit of the segment $[x,g(x)]$ forms a geodesic invariant under $g$, on which $g$ acts by translation. Then, one can apply \cite[II.6.2(4)]{BH}.
\end{proof}

\subsection{First properties}

Section \ref{sec:O4} on the orthogonal group yields some basic facts on the square complex:
 
\begin{lemma}\label{lem:uniquesquare}
Assume $v$ and $v'$ are opposite vertices of a same square in $\Comp$. 
Then the square containing $v$ and $v'$ is unique.   
\end{lemma}

\begin{proof}
There are two cases to consider (up to exchanging $v$ and $v'$):
\begin{enumerate}
\item $v$ is of type 1 and $v'$ is of type 3;
\item $v$ and $v'$ are both of type 2.
\end{enumerate}

In Case (1), we can assume $v' = \sbmat{x_1}{x_2}{x_3}{x_4}$.
Then $v = [f_1]$ with $f_1 \in V^*$ an isotropic vector, and by Witt's Theorem we can assume $f_1 = x_1$.
We conclude by Lemma \ref{lem:isotropicsecantplanes} that the unique square containing $v$ and $v'$ is the standard square. 

In Case (2), let $v''$ a vertex of type 3 that is at distance 1 from $v$ and $v'$. We can assume that $v'' =  \sbmat{x_1}{x_2}{x_3}{x_4}$. 
Then, there exist linear forms $l_1$, $l_2$, $l'_1$, $l'_2$  in $V^*$ such that $v=[l_1,l_2]$ and $v'=[ l'_1,l'_2]$. In particular $v''$ is the unique vertex of type 3 that is at distance 1 from $v$ and $v'$. 
Then $v$ and $v'$ correspond to two totally isotropic planes in $V^*$, with a 1-dimensional intersection. 
Let $f_1 \in V^*$ be a generator for this line. By Witt's Theorem we can assume $f_1 = x_1$, and the standard square is the unique square containing both $v$ and $v'$.
\end{proof}

\begin{corollary} \label{cor:intersectionsquares}
The standard square (hence any square) is embedded in the complex $\Comp$, and the intersection of two distinct squares is either:
\begin{enumerate}
\item empty;
\item a single vertex;
\item a single edge (with its two vertices).
\end{enumerate}
\end{corollary}

\begin{proof}
The first assertion is just the obvious remark that $[x_1,x_2] \neq [x_1, x_3]$, hence the corresponding vertices are distinct in $\Comp$.

Assume that two squares have an intersection different from the three stated cases. 
Then the intersection contains two opposite vertices of a square, hence the two squares are the same by Lemma \ref{lem:uniquesquare}.
\end{proof}

\subsection{Tame\texorpdfstring{$\mathbf{(\A_K^n)}$}{An} acting on a simplicial complex}

Let $K$ be a field.
In this section we construct a simplicial complex on which the group of tame automorphisms of $\A_K^n$ acts.
Our motivation here is twofold.
On the one hand we shall need the definition for $n = 2$, $K = \C(x)$ in the study of link of vertices of type 1 in $\Comp$.
On the other hand the construction for $n = 3$, $K = \C$ is very similar in nature to the construction of $\Comp$, and gives rise to interesting questions about the tame group of $\C^3$ (see Section \ref{sec:tameKn}).

\subsubsection{A general construction} \label{sec:general simplicial complex}

For any $1 \le r \le n$, we call \textbf{$r$-tuple of components} a map 
\begin{align*}
K^n &\to K^r \\
x = (x_1, \dots, x_n) &\mapsto \left( f_1(x), \dots, f_r(x) \right)
\end{align*}
that can be extended as a tame automorphism $f = (f_1,\dots, f_n)$ of $\A_K^n$.
One defines $n$ distinct types of vertices, by considering $r$-tuple of components modulo composition by an affine automorphism on the range, $r = 1, \dots, n$:
\begin{equation*}
[f_1, \dots,f_r] := A_r (f_1, \dots, f_r) = \{ a \circ (f_1, \dots, f_r) ; a \in A_r\}
\end{equation*}
where $A_r = \GL_r(K) \ltimes K^r$ is the $r$-dimensional affine group. 

Now for any tame automorphism $(f_1,\dots, f_n)$ we glue a $(n-1)$-simplex on the vertices  $[f_1]$, $[f_1,f_2]$, ..., $[f_1, \dots,f_n]$. 
This definition is independent of a choice of representatives and produces a $(n-1)$-dimensional simplicial complex on which the tame group acts by isometries.

\subsubsection{Dimension 2} \label{sec:serretree}

Let $K$ be a field.
The previous construction yields a graph $\T_K$.
In this section we show that $\T_K$ is isomorphic to the classical  Bass-Serre tree of $\Aut(\A_K^2)$.
We use the affine groups:
\begin{align*}
A_1 &= \{t \mapsto at + b ;\; a \in K^*, b \in K \};\\
A_2 &= \left\lbrace (t_1,t_2) \mapsto (a t_1 + b t_2 + c, a' t_1 + b' t_2 + c') ;\; \mat{a}{b}{a'}{b'} \in \GL_2, c,c' \in K \right\rbrace.
\end{align*}

The vertices of our graph  $\T_K$ are of two types: classes $A_1 f_1$ where $f_1 \colon K^2 \to K$ is a component of an automorphism, and classes $A_2 (f_1,f_2)$ where $(f_1,f_2) \in \Aut(\A^2_K)$. 
For each automorphism $(f_1,f_2) \in \Aut(\A^2_K)$, we attach an edge between $A_1 f_1$ and $A_2 (f_1,f_2)$.
Note that $A_2 (f_1,f_2)  = A_2 (f_2,f_1)$, so there is also an edge between the vertices $A_2 (f_1,f_2)$ and $A_1 f_2$.

Recall that $\Aut(\A^2_K)$ is the amalgamated product of $A_2$ and $E_2$ along their intersection, where $E_2$ is the elementary group defined as:
$$E_2 = \{ (x,y) \mapsto (ax + P(y), by + c);\; a,b \in K^*, c \in K\}.$$
The Bass-Serre tree associated with this structure consists in taking cosets $A_2 (f_1,f_2)$, $E_2 (f_1,f_2)$ as vertices, and cosets $(A_2 \cap E_2) (f_1,f_2)$ as edges (we use right cosets for consistency with the convention for $\T_K$, the classical construction with left cosets is similar). 

\begin{proposition}
The graph $\T_K$ is isomorphic to the Bass-Serre tree associated with the structure of amalgamated product of  $\Aut(\A^2_K)$. 
\end{proposition}

\begin{proof}
We define a map $\phi$ from the set of vertices of the Bass-Serre tree to the graph $\T_K$ by taking
\begin{align*}
A_2(f_1,f_2) &\mapsto A_2(f_1,f_2),\\
E_2(f_1,f_2) &\mapsto A_1 f_2.
\end{align*}
Clearly $\phi$ is a local isometry.
Moreover $\phi$ is bijective, since we can define $\phi^{-1}(A_1 f_2)$ to be $E_2(f_1,f_2)$ where $(f_1,f_2)$ is an automorphism.
Indeed any other way to extend $f_2$ is of the form $(af_1+P(f_2),f_2)$, and so the class $E_2(f_1,f_2)$ does not depend on the extension we choose. 
\end{proof}

\begin{remark}
If two vertices $A_1 f_1$ and $A_1 f_2$ are at distance 2 in $\T_K$, then $(f_1,f_2) \in \Aut(\A^2_K)$. 
Indeed, by transitivity of the action we may assume that the central vertex is $ A_2 (x,y)$. 
Then for $i = 1,2$ we can write $f_i=a_ix+b_iy+c_i$.
Observe that $(f_1,f_2)$ is  invertible if and only if $\det \smat{a_1}{b_1}{a_2}{b_2}\neq 0$. This is equivalent to $A_1f_1 \neq A_1 f_2$.
\end{remark}
 
\section{Geometry of the complex} \label{sec:complexes_bis}

In this section we establish Theorem \ref{thm:mainHyp}, which says that the complex $\Comp$ is $\CAT(0)$ and hyperbolic.
First we study the local curvature of the complex by studying the links of its vertices.

\subsection{Links of vertices}

Let $v$ be a vertex (of any type) in  $\Comp$. 
The link around $v$ is denoted by $\LL(v)$.
By definition this is the graph whose vertices are the vertices in $\Comp$ at distance exactly 1 from $v$, and endowed with the standard angular metric:  $v_1$ and $v_2$ are joined by an edge of length $\pi/2$ if they are opposite vertices of a same square, which necessarily has $v$ as a vertex (see \cite[\S I.7.15, p. 103]{BH} for details).

A \textbf{path} $\P$ in $\LL(v)$ is a simplicial map $[0,n\pi/2] \to \LL(v)$ which is locally injective (``no backtrack'').
We call $n$ the \textbf{length} of $\P$, and we denote $\P = v_0, \dots, v_n$ where $v_k$ is the vertex image of $k\pi/2$.
We say that $\P$ is a \textbf{loop} if $v_0 = v_n$.
By a slight abuse of notation we will often identify $\P$ with its image in $\LL(v)$.

\begin{remark} \label{rem:2loop}
Note that any loop in $\LL(v)$ has length at least 3.
Indeed a loop $v_0,v_1,v_0$ of length 2 in $\LL(v)$ should correspond to two distinct squares sharing $v, v_0$ and $v_1$ as vertices. This would contradict Corollary \ref{cor:intersectionsquares}. Similarly there is no self-loop in $\LL(v)$.
     
\end{remark}

\subsubsection{Vertex of type 1} \label{sec:link type 1}

We study the link of a vertex of type 1, and show that its geometry is closely related to the geometry of a simplicial tree.

Recall that in \S\ref{sec:serretree} we constructed a tree $\T_K$ on which $\Aut(\A_K^2)$ acts. 
We use this construction in the case $K = \C(f_1)$, where $f_1$ is a component.
Without loss in generality we can assume $f_1 = x_1$.
We note $\LL(x_1)$ instead of $\LL([x_1])$.

\begin{lemma}  \label{lem:link1connected}
The graph $\LL(x_1)$ is  connected.
\end{lemma}

\begin{proof}
Any vertex of  $\LL(x_1)$ is of the form $v = [x_1, f_2]$, where  $f = \smat{x_1}{f_2}{f_3}{f_4} \in \TSL$. Note that the vertices $[x_1,f_2]$ and $[x_1,f_3]$ are joined by one edge in $\LL(x_1)$. By Corollary \ref{cor:linktype1}, $f$ can be written as a composition of elements which are either equal to the transpose $\tau$ or which are of the form 
$$\mat{x_1}{ax_2 + x_1P(x_1,x_3)}{a^{-1}x_3}{\dots}.$$ 
Since we have
\[\tau \mat{x_1}{ax_2 + x_1P(x_1,x_3)}{a^{-1}x_3}{\dots} \tau = \mat{x_1}{a^{-1}x_2 }{ax_3 + x_1P(x_1,x_2)}{\dots},\]
it follows that $f$ or $\tau f$ is a composition of automorphisms of the form
\[\mat{x_1}{ax_2 + x_1P(x_1,x_3)}{a^{-1}x_3}{\dots}\quad \text{or} \quad \mat{x_1}{a^{-1}x_2 }{ax_3 + x_1P(x_1,x_2)}{\dots}.\]
This gives a path in $\LL(x_1)$ from $v$ to either $[x_1,x_2]$ or $[x_1,x_3]$.
\end{proof}

Recall that vertices of type 2 are called horizontal or vertical depending if they lie in the orbit of  $[x_1, x_2]$ or  $[x_1,x_3]$ under the action of $\STSL$.

\begin{lemma} \label{lem:link1loop_are_even}
Any loop in $\LL(x_1)$ has even length.
\end{lemma}

\begin{proof}
This follows from the simple remark that the vertices of the loop must be alternatively horizontal and vertical.
\end{proof}

Let $ \LL(x_1)'$ be the first barycentric subdivision of $\LL(x_1)$, that is, the graph obtained from $\LL(x_1)$ by subdividing each edge $v,v'$ in two edges $v,v''$ and $v'',v'$, where $v''$ is the middle point of $v,v'$. 
If $\smat{ax_1}{f_2}{f_3}{\dots} \in \TSL$, it is natural to identify $\sbmat{ax_1}{f_2}{f_3}{\dots}$ with the vertex of $\LL(x_1)'$ that is the middle point in $\LL(x_1)$ of the edge between $[x_1, f_2]$ and $[x_1, f_3]$. Indeed, recall from \cite[\S I.7.15, p. 103]{BH} that the link $\LL( x_1)$ may be seen as the set of directions at $x_1$. 


Now we define a simplicial map 
\[\pi \colon \LL(x_1)' \to \T_{\C(x_1)}.\]
First we send each vertex $[x_1,f_2] \in \LL (x_1)'$ to the vertex $A_1 f_2 \in \T_{\C(x_1)}$.
This makes sense because of Corollary \ref{cor:linktype1}: $f_2$ is a component of a polynomial automorphism in $x_2, x_3$ with coefficients in $\C(x_1)$.
Second we send each vertex  $\sbmat{ax_1}{f_2}{f_3}{f_4} \in \LL (x_1)'$ to the vertex $A_2 (f_2,f_3) \in \T_{\C(x_1)}$.
Observe that since we start from the barycentric subdivision $\LL(x_1)'$ we obtain a map $\pi$ which is simplicial: If $\smat{ax_1}{f_2}{f_3}{\dots} \in \Stab([x_1])$, then $A_1 f_2$ and $A_1 f_3$ are at distance 2 in the image of $\pi$.


\begin{lemma} \label{lem:fundamentaldomain}
\begin{enumerate}
\item The action of $\Stab([x_1])$ on $\LL(x_1)'$ admits the edge between $[x_1, x_2]$ and $[id]$ as a fundamental domain. 
In particular, the action is transitive on vertices of type 2 of $\LL(x_1)'$.
\item  Let $v,v'$ be two vertices of $\LL(x_1)$ and let $h$ be an element of  $\Stab([x_1])$. Then, the equality $\pi(v) = \pi(v')$ implies the equality $\pi(h(v)) = \pi (h(v'))$.
\end{enumerate}
\end{lemma}
 
\begin{proof}
(1)  This is again a direct consequence of Corollary~\ref{cor:linktype1}.

(2) We can assume $v = [x_1, x_2]$, and so $v' = [x_1, x_2 + x_1P(x_1)]$ for some polynomial $P \in \C[x_1]$.
We can write 
$$h^{-1} = \mat{ax_1}{x_2}{x_3}{a^{-1} x_4}\mat{x_1}{f_2}{f_3}{f_4}$$ 
where $(f_2,f_3) \in \Aut (\A^2_{\C(x_1)})$.
Then 
$h(v) = [ax_1, f_2]$ and $h(v') = [ax_1, f_2 + ax_1P(ax_1)]$, so 
\[\pi(h(v)) = A_1 f_2 = A_1 ( f_2 + ax_1P(ax_1)) = \pi(h(v')).  \qedhere\]
\end{proof}

Point (2) of the last lemma means that the natural action of $\Stab([x_1])$ on $\LL(x_1)$ induces an action on  $\pi ( \LL(x_1) )$ such that $\pi \colon  \LL(x_1) \to \pi( \LL(x_1))$ is equivariant.

\begin{lemma} \label{lem:link1toatree}
\begin{enumerate}
\item The set $\pi( \LL(x_1))$ is a subtree of $\T_{\C(x_1)}$.
\item Let $w= A_1 f_2$ and $w' = A_1 f_3$ be two vertices at distance 2 in the image of $\pi$. Then the preimage by $\pi$ of the segment between $w$ and $w'$ is a complete bipartite graph between $\pi^{-1}(w)$ and $\pi^{-1}(w')$.
\end{enumerate}
\end{lemma}

\begin{proof}
(1)  This follows from the fact that $\LL(x_1)$ is connected (see Lemma \ref{lem:link1connected}), and the fact that  $\pi \colon \LL(x_1)' \to \T_{\C(x_1)}$ is a simplicial map. 

(2) By transitivity of the action of $\TSL$ on squares we can assume  $f_2 = x_2$ and $f_3 = x_3$.
Then any vertex in $\pi^{-1}(w)$ has the form $v = [x_1, x_2+x_1P(x_1)]$. Similarly any vertex in $\pi^{-1}(w')$ has the form $v' = [x_1, x_3+x_1Q(x_1)]$. 
But then for any choices of $P,Q$ we remark that 
$$g = \mat{x_1}{x_2+x_1P(x_1)}{x_3+x_1Q(x_1)}{x_4+x_3P(x_1)+x_2Q(x_1)+x_1P(x_1)Q(x_1)}$$ 
is a tame automorphism, hence $v, v'$ are linked by an edge in $\LL(x_1)$, with midpoint $[g]$.
\end{proof}

\subsubsection{Vertex of type 2 or 3}

The link of a vertex of type 1 projects surjectively to an unbounded tree (the fact that $\pi( \LL(x_1))$ is unbounded follows from Proposition \ref{pro:first amalgamated structure of Stabx1} below and its proof), in particular this is an unbounded graph.
This is completely different for the link of a vertex of type 2 or 3: We show that both are complete bipartite graphs.

\begin{proposition} \label{pro:linktype2}
Let $v_2$ be a vertex of type 2. 
Then any vertex of type 1 in $\LL(v_2)$ is linked to any vertex of type 3 in $\LL(v_2)$. 
In other words $\LL(v_2)$ is a complete bipartite graph.
\end{proposition}

\begin{proof}
Let $v_1$ (resp. $v_3$) be a vertex of type 1 (resp. 3) in $\LL(v_2)$.
By transitivity on edges, we can assume that $v_2 = [x_1,x_2]$ and $v_3 = \sbmat{x_1}{x_2}{x_3}{x_4}$.
Then if $v_1 = [f_1]$, we complete $f_1$ in a basis $(f_1,f_2)$ of $\Vect(x_1,x_2)$. 
By Lemma \ref{lem:uniqueauto}, there exists a unique basis $(f_3,f_4)$ of $\Vect(x_3,x_4)$ such that $f=\smat{f_1}{f_2}{f_3}{f_4}$ belongs to $\OO_4$. 
It is then clear that $v_1$ and $v_3$ are linked in $\LL(v_2)$: $v_1,v_2,v_3$ belong to a same square, as illustrated in Figure \ref{fig:v1v2v3}.
\end{proof}

\begin{figure}[ht]
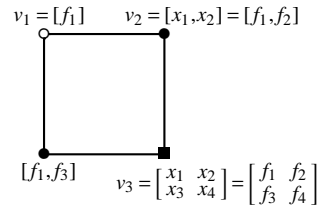

$$
\mygraph{
!{<0cm,0cm>;<1.6cm,0cm>:<0cm,1.6cm>::}
!~-{@{-}@[|(2)]}
!{(-1,1)}*{\circ}="f1"
!{(-.955,1)}="f1right"
!{(-1,.965)}="f1down"
!{(0,1)}*-{\bullet}="f1f2"
!{(-1,0)}*-{\bullet}="f1f3"
!{(0,0)}*-{\blob}="f1f2f3f4"
"f1right"-^<{v_1 \,=\, [f_1]}^>(1.4){v_2 \,=\, [x_1,x_2] \,=\, [f_1,f_2]}"f1f2"
"f1f3"-"f1down"
"f1f2f3f4"-"f1f2" "f1f2f3f4"-^<(-.4){v_3 \,=\, \sbmat{x_1}{x_2}{x_3}{x_4} \,=\, \sbmat{f_1}{f_2}{f_3}{f_4}}^>(.95){[f_1,f_3] }"f1f3"
}
$$
\caption{The square containing $v_1,v_2,v_3$.} \label{fig:square illustrating the link of a vertex of type 2}\label{fig:v1v2v3}
\end{figure}

\begin{proposition} \label{pro:linktype3}
Let $v_3$ be a vertex of type 3, and let $v_2, v_2' \in \LL(v_3)$ be two distinct vertices (necessarily of type 2).
Then  $d(v_2,v_2') = \pi/2$ or $\pi$ in $\LL(v_3)$, and precisely:
\begin{itemize}
\item either $v_2, v_2'$ belong to a same square (which is unique);
\item or for any $v''_2$ in $\LL(v_3)$ such that $d(v_2,v''_2) = \pi/2$ in $\LL(v_3)$, then $v_2, v''_2, v_2'$ is a path in $\LL(v_3)$. 
\end{itemize}
In particular $\LL(v_3)$ is a complete bipartite graph.
\end{proposition}

\begin{proof}
Without loss in generality we can assume $v_3 = \sbmat{x_1}{x_2}{x_3}{x_4}$.
Then $v_2$ and $v_2'$ correspond to totally isotropic planes $W, W'$ in $V^*$, and by Remark \ref{rem:geometry of isotropic cone} they correspond to lines in a smooth quadric surface in $\p^3$.

There are two possibilities:
\begin{enumerate}[(i)]
\item The two lines intersect in one point, meaning that the corresponding totally isotropic planes intersect along a one dimensional space $\Vect ( f_1 )$, and then by Lemma \ref{lem:isotropicsecantplanes} we can write $v_2 = [f_1, f_2]$, $v_2' = [f_1,f_3]$ with $\smat{f_1}{f_2}{f_3}{\dots} \in \OO_4$.
\item The two lines belong to the same ruling, and taking a third line in the other ruling, which corresponds to a vertex $v_2'' \in \LL(x_1)$, we can apply twice the previous observation: first to $v_2, v_2''$, and then to $v_2', v_2''$. \qedhere
\end{enumerate}  
\end{proof}

In the second case of the proposition, the vertices $v_2, v_2', v_3$ are part of a unique ``big square'' (see Figure \ref{fig:bigsquare}): This follows from Lemma \ref{lem:uniqueauto}.

\subsubsection{Negative curvature}

As a consequence of our study of links we obtain:

\begin{proposition}\label{pro:localcurvature}
Let $v \in \Comp$ be a vertex. Then any (locally injective) loop in the link $\LL(v)$ has length at least 4. 
In particular the square complex $\Comp$ has non positive local curvature.
\end{proposition}

\begin{proof}
By Remark \ref{rem:2loop} we know that any loop has length at least 3. So we only have to exclude loops of length 3.
Clearly such a loop cannot exist in the link of a vertex of type 2 or 3, since by Propositions \ref{pro:linktype2} and \ref{pro:linktype3} these are complete bipartite graphs: Any loop in $\LL(v)$ has even length for such a vertex.
This leaves the case of a vertex of type 1, and this was covered by Lemma \ref{lem:link1loop_are_even}.

For the last assertion see \cite[II.5.20 and II.5.24]{BH}.
\end{proof}

\subsection{Simple connectedness}

\begin{proposition}\label{pro:1connected}
The complex $\Comp$ is simply connected.
\end{proposition}  

\begin{figure}[ht]
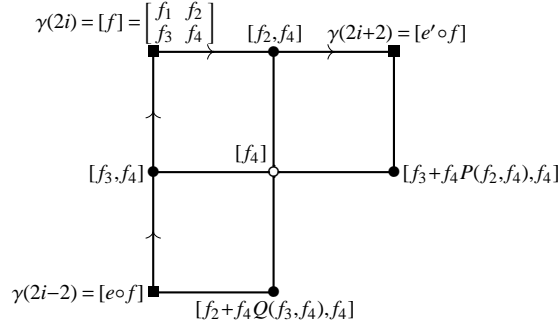

$$\mygraph{
!{<0cm,0cm>;<1.6cm,0cm>:<0cm,1.6cm>::}
!~-{@{-}@[|(2)]}
!{(-1,1)}*-{\blob}="f"
!{(0,1)}*-{\bullet}="f2f4"
!{(1,1)}*-{\blob}="e'f"
!{(-1,0)}*-{\bullet}="f3f4"
!{(0,0)}*{\circ}="f4"
!{(.043,0)}="f4right"
!{(-.045,0)}="f4left"
!{(0,-.04)}="f4down"
!{(1,0)}*-{\bullet}="f4e'f"
!{(-1,-1)}*-{\blob}="ef"
!{(0,-1)}*-{\bullet}="f4ef"
"f"-|@{>}^<(-0.2){\gamma(2i) \,=\, [f] \,=\, \sbmat{f_1}{f_2}{f_3}{f_4}}^>{[f_2,f_4]}"f2f4"-|@{>}^>{\gamma(2i+2) \,=\, [e' \circ f]}"e'f"-^>{[f_3+f_4P(f_2,f_4), f_4]}"f4e'f"-"f4right"
"f"-|@{<}_>{[f_3,f_4]}"f3f4"-^>(.85){[f_4]}"f4left"
"f4"-"f2f4"
"f3f4"-|@{<}_>{\gamma(2i-2) \,=\, [e \circ f]}"ef"-_>{[f_2+f_4Q(f_3,f_4), f_4]}"f4ef"-"f4down"
}$$
\caption{Initial situation around the maximal vertex $[f]$.}\label{fig:starting position}
\end{figure}

\begin{proof}
Let $\gamma$ be a loop in $\Comp$. We want to show that it is homotopic to a trivial loop.
Without loss in generality, we can assume that the image of $\gamma$ is contained in the 1-skeleton of the square complex, that $\gamma$ is locally injective, and that $\gamma(0) = \sbmat{x_1}{x_2}{x_3}{x_4}$ is the vertex of type 3 associated with the identity.

A priori (the image of) $\gamma$ is a sequence of arbitrary edges.
By Lemma \ref{lem:link1connected}, we can perform a homotopy to avoid each vertex of type 1.
So now we assume that vertices in $\gamma$ are alternatively of type 2 and 3: Precisely for each $i$, $\gamma(2i)$ has type 3 and $\gamma(2i+1)$ has type 2.

For each vertex $v = [f]$ of type 3 of $\Comp$, we define $\deg v:= \deg f$. 
This definition is not ambiguous, since by Lemma \ref{lem:degmax} we know that $\deg v$ does not depend on the choice of representative $f$. 
Let $i$ be the greatest integer such that $\deg \gamma (2i) = \max_{j} \deg \gamma (2j)$. 
In particular, we have 
$$\deg \gamma(2i + 2) < \deg \gamma(2i) \quad \text{ and } \quad \deg \gamma(2i - 2) \le \deg \gamma(2i).$$ 
Let $f =\smat{f_1}{f_2}{f_3}{f_4}$ be such that $\gamma(2i) = [f]$.

By Lemma \ref{lem:action on v3 next to v2} there exist generalized elementary automorphisms $e, e'$  such that $\gamma(2i-2) = [e \circ f]$ and $\gamma(2i+2) = [e' \circ f]$. 
Observe that for any element $a \in \OO_4$ we have $[f] = [a \circ f]$, $[e \circ  f] = [a \circ e \circ a^{-1} \circ a \circ f]$ and $[e' \circ f] = [a \circ e' \circ a^{-1} \circ a \circ f]$. 
In consequence, by Corollary \ref{cor:pair of elementary automorphisms} we can assume that  
$$e' = \mat{x_1 + x_2P(x_2,x_4)}{x_2}{x_3+x_4P(x_2,x_4)}{x_4}$$
and $e$ is of one of the three forms given in the corollary.

Observe that 
$$e = \mat{x_1+x_2Q(x_2,x_4)}{x_2}{x_3+x_4Q(x_2,x_4)}{x_4}$$
would contradict that the loop is locally injective, since the vertex of type 2 just after and just before $[f]$ would be $[f_2,f_4]$.
The case 
$$e = \mat{x_1}{x_2+x_1Q(x_1,x_3)}{x_3}{x_4+x_3Q(x_1,x_3)}$$
is also impossible: Since $P$ is not constant, by Lemma \ref{lem:both degree drop} we would get $\deg f_1 > \deg f_2$, $\deg f_3 > \deg f_4$ and finally $\deg e \circ f > \deg f$, a contradiction.
So we are left with the third possibility 
$$ e = \mat{x_1+x_3Q(x_3,x_4)}{x_2+x_4Q(x_3,x_4)}{x_3}{x_4}.$$
In particular the vertices of type 2 before and after $\gamma(2i)$ belong to a same square, as shown on Figure \ref{fig:starting position}; and we are in the setting of Lemma \ref{lem:12inLV}.

\begin{figure}[ht]
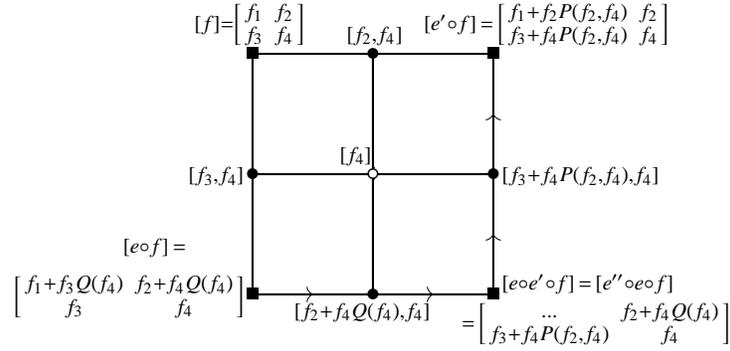

$$\mygraph{
!{<0cm,0cm>;<1.6cm,0cm>:<0cm,1.6cm>::}
!~-{@{-}@[|(2)]}
!{(-1,1)}*-{\blob}="f"
!{(0,1)}*-{\bullet}="f2f4"
!{(1,1)}*-{\blob}="e'f"
!{(-1,0)}*-{\bullet}="f3f4"
!{(0,0)}*{\circ}="f4"
!{(.04,0)}="f4right"
!{(-.05,0)}="f4left"
!{(0,-.04)}="f4down"
!{(1,0)}*-{\bullet}="f4e'f"
!{(-1,-1)}*-{\blob}="ef"
!{(0,-1)}*-{\bullet}="f4ef"
!{(1,-1)}*-{\blob}="tf"
"f"-^<{[f] = \sbmat{f_1}{f_2}{f_3}{f_4}}^>{[f_2,f_4]}"f2f4"-^>(1.4){[e' \circ f] \,=\, \sbmat{f_1+f_2P(f_2,f_4)}{f_2}{f_3+f_4P(f_2,f_4)}{f_4}}"e'f"-|@{<}^>{[f_3+f_4P(f_2,f_4), f_4]}"f4e'f"-"f4right"
"f"-_>{[f_3,f_4]}"f3f4"-^>(.9){[f_4]}"f4left"
"f4"-"f2f4"
"f3f4"-_>(0.6){[e \circ f] \,= \qquad}_>{ \sbmat{f_1+f_3Q(f_4)}{f_2+f_4Q(f_4)}{f_3}{f_4}}"ef"-|@{>}_>(0.9){[f_2+f_4Q(f_4), f_4]}"f4ef"-"f4down"
"f4ef"-|@{>}_>(1.8){ =\, \sbmat{\dots}{f_2+f_4Q(f_4)}{f_3+f_4P(f_2,f_4)}{f_4} }                                                                                                                             "tf"-_<(.1){[e \circ e' \circ f] \,=\, [e'' \circ e \circ f]}|@{>}"f4e'f"
}$$
\caption{Local homotopy in Case (1): $Q \in \C[f_4]$.}\label{fig:case1}
\end{figure}

\begin{figure}[ht]
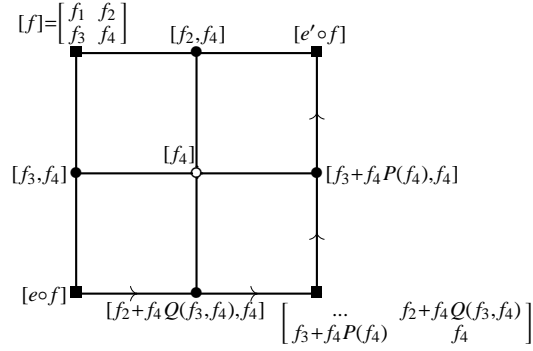

$$\mygraph{
!{<0cm,0cm>;<1.6cm,0cm>:<0cm,1.6cm>::}
!~-{@{-}@[|(2)]}
!{(-1,1)}*-{\blob}="f"
!{(0,1)}*-{\bullet}="f2f4"
!{(1,1)}*-{\blob}="e'f"
!{(-1,0)}*-{\bullet}="f3f4"
!{(0,0)}*{\circ}="f4"
!{(.04,0)}="f4right"
!{(-.046,0)}="f4left"
!{(0,-.038)}="f4down"
!{(1,0)}*-{\bullet}="f4e'f"
!{(-1,-1)}*-{\blob}="ef"
!{(0,-1)}*-{\bullet}="f4ef"
!{(1,-1)}*-{\blob}="tf"
"f"-^<{[f] = \sbmat{f_1}{f_2}{f_3}{f_4}}^>{[f_2,f_4]}"f2f4"-^>{[e' \circ f]}"e'f"-|@{<}^>{[f_3+f_4P(f_4), f_4]}"f4e'f"-"f4right"
"f"-_>{[f_3,f_4]}"f3f4"-^>(.9){[f_4]}"f4left"
"f4"-"f2f4"
"f3f4"-_>{[e \circ f]}"ef"-|@{>}_>(.9){[f_2+f_4Q(f_3,f_4), f_4]}"f4ef"-"f4down"
"f4ef"-|@{>}_>(1.7){\sbmat{\dots}{f_2+f_4Q(f_3,f_4)}{f_3+f_4P(f_4)}{f_4}}"tf"-|@{>}"f4e'f"
}$$
\caption{Local homotopy in Case (2): $P \in \C[f_4]$.}\label{fig:case2}
\end{figure}

\begin{figure}[ht]
$$
\mygraph{
!{<0cm,0cm>;<1.7cm,0cm>:<0cm,1.7cm>::}
!~-{@{-}@[|(2)]}
!{(-1,1)}*-{\blob}="f"
!{(0,1)}*-{\bullet}="f2f4"
!{(1,1)}*-{\blob}="e'f"
!{(-1,0)}*-{\bullet}="f3f4"
!{(0,0)}*{\circ}="f4"
!{(.04,0)}="f4right"
!{(-.045,0)}="f4left"
!{(0,-.04)}="f4down"
!{(0,-.59)}="f4downfake"
!{(.025,-.025)}="f4downright"
!{(1,0)}*-{\bullet}="f4e'f"
!{(-1,-1)}*-{\blob}="ef"
!{(0,-1)}*-{\bullet}="f4ef"
!{(.2,-.7)}*-{\blob}="v1"
!{(1.2,-.7)}*-{\bullet}="v2"
!{(1.2,-.7)}="v2up"
!{(2.2,-.7)}*-{\blob}="v3"
"f"-^<{[f]}^>{[f_2,f_4]}"f2f4"-^>(1.8){[e' \circ f] \,=\, \sbmat{f_1+f_2P(f_2,f_4)}{f_2}{f_3+f_4P(f_2,f_4)}{f_4}}"e'f"-|@{<}^>(0.9){[f_3+f_4P(f_2,f_4), f_4]}"f4e'f"-"f4right"
"f"-_>(.9){[f_3,f_4]}"f3f4"-^>(.9){[f_4]}"f4left" "f4"-"f2f4"
"f3f4"-|@{<}_>(.6){[e \circ f] \,= \qquad}_>(.9){ \sbmat{f_1+f_3Q(f_3,f_4)}{f_2+f_4Q(f_3,f_4)}{f_3}{f_4}}"ef"-_>{[f_2+f_4Q(f_3,f_4), f_4]}"f4ef"-"f4downfake"-@{.}"f4down"
"f4ef"
"f3f4"-|@{>}"v1"-|@{>}^<(0.2){[\tilde e \circ f]}_>(.8){[f_2+f_4R(f_4), f_4]}"v2"-|@{>}_>(1.3){ \sbmat{\dots}{f_2+f_4R(f_4)}{f_3+f_4P(f_2,f_4)}{f_4}}"v3"-_<{[\tilde e \circ e' \circ f] \,=}|@{>}"f4e'f" "v2up"-"f4downright"
}$$
\caption{Local homotopy in Case (3).}\label{fig:case3}
\end{figure}

\begin{figure}[ht]
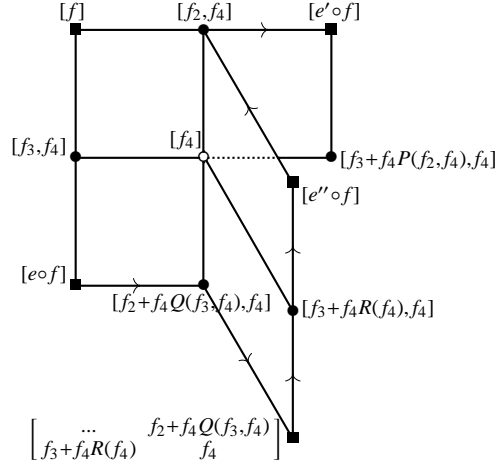

$$\mygraph{
!{<0cm,0cm>;<1.7cm,0cm>:<0cm,1.7cm>::}
!~-{@{-}@[|(2)]}
!{(-1,1)}*-{\blob}="f"
!{(0,1)}*-{\bullet}="f2f4"
!{(1,1)}*-{\blob}="e'f"
!{(-1,0)}*-{\bullet}="f3f4"
!{(0,0)}*{\circ}="f4"
!{(.04,0)}="f4right"
!{(-.041,0)}="f4left"
!{(0,-.035)}="f4down"
!{(0.585,0)}="f4rightfake"
!{(.025,-.025)}="f4downright"
!{(1,0)}*-{\bullet}="f4e'f"
!{(-1,-1)}*-{\blob}="ef"
!{(0,-1)}*-{\bullet}="f4ef"
!{(.7,-.2)}*-{\blob}="v1"
!{(.7,-.2)}="v1up"
!{(.7,-1.2)}*-{\bullet}="v2"
!{(.7,-2.2)}*-{\blob}="v3"
!{(.7,-2.2)}="v3up"
"f"-^<{[f]}^>{[f_2,f_4]}"f2f4"-|@{>}^>{[e' \circ f]}"e'f"-^>{[f_3+f_4P(f_2,f_4), f_4]}"f4e'f"-"f4rightfake"-@{.}"f4right"
"f"-_>(.9){[f_3,f_4]}"f3f4"-^>(.9){[f_4]}"f4left" "f4"-"f2f4"
"f3f4"-_>(.9){[e \circ f]}"ef"-|@{>}_>(0.9){[f_2+f_4Q(f_3,f_4), f_4]}"f4ef"-"f4down"
"f4ef"-|@{>}"v3up"-|@{>}^<{\sbmat{\dots}{f_2+f_4Q(f_3,f_4)}{f_3+f_4R(f_4)}{f_4}\,}_>{[f_3+f_4R(f_4), f_4]}"v2"-|@{>}_>(.9){[e'' \circ f]}"v1up"-|@{>}"f2f4" "v2"-"f4downright"
}$$
\caption{Local homotopy in Case (4).}\label{fig:case4}
\end{figure}

In each one of the four cases of Lemma \ref{lem:12inLV}, we now explain how to perform a local homotopy around $[f_4]$ such that the path avoids the vertex of maximal degree $\gamma(2i)$.

Consider first Case (1), that is to say $Q \in \C[x_4]$ (see Figure \ref{fig:case1}).
Then 
$$e \circ e' = \mat{\dots}{x_2 + x_4Q(x_4)}{x_3+x_4P(x_2,x_4)}{x_4}.$$
Remark that $e \circ e' = e'' \circ e$, where 
$$e'' = \mat{\dots}{x_2}{x_3 + x_4P(x_2 - x_4Q(x_4), x_4)}{x_4}$$ 
is elementary.
Thus we can make a local homotopy in a $2 \times 2$ grid around $[f_4]$ such that the new path goes through $[e \circ e' \circ f]$ instead of $[f]$.
Since $\deg (f_2 + f_4Q(f_4)) \le \deg f_2$, we have $\deg e\circ e'\circ f \le \deg e' \circ f$. 
Recall also that $\deg e' \circ f < \deg f$.
So we get
$$\deg [e \circ e' \circ f] \le \deg [e' \circ f] < \deg [f].$$	

Case (2) is analogous to Case (1) (see Figure \ref{fig:case2}).

Consider Case (3): see Figure \ref{fig:case3}. 
There exists $R(x_4) \in \C [x_4]$ such that $\deg (f_2 + f_4 R(f_4)) < \deg f_2$. 
Set 
$$\tilde e=\mat{x_1+x_3 R(x_4)}{x_2+x_4 R(x_4)}{x_3}{x_4}.$$
We have:
$$\tilde e\circ f =\mat{f_1+f_3 R(f_4)}{f_2+f_4 R(f_4)}{f_3}{f_4}.$$
By Lemma \ref{lem:degree of each component drops},  the inequality $\deg (f_2 + f_4 R(f_4)) < \deg f_2$ is equivalent to any of the following ones:  $\deg (f_1+f_3 R(f_4)) < \deg f_1$ and $\deg  \tilde e \circ f < \deg f$. So we get
$$\deg [\tilde e \circ f] < \deg [f].$$
We conclude by applying Case (1) to the path from $[\tilde e \circ f]$ to $[e' \circ f]$ passing through $[f]$.

Case (4) is analogous to Case (3) (see Figure \ref{fig:case4}).

The result follows by double induction on the maximal degree and on the number of vertices realizing this maximal degree. 
\end{proof}

We obtain the first part of Theorem \ref{thm:mainHyp}:
\begin{corollary}
$\Comp$ is a $\CAT(0)$ square complex.
\end{corollary}

\begin{proof}
Using Propositions \ref{pro:localcurvature} and  \ref{pro:1connected}, this is a consequence of the Cartan-Hada\-mard Theorem: see \cite[Theorem 5.4(4), p. 206]{BH}.
\end{proof}

As a side remark, we can now show  that the action of $\TSL$ on the square complex $\Comp$ is faithful. In fact, we have the following more precise result:

\begin{lemma} \label{lem:faithfulness of the action on the complex}
The action of $\TSL$ on the set of vertices of type 1 (resp. 2, resp. 3) of $\Comp$ is faithful.
\end{lemma}

\begin{proof}
If $g \in \TSL$ acts trivially on vertices of type 3, then by unicity of the middle point of a segment in a $\CAT(0)$ space,  it also acts trivially on vertices of type 2.

Similarly, if $g \in \TSL$ acts trivially on vertices of type 2, then it also acts trivially on vertices of type 1 (which are realized as middle point of vertices of type 2).

So it is sufficient to consider the case of $g \in \TSL$ acting trivially on vertices of type 1. 
Since $[x_1],[x_2]$ and $[x_1 +x_2]$ are such vertices, $g$ must act by homothety on the corresponding lines $\Vect (x_1), \Vect(x_2)$ and $\Vect (x_1+x_2)$.
This implies that $g$ acts by homothety on the plane $\Vect(x_1,x_2)$. 
Similarly, $g$ acts by homothety on $\Vect(x_2,x_3)$ and $\Vect(x_3,x_4)$. 
Therefore, there exists a nonzero complex number $\lambda$ such that $g= \lambda \smat{x_1}{x_2}{x_3}{x_4}$. 
Finally, since $[x_1+x_2^2]$ is a vertex of type 1, $g$ acts by homothety on the line $\Vect(x_1 + x_2^2)$. 
We get $\lambda =1$ and $g=\id$.
\end{proof}

\subsection{Hyperbolicity}

We investigate whether the complex $\Comp$ contains large $n \times n$ grid, that is, large isometrically embedded euclidean squares.
We start with the following result, that shows that $4 \times 4$ grids do exist but are rather constrained.

\begin{lemma} \label{lem:4x4}
If $N, S, E, W$ are polynomials in one variable, then we can construct a $4 \times 4$ grid in $\Comp$ as depicted on Figure \ref{fig:4x4}.
Moreover, up to the action of  $\TSL$, any $4 \times 4$ grid in $\Comp$ centered on a vertex of type 3 is of this form.  
\end{lemma}

\begin{figure}[ht]
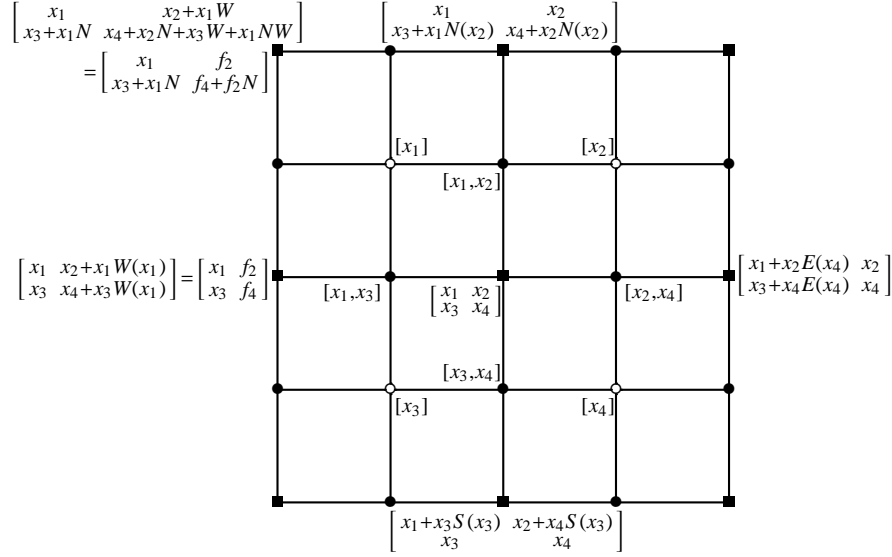

$$\mygraph{
!{<0cm,0cm>;<1.5cm,0cm>:<0cm,1.5cm>::}
!~-{@{-}@[|(2)]}
!{(-1,1)}*{\circ}="x1"
!{(-1.05,1)}="x1left"
!{(-0.955,1)}="x1right"
!{(-1,.962)}="x1down"
!{(0,1)}*-{\bullet}="x1x2"
!{(-1,0)}*-{\bullet}="x1x3"
!{(1,1)}*{\circ}="x2"
!{(.955,1)}="x2left"
!{(1.046,1)}="x2right"
!{(1,.962)}="x2down"
!{(1,0)}*-{\bullet}="x2x4"
!{(1,-1)}*{\circ}="x4"
!{(.955,-1)}="x4left"
!{(1.046,-1)}="x4right"
!{(1,-1.04)}="x4down"
!{(0,-1)}*-{\bullet}="x3x4"
!{(-1,-1)}*{\circ}="x3"
!{(-1.05,-1)}="x3left"
!{(-.955,-1)}="x3right"
!{(-1,-1.04)}="x3down"
!{(0,0)}*-{\blob}="x1x2x3x4"
!{(0,2)}*-{\blob}="N"
!{(0,-2)}*-{\blob}="S"
!{(2,0)}*-{\blob}="E"
!{(-2,0)}*-{\blob}="W"
!{(2,2)}*-{\blob}="NE"
!{(2,-2)}*-{\blob}="SE"
!{(-2,2)}*-{\blob}="NW"
!{(-2,-2)}*-{\blob}="SW"
!{(-1,2)}*-{\bullet}="NNW"
!{(1,2)}*-{\bullet}="NNE"
!{(2,1)}*-{\bullet}="NEE"
!{(2,-1)}*-{\bullet}="SEE"
!{(1,-2)}*-{\bullet}="SSE"
!{(-1,-2)}*-{\bullet}="SSW"
!{(-2,-1)}*-{\bullet}="SWW"
!{(-2,1)}*-{\bullet}="NWW"
"x1right"-^<(.15){[x_1]}_>(.7){[x_1,x_2]}"x1x2"-^>(.88){[x_2]}"x2left"
"x2down"-"x2x4"-^<(.15){[x_2,x_4]}"x4"
"x4left"-^<(.12){[x_4]}"x3x4"-_<(.3){[x_3,x_4]}^>(.85){[x_3]}"x3right"
"x3"-^>(.85){[x_1,x_3]}"x1x3"-"x1down"
"x1x2x3x4"-"x1x2" "x1x2x3x4"-"x2x4" "x1x2x3x4"-"x3x4"
"x1x2x3x4"-^<(.35){\sbmat{x_1}{x_2}{x_3}{x_4}}"x1x3"
"NW"-^<(-1){\sbmat{x_1}{x_2 + x_1W}{x_3+x_1N}{x_4+x_2N+x_3W+x_1NW}}"NNW"-"N"-^<{\sbmat{x_1}{x_2}{x_3+x_1N(x_2)}{x_4+x_2N(x_2)}}"NNE"-"NE"-"NEE"-"E"-^<{\sbmat{x_1+x_2E(x_4)}{x_2}{x_3+x_4E(x_4)}{x_4}}"SEE"-"SE"-"SSE"-"S"-^<{\sbmat{x_1+x_3S(x_3)}{x_2+x_4S(x_3)}{x_3}{x_4}}"SSW"-"SW"-"SWW"-"W"-^<{\sbmat{x_1}{x_2+x_1W(x_1)}{x_3}{x_4+x_3W(x_1)}  \,=\, \sbmat{x_1}{f_2}{x_3}{f_4}}"NWW"-^>(.8){=\, \sbmat{x_1}{f_2}{x_3+x_1N}{f_4+f_2N}}"NW"
"NWW"-"x1left" "x1"-"NNW"
"NEE"-"x2right" "x2"-"NNE"
"SWW"-"x3left" "x3down"-"SSW"
"SEE"-"x4right" "x4down"-"SSE"
"N"-"x1x2" "E"-"x2x4" "S"-"x3x4" "W"-"x1x3"
}$$
\caption{$4 \times 4$ grid associated with polynomials $N, S, W$ and $E$.}
\label{fig:4x4}
\end{figure}

\begin{proof}
Consider a $4 \times 4$ grid centered on a vertex of type 3.
By Corollary \ref{lem:transitive on big squares}, we may assume that the $2 \times 2$  subgrid with same center is the standard big square (Figure \ref{fig:bigsquare}). 
By Lemma \ref{lem:action on v3 next to v2} the upper central vertex of type 3 is of the form $[f]$, where 
$$f = \mat{x_1}{x_2}{x_3+x_1N(x_1,x_2)}{x_4+x_2N(x_1,x_2)} \in \Eb$$
for some polynomial $N$ - for North - in $\C[x_1,x_2]$. Similarly there exist elementary automorphisms of other types associated with polynomials $S,E,W$, which \textit{a priori} are  polynomials in $2$ variables, as depicted on Figure \ref{fig:4x4}.
But now the upper left square in Figure \ref{fig:4x4} exists if and only if
$$\mat{x_1}{x_2 + x_1W(x_1,x_3)}{x_3+x_1N(x_1,x_2)}{\dots} $$
is an automorphism. In particular, the Jacobian determinant must be equal to $1$, so that $\frac{\partial W}{\partial x_3} \frac{\partial N}{\partial x_2} = 0$, i.e. $W$ or $N$ is in $\C[x_1]$.
Up to exchanging $x_2$ and $x_3$ (that is, up to conjugating by the transpose automorphism), we can assume $W \in \C[x_1]$.
Then by using the same argument in the three other corners we obtain $S \in \C[x_3]$, $E\in \C[x_4]$ and $N \in \C[x_2]$. 
\end{proof}

Now we show that arbitrary large grids do not exist. In particular flat disks embedded in $\Comp$ are uniformly bounded. 

\begin{proposition} \label{pro:nogrid}
The complex $\Comp$ does not contain any $6 \times 6$ grid centered on a vertex of type 1.
\end{proposition}

\begin{proof}
Suppose now that we have such a $6 \times 6$ grid.
By Lemma \ref{lem:4x4} we can assume that the lower right $4 \times 4$ subgrid has the form given on Figure \ref{fig:4x4}. 
Then we would have a lower left $4 \times 4$ subgrid centered on the vertex $\sbmat{x_1}{f_2}{x_3}{f_4}$, where we denote $f_2 = x_2+x_1W(x_1)$ and $f_4 = x_4+x_3W(x_1)$.
With the same notation, the center of the upper left $4 \times 4$ subgrid can be rewritten as $\sbmat{x_1}{f_2}{x_3+x_1N}{f_4 + f_2N}$. 
Then again by Lemma \ref{lem:4x4} we should have $N \in \C[x_1]$ or $N \in \C[f_2]$, in contradiction with $N \in \C[x_2]$.
\end{proof}

We obtain the last part of Theorem \ref{thm:mainHyp}:

\begin{corollary}
The complex $\Comp$ is hyperbolic.
\end{corollary}

\begin{proof}
Since the embedding of the 1-skeleton of $\Comp$ into $\Comp$ is a quasi-isometry, it is sufficient to prove that the 1-skeleton is hyperbolic (see \cite[Theorem III.H.1.9]{BH}).
Consider $x,y$ two vertices, and define the interval $\llbracket x,y \rrbracket$ to be the union of all edge-path geodesics from $x$ to $y$.
Then $\llbracket x,y \rrbracket$ embeds as a subcomplex of $\Z^2$ (\cite[Theorem 3.5]{AOS}).
Since there is no large flat grid in the complex $\Comp$, it follows that the 1-skeleton of $\Comp$ satisfies the ``thin bigon criterion'' for hyperbolicity of graphs (see \cite[page 111]{Wi}, \cite{Papa}).  
\end{proof}

\section{Amalgamated product structures}
\label{sec:structures}

There are several trees involved in the geometry of the complex $\Comp$.
We have already encountered in \S \ref{sec:link type 1} the tree associated with the link of a vertex of type 1. 
We will see shortly in \S \ref{sec:product of trees} that there are trees associated with hyperplanes in the complex, and also with the connected components of the complements of two families of hyperplanes.
At the algebraic level this translates into amalgamated product structures for several subgroups of $\TSL$: see Figure \ref{fig:nesting} for a diagrammatic summary of the products studied in this section. 

\subsection{Stabilizer of \texorpdfstring{$\mathbf{[x_1]}$}{x1}} \label{sec:stab x1}

In this section we study in details the structure of $\Stab([x_1])$.
First we show that it admits a structure of amalgamated product.
Then we describe the two factors of the amalgam: the group $H_1$ in Proposition \ref{pro:description of H_1} and $H_2$ in Proposition \ref{pro:description of H_2}. 
We will show in Lemma~\ref{lem:amalgamH_1} that $H_1$ is itself the amalgamated product  of two of its subgroups $K_1$ and $K_2$ (see Definition~\ref{def:K_1 and K_2}) along their intersection.
It turns out that $H_1 \cap H_2= K_2$. 
Therefore, the amalgamated structure of  $\Stab([x_1])$ given in Proposition~\ref{pro:first amalgamated structure of Stabx1} can be ``simplified by $K_2$''. This is Lemma~\ref{pro:second amalgamated structure of Stabx1}.

\subsubsection{A first product} \label{sec:H1 H2}

Recall from \S \ref{sec:link type 1} that there is a map $\pi$ from $\LL(x_1)'$ to a simplicial tree.
In this context it is natural to introduce the following two subgroups of $\Stab([x_1])$: 
\begin{itemize}
\item The stabilizer $H_1$ of the fiber of $\pi$ containing $[\id]$.
\item The stabilizer $H_2$ of the fiber of $\pi$ containing $[x_1, x_3]$.
\end{itemize}

\begin{proposition} \label{pro:first amalgamated structure of Stabx1}
The group $\Stab([x_1])$ is the amalgamated product of $H_1$ and $H_2$ along their intersection:
\[\Stab([x_1]) = H_1 *_{H_1 \cap H_2}H_2.\]
\end{proposition}

\begin{proof}
Consider the action of $\Stab([x_1])$ on the image of $\pi$, which is a connected tree by Lemma \ref{lem:link1toatree}. 
By Lemma \ref{lem:fundamentaldomain}, the fundamental domain is the edge $A_2 (x_2,x_3)$, $A_1 x_3$.
By a classical result (e.g. \cite[I.4.1, Th. 6, p. 48]{S2}) $\Stab([x_1])$ is the amalgamated product of the stabilizers of these two vertices along their intersection: This is precisely our definition of $H_1$ and $H_2$.  
\end{proof}

\subsubsection{Structure of \texorpdfstring{$H_1$}{H1}} \label{sec:H1}

If $R$ is a commutative ring, we put
$$B(R) = \mat{*}{*}{0}{*} \cap \GL_2(R) =\left\{ \mat{a}{b}{0}{d}; \, a,d \in R^*, \, b \in R \right\}.$$
For example $B( \C [x_1]) = \left\{ \mat{a}{b}{0}{d}; \, a,d \in \C^*, \, b \in \C [x_1] \right\}.$

We also introduce the following three subgroups of $\GL_2 (\C[x_1])$:
\begin{itemize}
\item The group $M_1$  of matrices $\mat{b}{0 }{0}{\rule{2mm}{0mm} b^{-1}}$ and $\mat{0 \rule{2mm}{0mm} }{b}{  b^{-1}}{0}$,  $b \in \C^*$;
\item The group $M_2$ of matrices $\mat{b}{x_1 P(x_1)}{0}{b^{-1}}$, $b \in \C^*$, $P \in \C[x_1]$;
\item The group $M$ generated by $M_1$ and $M_2$.
\end{itemize}

The following result is classical (see \cite[Theorem 6, p. 118]{S2}).

\begin{theorem}[Nagao]
The group $\GL_2 ( \C[x_1])$ is the amalgamated product of the subgroups $\GL_2 (\C)$ and $B( \C [x_1]) $ along their intersection $B(\C)$:
$$\GL_2 ( \C[x_1]) = \GL_2 (\C) *_{B(\C)} B( \C [x_1]).$$
\end{theorem}

Since 
$$M_i \cap B (\C) = \left\{ \mat{b}{0}{0}{b^{-1}}; \, b \in \C^* \right\}$$
is independent of $i$, the following result is a consequence of \cite[Proposition 3, p. 14]{S2}:

\begin{corollary} \label{corollary of Nagao}
The group $M$ is the amalgamated product of $M_1$ and $M_2$ along their intersection:
$$M=M_1*_{M_1 \cap M_2}M_2.$$
\end{corollary}

\begin{remark}
We did not find any simpler definition of $M$. Let $\ev \colon \GL _ 2 ( \C [x_1] ) \to \GL _ 2 ( \C)$ be the evaluation at the origin. One can check that $M$ is strictly included into $\ev ^{-1} (M_1)$.
\end{remark}

\begin{proposition} \label{pro:description of H_1}
The group $H_1$ is the set of automorphisms $f= \smat{f_1}{f_2}{f_3}{\dots}$ such that there exist $a \in \C^*$,
$A \in M$  and $P_1, P_2 \in \C [x_1]$ satisfying: 
$$f_1 =a x_1, \quad \colvec{f_2}{f_3} = A \colvec{\rule{0mm}{2mm}x_2}{\rule{0mm}{2mm} x_3} + \colvec{x_1P_1(x_1)}{x_1P_2(x_1)}.$$
In particular $H_1$ is generated by the matrices
$$\mat{ax_1}{bx_2+x_1P(x_1) \, x_3 + x_1Q(x_1)}{b^{-1}x_3}{\dots},a,b \in \C^*, \, P,Q \in \C[x_1] \hspace{2mm} {\rm and} \hspace{2mm} \tau =  \mat{x_1}{x_3}{x_2}{x_4}.$$
\end{proposition}

\begin{proof}
By definition, $H_1$ is the set of elements $f= \smat{f_1}{f_2}{f_3}{f_4}$ of $\Stab([x_1])$ such that $(f_2,f_3)$ induces an affine automorphism of $\A^2_{\C(x_1)}$.
By Corollary~\ref{cor:linktype1}, $(f_2,f_3)$ defines an automorphism of $\A^2_{\C[x_1]}$. 
The linear part of this automorphism corresponds to the matrix $A$. 
The form of the translation part comes from the fact that any element of $\TSL$ is the restriction of an automorphism of $\C^4$ fixing the origin.

Conversely, we must check that any element $f= \smat{f_1}{f_2}{f_3}{f_4}$ of the given form defines an element of $\TSL$. This follows from the definition of $M$.
\end{proof}

The following lemma gives a condition under which the amalgamated structure of a group $G= G_1 *_AG_2$  is extendable to a semi-direct product $G \rtimes_{\varphi} H$.

\begin{lemma}\label{extension of an amalgamated product to a semidirect product}
Let $G=G_1 *_AG_2$ be an amalgamated product, where $G_1,G_2$ and $A$ are subgroups of $G$ such that $A= G_1 \cap G_2$.
Assume that $\phi \colon H \to \Aut \, G$ is an action of a group $H$ on $G$, which globally preserves the subgroups $G_1,G_2$ and $A$, then we have:
$$G \rtimes H= (G_1 \rtimes H) *_{A \rtimes H}(G_2 \rtimes H).$$
\end{lemma}

\begin{proof}
We may assume that all the groups involved in the statement are subgroups of the group $K:= G \rtimes H$ and that $H$ acts on $G$ by conjugation, i.e. $\forall \, h \in H, \, \forall \, g \in G, \, \varphi (h) (g) = hgh^{-1}$. 
Set $K_1=G_1H$, $K_2= G_2H$ and $B=AH$ (since $G_1,G_2$ and $A$ are normalized by $H$, the sets $K_1,K_2$ and $B$ are subgroups of $K$).

We want to prove that $K= K_1 *_BK_2$.

For this, we must first check that $K$ is generated by $K_1$ and $K_2$. This is obvious.

Secondly, we must check that if $w=u_1u_2 \ldots u_r$ is a word  such that $u_1,\ldots,u_r$ belong alternatively to $K_1 \setminus K_2$ and $K_2 \setminus K_1$, then $w \neq 1$.

Assume by contradiction that $w=1$. Write $u_i = g_ih_i$, where $g_i$ belongs to $G_1 \cup G_2$ and $h_i$ belongs to $H$. 
Set $g'_1=g_1$, $g'_i=(h_1 \ldots h_{i-1}) g_i (h_1 \ldots h_{i-1})^{-1}$ for $2 \leq i \leq r$ and $h = h_1\ldots h_r$, then we have
\[w=(g'_1\ldots g'_r) \,  h.\] 
Since $g'_1\ldots g'_r = h^{-1} \in G \cap H =\{ 1\}$, we get $g'_1\ldots g'_r=1$. 
We have obtained a contradiction. 
Indeed $w':= g'_1\ldots g'_r$ is a reduced expression of $G_1 *_AG_2$ (meaning that the $g'_i$ alternatively belong to $G_1 \setminus G_2$ and $G_2 \setminus G_1$), so that we cannot have $w'=1$.
\end{proof}

\begin{definition}\label{def:K_1 and K_2}
We introduce  the following two subgroups of $H_1$:
\begin{itemize}
\item The group $K_1$ of automorphisms of the form
$$\mat{ax_1}{bx_2 + x_1 P(x_1)}{b^{-1}x_3 + x_1 Q(x_1)}{\dots} \; {\rm or} \; \mat{ax_1}{bx_3 + x_1 P(x_1)}{b^{-1}x_2 + x_1 Q(x_1)}{\dots}$$
where $a,b \in \C^*, \;P,Q \in \C[x_1];$
\item The group $K_2$  of automorphisms of the form
$$\mat{ax_1}{bx_2 + x_1 P(x_1)x_3 + x_1Q(x_1)}{b^{-1}x_3 + x_1 R(x_1)}{\dots}$$
where $a,b \in \C^*, \;P,Q,R \in \C[x_1].$
\end{itemize}
\end{definition}

The intersection $K_1 \cap K_2$ is a subgroup of index $2$ in $K_1$, and precisely 
$$K_1= (K_1 \cap K_2) \cup \tau (K_1 \cap K_2).$$

\begin{lemma}\label{lem:amalgamH_1}
The group $H_1$ is the amalgamated product of $K_1$ and $K_2$ along their intersection:
$$H_1= K_1 *_{K_1 \cap K_2}K_2.$$
\end{lemma}

\begin{proof}
Since $H_1$ is the semi-direct product of 
$$G:= \left\lbrace h= \mat{h_1}{h_2}{h_3}{h_4} \in H_1, \; h_1=x_1 \right\rbrace \;\text{ and }\; H:= \left\lbrace \mat{ax_1}{x_2}{x_3}{a^{-1}x_4}, \, a \in \C^*\right\rbrace,$$
it is enough, by Lemma~\ref{extension of an amalgamated product to a semidirect product}, to show that $G$ is the amalgamated product of $G_1:= K_1 \cap G$ and $G_2:= K_2 \cap G$ along their intersection.

Now consider the normal subgroup of  $G$, whose elements are the ``translations'':
$$T:= \left\{ \mat{x_1}{x_2+x_1P(x_1)}{x_3+x_1Q(x_1)}{\dots}, \, P,Q \in \C[x_1] \right\}.$$
Note that $G_1$ and $G_2$ both contain $T$. It is enough to show that $G/T$ is the amalgamated product of $G_1/T$ and $G_2/T$ along their intersection.

This follows from Corollary~\ref{corollary of Nagao}.
Indeed, the natural isomorphism from $G/T$  to $M$ sends $G_i/T$ to $M_i$.
\end{proof}

\subsubsection{Structure of \texorpdfstring{$H_2$}{H2}} \label{sec:H2}

\begin{proposition}\label{pro:description of H_2}
The group $H_2$ is the set of automorphisms of the form
$$\mat{ax_1}{bx_2 + x_1 P(x_1,x_3)}{b^{-1}x_3 + x_1 Q(x_1)}{\dots}, a,b \in \C^*, \;P \in \C [x_1,x_3], \; Q \in \C[x_1].$$
\end{proposition}

\begin{proof}
The proof is analogous to the one of Proposition~\ref{pro:description of H_1}. 
The element $f= \smat{f_1}{f_2}{f_3}{f_4}$ of $\Stab([x_1])$ belongs to $H_2$ if and only if $(f_2,f_3)$ induces a triangular automorphism of $\A^2_{\C(x_1)}$. 
This implies the existence of $a \in \C^*$, $\alpha,\gamma,\delta \in \C[x_1]$ and $\beta \in \C[x_1,x_3]$ such that
$$f_1=ax_1, \; f_2= \alpha x_2 + \beta, \; f_3= \gamma x_3 + \delta.$$
Since $(f_2,f_3)$ defines an automorphism of $\A^2_{\C[x_1]}$, its Jacobian determinant $\alpha \gamma$ is a nonzero complex number. 
This shows that $\alpha$ and $\gamma$ are nonzero complex numbers. Replacing $x_1$ by $0$ in the equation $f_1f_4-f_2f_3 = x_1x_4-x_2x_3$, we get:
$$(\alpha x_2 + \beta (0,x_3))(\gamma x_3 + \delta (0))=x_2x_3.$$
Therefore, there exists $b \in \C^*$ such that  $\alpha =b$, $\gamma =b^{-1}$ and we  have $\beta(0,x_3) = \delta (0) = 0$. The result follows.
\end{proof}

As a direct corollary from Propositions \ref{pro:description of H_1} and \ref{pro:description of H_2} we get:

\begin{corollary} \label{cor:K2}
The group $K_2$ is equal to the intersection $H_1 \cap H_2$.
In particular $K_1 \cap K_2 = K_1 \cap H_2$.
\end{corollary}

\subsubsection{A simplified product}

Finally we get the following alternative amalgamated structure for $\Stab([x_1])$:

\begin{proposition} \label{pro:second amalgamated structure of Stabx1}
The group $\Stab([x_1])$ is the amalgamated product of $K_1$ and $H_2$ along their intersection:
\[\Stab([x_1]) = K_1 * _{K_1 \cap H_2}H_2.\]
\end{proposition}

\begin{proof}
By Proposition \ref{pro:first amalgamated structure of Stabx1}, Lemma \ref{lem:amalgamH_1} and Corollary \ref{cor:K2}, the groups $K_1$ and $H_2$ clearly generate $\Stab([x_1])$. 
To obtain the amalgamated product structure it is enough (using conjugation) to show that any word
\begin{equation}
w = a_1 b_1 \dots a_r b_r,
\label{initial expression of w}
\end{equation}
with $a_i \in K_1 \smallsetminus H_2$ and $b_i \in H_2 \smallsetminus K_1$ is not the identity. 
Set $I:= \{ i \in \llbracket 1 , r \rrbracket; \; b_i \in H_1 \}$. Write $I$ (which may be empty) as the disjoint union of intervals
\[I = \llbracket i_1,j_1 \rrbracket \cup \ldots \cup \llbracket i_s,j_s \rrbracket,\]
where $j_1 +2 \leq i_2$, \ldots, $j_{s-1} +2 \leq i_s$.
Then, for each interval $\llbracket i_k,j_k \rrbracket$, set
\begin{equation}
a'_k:= a_{i_k}b_{i_k} \ldots  a_{j_k}b_{j_k} a_{j_k +1},
\label{substitution}
\end{equation}
where we possibly take $a_{r+1} =1$ in case  $a_{r+1}$ appears in the formula. 
Since the elements $b_{i_k}, \ldots, b_{j_k}$ belong to $H_1 \cap H_2 = K_2$, they also belong to $K_2 \smallsetminus K_1$. Since the elements $a_{i_k}, \ldots, a_{j_k}$ belong to $ K_1 \smallsetminus H_2 = K_1 \smallsetminus K_2$ and $a_{j_k +1} \in K_1$, we get $a'_k \in H_1 \smallsetminus K_2$ by Lemma \ref{lem:amalgamH_1}. Since $H_1 \smallsetminus K_2= H_1 \smallsetminus H_2$, it follows that $a'_k \in H_1 \smallsetminus H_2$. For $1 \leq k \leq s$, make the substitution given by (\ref{substitution}) in (\ref{initial expression of w}). Then, observe that  all letters appearing in this new expression of $w$ successively belong to $H_1 \smallsetminus H_2$ (the letters $a_i$ or $a'_i$) or to 
$H_2 \smallsetminus H_1$ (the letters $b_i$, $i \notin I$). We obtain $w \neq 1$ by Proposition \ref{pro:first amalgamated structure of Stabx1}.
\end{proof}

Alternatively, Proposition~\ref{pro:second amalgamated structure of Stabx1} follows from the following remark. Let $A,B_1,B_2$ and $C$ be four groups and assume that we are given three morphisms of groups: $C \to A$, $C \to B_1$ and $B_1 \to B_2$.
Then, we have a natural isomorphism
$$(A*_CB_1)*_{B_1}B_2 \simeq A*_CB_2.$$
This isomorphism is a direct consequence of the universal property of the amalgamated product (e.g. \cite[I.1.1]{S2}).\\

\subsection{Product of trees} \label{sec:product of trees}

Following \cite{BS} we construct a product of trees in which  the complex $\Comp$ embeds.

Recall that a \textbf{hyperplane} in a $\CAT(0)$ cube complex is an equivalence class of edges, for the equivalence relation generated by declaring two edges equivalent if they are opposite edges of a same 2-dimensional  cube. 
We identify a hyperplane with its geometric realization as a convex subcomplex of the first barycentric subdivision of the ambient complex: consider geodesic segments between the middle points of any two edges in a given equivalence class. 
See \cite[\S 2.4]{Wi} or \cite[\S 3]{BS} (where hyperplanes are named hyperspaces) for alternative equivalent definitions.  

In the case of the complex $\Comp$, hyperplanes are 1-dimensional $\CAT(0)$ cube complexes, in other words they are trees.
The action of $\STSL$ on the hyperplanes of $\Comp$ has two orbits, whose representatives are the two hyperplanes through the center of the standard square.
We call them horizontal or vertical hyperplanes, in accordance with our convention for edges (see Definition \ref{def:horizontal and vertical edges}).
We define the \textbf{vertical tree} $\T_V$ as follows. 
We call \textbf{vertical region} a connected component of $\Comp$ minus all vertical hyperplanes.
The vertices of $\T_V$ correspond to such vertical regions, and we put an edge when two regions admit a common hyperplane in their closures.
The classical fact that the complement of a hyperplane has exactly 2 connected components translates into the fact that the complement of each edge in the graph $\T_V$ is disconnected, so $\T_V$ is indeed a tree.
The \textbf{horizontal tree} $\T_H$ is defined similarly.

The product $\T_V \times \T_H$ has a structure of square complex, and we put a metric on this complex by identifying each square with a Euclidean square with edges of length one.
Morever there is a natural embedding $\Comp \subseteq \T_V \times \T_H$, which is a quasi-isometry on its image (see \cite[Proposition 3.4]{BS}).
We denote by $\pi_V\colon \Comp \to \T_V$ and $\pi_H\colon \Comp \to \T_H$ the two natural projections.
Any element  $f \in \STSL$ induces an isometry on $\T_V$ and on $\T_H$, which we denote respectively by $\pi_V(f)$ and $\pi_H(f)$. 

\begin{lemma} \label{lem:elliptic on projections}
Let $f$ be an element in $\STSL$. 
Then $f$ is elliptic on $\Comp$ if and only if $f$ is elliptic on both factors $\T_V$ and $\T_H$. 
\end{lemma}

\begin{proof}
If $x \in \Comp$ is fixed, then $\pi_V(x)$ and $\pi_H(x)$ are fixed points for the induced isometries on trees.

Conversely, assume that $x_V \in \T_V$ and $x_H \in \T_H$ are fixed points for the action of $f$.
Then $x =(x_V, x_H) \in \T_V \times \T_H$ is a fixed point in the product of trees. 
Consider $d \ge 0$ the distance from $x$ to $\Comp$, and consider $B \subseteq \Comp$ the set of points realizing this distance. 
This is a bounded set (because the embedding $\Comp \subseteq \T_V \times \T_H$ is a quasi-isometry), hence it admits a circumcenter which must be fixed by $f$.     
\end{proof}

\begin{lemma} \label{lem:hyperelliptic on projections}
Let $f$ be an elliptic element in $\STSL$. 
Then $f$ is hyperelliptic on $\Comp$ if and only if $f$ is hyperelliptic on at least one of the factors $\T_V$ or $\T_H$. 
\end{lemma}

\begin{proof}
Assume $f$ hyperelliptic, and let $(y_i)_{i \ge 0}$ be a sequence of fixed points of $f$, such that $\lim_{i \to \infty} d(y_0, y_i) = \infty$.
Then one of the sequences $d(\pi_V(y_0), \pi_V(y_i))$ or $d(\pi_H(y_0), \pi_H(y_i))$ must also be unbounded.

Conversely, assume that $f$ is hyperelliptic on one of the factors, say on $\T_V$.
Let $(z_i)_{i \ge 0}$  be an unbounded sequence of fixed points in $\T_V$. 
Then for each $i$, $\pi_V^{-1}(z_i) \cap \Comp$ is a non-empty convex subset invariant under $f$. In particular it contains a fixed point $y_i$ of the elliptic isometry $f$.
The sequence $(y_i)_{i \ge 0}$ is unbounded, hence $f$ is hyperelliptic.   
\end{proof}

\begin{definition} \label{def:EV LV EH LH}
The \textbf{vertical elementary group} $E_V$ is the stabilizer of the vertical region containing $[x_1]$.
The \textbf{vertical linear group} $L_V$ is the stabilizer of the vertical region containing $[\id]$.
We can similarly define horizontal groups $E_H$ and $L_H$, by considering the stabilizers of horizontal regions containing the same vertices. 
\end{definition}

\begin{proposition}
The group $\STSL$ is the amalgamated product of $E_V$ and $L_V$ along their intersection $E_V \cap L_V$.
The same result holds for $E_H$ and $L_H$:
\[\STSL = E_V *_{E_V \cap L_V} L_V = E_H *_{E_H \cap L_H} L_H.\]
\end{proposition}

\begin{proof}
An edge in $\T_V$ corresponds to a vertical hyperplane.
Observe that vertical regions are of two types, depending whether they contain vertices of type $1$ and $2$, or of type $2$ and $3$.
In particular two vertical regions of different type cannot be in the same orbit under the action of $\STSL$.
Since $\STSL$ acts transitively on vertical hyperplanes, we obtain that $\STSL$ acts without inversion, with fundamental domain an edge, on the tree $\T_V$.
Hence $\STSL$ is the amalgamated product of the stabilizers of the vertices of an edge, which is exactly our definition of $E_V$ and $L_V$.  
\end{proof}

We denote by $\SStab([x_1])$ the group $\Stab([x_1]) \cap \STSL$.
Remark that $\Stab([x_1,x_2])$ and $\Stab([x_1,x_3])$ are already subgroups of $\STSL$.

\begin{proposition}
The group $E_V$ is the amalgamated product of $\SStab([x_1])$ and $\Stab([x_1,x_3])$ along their intersection $\Stab([x_1], [x_1,x_3])$.

The group $L_V$  is the amalgamated product of $\Stab([x_1,x_2])$ and $\SO_4$ along their intersection.

Similar structures hold for $E_H$ and $L_H$.
\end{proposition}

\begin{proof}
Let $\mathcal R$ be the vertical region containing $[x_1]$.
To prove the assertion for $E_V$, it is sufficient to show that $E_V$ acts transitively on vertical edges contained in $\mathcal R$ (clearly it acts without inversion).
But this is clear, since $\STSL$ acts transitively on vertical edges between vertices of type 1 and 2.

The proofs of the other assertions are similar.
\end{proof}

In turn, the group $E_V \cap L_V$ admits a structure of amalgamated product.

\begin{proposition} \label{pro:product E_V cap L_V}
The group $E_V \cap L_V$ is the amalgamated product of the stabilizers of edges $\Stab([x_1], [x_1,x_2])$ and $\Stab([x_1,x_3], [\id])$ along their intersection $S$.
\end{proposition}

\begin{proof}
The group $E_V \cap L_V$ acts on the vertical hyperplane through the standard square, which is a tree.
Since $\STSL$ acts transitively on squares, the fundamental domain of the action is the standard square, and $E_V \cap L_V$ is the amalgamated product of the stabilizers of the horizontal edges.
\end{proof}

On Figure \ref{fig:nesting} we try to represent all the amalgamated product structures that we found in this section.
By a diagram of the form  
$$\mygraph{
!{<0cm,0cm>;<1cm,0cm>:<0cm,1cm>::}
!~-{@{-}@[|(4)]}
!{(1,1)}*+{G}="G"
!{(0,0)}*+{A}="A"
!{(2,0)}*+{B}="B"
!{(1,-1)}*+{C}="C"
"G"-"A"-"C"-"B"-"G"
}$$
with the four edges of the same color we mean that $G$ is the amalgamated product of its subgroups $A$ and $B$ along their intersection $C = A \cap B$. 

For example, on the left hand side of Figure \ref{fig:nesting}, we see that  $\Stab([x_1])$ admits two structures of amalgamated products:  $H_1 *_{H_1 \cap H_2}H_2$ and $ K_1 * _{K_1 \cap H_2}H_2$ (see Propositions \ref{pro:first amalgamated structure of Stabx1} and \ref{pro:second amalgamated structure of Stabx1}).

We are now in position to prove that the groups $\TQ$ and $\TSL$ are isomorphic. We use the following general lemma.

\begin{lemma} \label{lem:morphism to amalgamated product}
Let $G = A *_{A \cap B} B$ be an amalgamated product and $\phi \colon G' \to G$ be a morphism.
Assume there exist subgroups $A', B'$ in $G'$ such that $G' = \langle A',B' \rangle$ and such that $\phi$ induces isomorphisms $A' \RightArrowIsomorphism A$, $B' \RightArrowIsomorphism B$ and $A' \cap B' \RightArrowIsomorphism A \cap B$.
Then $\phi$ is an isomorphism.
\end{lemma}

\begin{proof}
By the universal property of the amalgamated product, the natural morphisms $\psi_A \colon A \to G'$ and $\psi_B \colon B \to G'$  give us a morphism $\psi \colon G \to G'$ such that $\phi \circ \psi =\id_G$. It is clear that $\psi$ is an isomorphism, so that $\phi$ also.
\end{proof}

Recall that we have a natural morphism of restriction $\rho \colon \Aut _q (\C^4) \to \Aut(\SL_2)$.
We denote by $\pi$ the induced morphism on $\TQ$.

\begin{proposition} \label{pro:TQ=TSL}
The map $\pi \colon \TQ \to \TSL$ is an isomorphism.
\end{proposition}

\begin{proof}
Clearly the group $\TQ$ contains subgroups isomorphic (via the restriction map) to $H_2$, $K_1$, $K_2$, $\Er$ and $\OO_4$.
By Lemma \ref{lem:morphism to amalgamated product} applied to the various amalgamated products showed in Figure \ref{fig:nesting}, we obtain the existence of subgroups in  $\TQ$ isomorphic to $\Stab([x_1])$, $E_V$, $L_V$ and finally $\TQ \simeq \TSL$.
\end{proof}

\begin{landscape}
\begin{figure}[H]
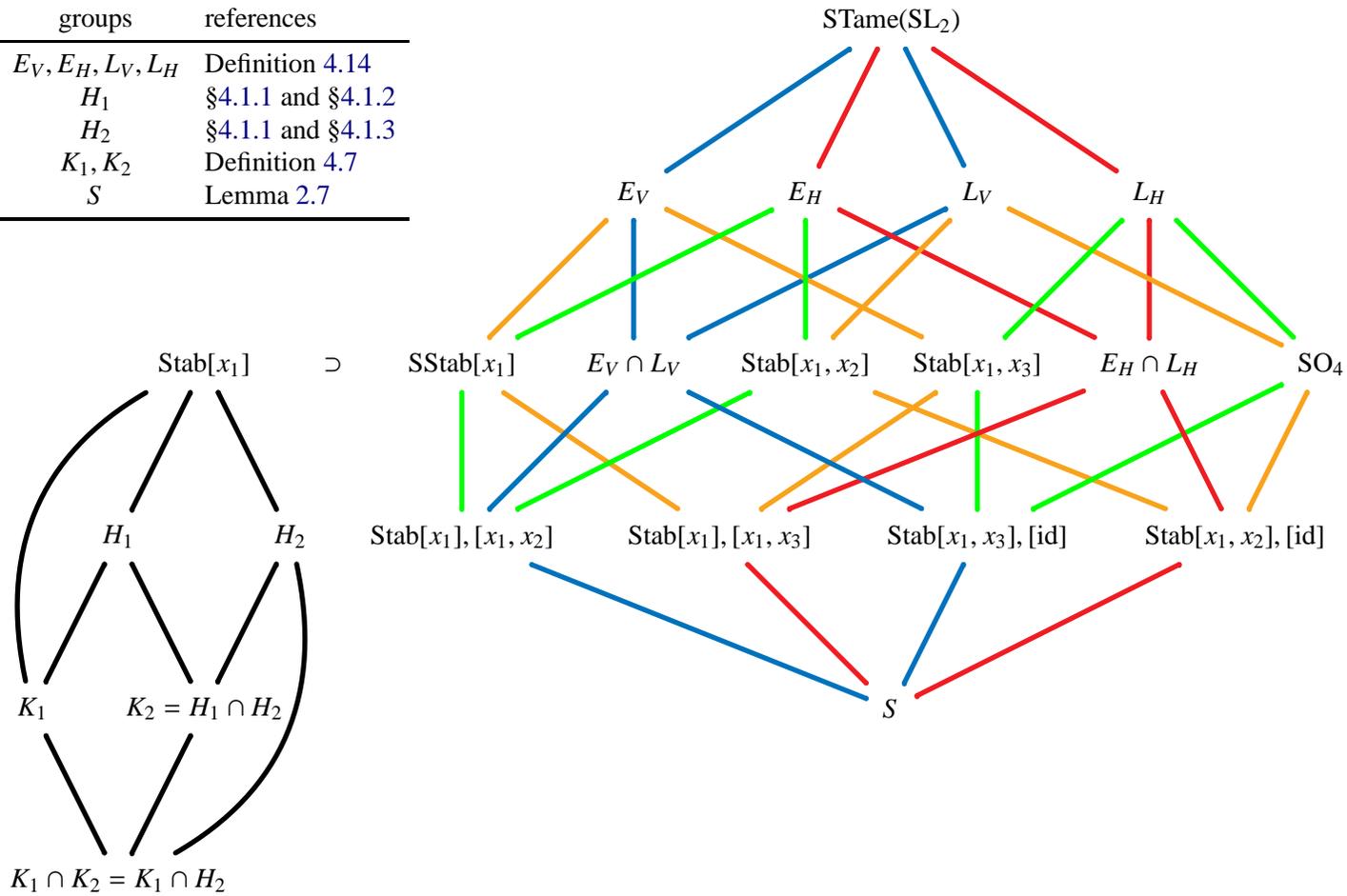

$$
\mygraph{
!{<0cm,0cm>;<2.4cm,0cm>:<0cm,2.4cm>::}
!~-{@{-}@[|(5)]}
!{(-1.5,3.5)}*{
\begin{tabular}{Cl}
\toprule
\text{groups} &\text{references} \\
\midrule
E_V, E_H, L_ V, L_H & Definition \ref{def:EV LV EH LH} \\
H_1 & \S\ref{sec:H1 H2} and \S\ref{sec:H1} \\
H_2 & \S\ref{sec:H1 H2} and \S\ref{sec:H2} \\
K_1, K_2 & Definition \ref{def:K_1 and K_2}\\
S & Lemma \ref{lem:action on squares} \\
\bottomrule
\end{tabular}
}
!{(2.5,4)}*++{\STSL}="stame"
!{(1,3)}*++{E_V}="EV"
!{(2,3)}*++{E_H}="EH"
!{(3,3)}*++{L_V}="LV"
!{(4,3)}*++{L_H}="LH"
!{(0,2)}*++{\SStab [x_1]}="stabx1"
!{(2,2)}*++{\Stab [x_1,x_2]}="stabx1x2"
!{(1,2)}*++{E_V \cap L_V}="EVLV"
!{(4,2)}*++{E_H \cap L_H}="EHLH"
!{(3,2)}*++{\Stab [x_1,x_3]}="stabx1x3"
!{(5,2)}*++{\SO_4}="so4"
!{(0,1)}*++{\Stab [x_1], [x_1,x_2]}="stabx1x1x2"
!{(1.5,1)}*++{\Stab [x_1], [x_1,x_3]}="stabx1x1x3"
!{(3,1)}*++{\Stab [x_1,x_3], [\id]}="stabx1x3id"
!{(4.5,1)}*++{\Stab [x_1,x_2], [\id]}="stabx1x2id"
!{(2.5,0)}*++{S}="s"
!{(-1.5,2)}*++{\Stab [x_1]}="stab"
!{(-0.75,2)}*{\supset}
!{(-2,1)}*++{H_1}="H1"
!{(-1,1)}*++{H_2}="H2"
!{(-2.5,0)}*++{K_1}="K1"
!{(-1.5,0)}*++{K_2 = H_1\cap H_2}="K2"
!{(-2,-1)}*++{K_1 \cap K_2= K_1 \cap H_2}="K1K2"
"stame"-@[RoyalBlue]"EV"-@[RoyalBlue]"EVLV"-@[RoyalBlue]"LV"-@[RoyalBlue]"stame"
"stame"-@[Red]"EH"-@[Red]"EHLH"-@[Red]"LH"-@[Red]"stame"
"EV"-@[YellowOrange]"stabx1"-@[YellowOrange]"stabx1x1x3"-@[YellowOrange]"stabx1x3"-@[YellowOrange]"EV"
"EH"-@[green]"stabx1"-@[green]"stabx1x1x2"-@[green]"stabx1x2"-@[green]"EH"
"LV"-@[YellowOrange]"stabx1x2"-@[YellowOrange]"stabx1x2id"-@[YellowOrange]"so4"-@[YellowOrange]"LV"
"LH"-@[green]"stabx1x3"-@[green]"stabx1x3id"-@[green]"so4"-@[green]"LH"
"EHLH"-@[Red]"stabx1x1x3"-@[Red]"s"-@[Red]"stabx1x2id"-@[Red]"EHLH"
"EVLV"-@[RoyalBlue]"stabx1x1x2"-@[RoyalBlue]"s"-@[RoyalBlue]"stabx1x3id"-@[RoyalBlue]"EVLV"
"stab"-"H1"-"K2"-"H2"-"stab"
"H1"-"K1"-"K1K2"-"K2"
"stab"-@/_12mm/"K1"
"H2"-@/^12mm/"K1K2"
}$$
\caption{Russian nesting amalgamated products.}
\label{fig:nesting}
\end{figure}
\end{landscape}

\begin{remark}
By \cite{LV}, any non-linear element of $\TSL$ admits an elementary reduction (see Theorem \ref{th:main theorem of LV}). However,  even if the groups  $\TSL$ and $\TQ$ are naturally isomorphic, we cannot deduce at once that an analogous result holds for $\TQ$. Such a result is the aim of the Annex (see Theorem \ref{thm:main}).
\end{remark}

We recall that an element $f_1$ of $\C[\SL_2]$ is called a component if it can be completed to an element  $f = \smat{f_1}{f_2}{f_3}{f_4}$ of $\TSL$ (see \S \ref{sec:definition of complex}). In the same way, an element $f_1$ of $\C [x_1,x_2,x_3,x_4]$ will be called a component if it can be completed to an element of $\TQ$. In the same spirit as Proposition \ref{pro:TQ=TSL}, we show the following stronger result.

\begin{proposition} \label{pro: compoents of TQ= components of TSL}
The canonical surjection
\[\C [x_1,x_2,x_3,x_4] \to \C[\SL_2] = \C[x_1,x_2,x_3,x_4]/(q-1)\] induces a bijection between the corresponding subsets of components.
\end{proposition}

\begin{proof}
We can associate a square complex $\tilde{\Comp}$ to the group $\TQ$ in exactly the same way  we associated a complex $\Comp$ to $\TSL$ in \S \ref{sec:definition of complex}. 
The canonical surjection, alias the restriction map, defines a continuous map $p \colon \tilde{\Comp} \to \Comp$. One would easily check that $p$ is a covering (the verification is local), so that the simple connectedness of $\Comp$ (Proposition \ref{pro:1connected}) and the obvious connectedness of $\tilde{\Comp}$ implies that $p$ is a homeomorphism. In particular, $p$ induces a bijection between vertices of type 1 of $\tilde{\Comp}$ and $\Comp$. Assume now that $u,v \in \C[x_1,x_2,x_3,x_4]$ are two components such that $u \equiv v \mod (q-1)$. The vertices $[u \mod (q-1) ]$ and $[v \mod (q-1)]$ of $\Comp$ being equal, the vertices $[u]$ and $[v]$ of $\tilde{\Comp}$ are also equal. This implies  that $v = \lambda u$ for some nonzero complex number $\lambda$. Since $u$ and $v$ induce the same (nonzero)  function on the quadric, we get $\lambda =1$, i.e. $u=v$.
\end{proof}

\section{Applications}
\label{sec:theorems}

In this section we apply the previous machinery to obtain two basic results about the group $\TSL$: the linearizability of finite subgroups and the Tits alternative.

\subsection{Linearizability}

This section is devoted to the proof of Theorem \ref{thm:mainLin} from the introduction, which states that any finite subgroup of $\TSL$ is linearizable. This is a first nice application of the action of $\TSL$ on the CAT(0) square complex $\Comp$.  

The following lemma will be used several times in the proof.
The idea comes from \cite[Proposition 4]{Furushima}.
In the statement and in the proof, we use the natural structure of vector space on the semi-group of self-maps of a vector space $V$, given by $(\lambda f + g)(v) = \lambda f(v) + g(v)$ for any $f,g\colon V \to V$, $\lambda \in \C$, $v \in V$.

\begin{lemma} \label{lem:abstract linearization}
Let $G$ be a group of transformations of a vector space $V$ that admits a semi-direct product structure $G = M \rtimes L$.  
Assume that $M$ is stable by mean (i.e. for any finite sequence $m_1, \dots, m_r$ in $M$, the mean $\frac1r \sum_1^r m_i$ is in $M$) and that $L$ is linear (i.e. $L \subseteq \GL (V)$). Then any finite subgroup in $G$ is conjugate by an element of $M$  to a subgroup of $L$.
\end{lemma}

\begin{proof}
Consider the morphism of groups 
\begin{align*}
\phi\colon G = M \rtimes L &\to L\\
g= m \circ \ell &\mapsto \ell 
\end{align*}
For any $g \in G$ we have $\phi(g)^{-1} \circ g \in M$.
Given a finite group $\Gamma \subseteq G$, define $m =  \frac{1}{|\Gamma|}\sum_{g \in \Gamma} \phi(g)^{-1} \circ g$.
By the mean property, $m \in M$.
Then, for each $f \in \Gamma$, we compute:
\begin{align*}
m \circ f  &=  \frac{1}{|\Gamma| } \sum\limits_{g \in \Gamma} \phi(g)^{-1} \circ g \circ f \\
&= \frac{1}{|\Gamma|} \sum\limits_{g \in \Gamma}\phi(f) \circ [ \phi(f)^{-1} \circ  \phi(g)^{-1}] \circ g \circ f \\
&= \phi(f) \circ m.
\end{align*}
Hence $m \Gamma m^{-1}$ is equal to $\phi(\Gamma)$, which is a subgroup of $L$.
\end{proof}

As a first application, we solve the problem of linearization for finite subgroups in the triangular group of $\Aut(\C^n)$.
Recall that $f = (f_1, \dots, f_n) \in \Aut(\C^n)$ is \textbf{triangular} if for each $i$, $f_i = a_i x_i + P_i$ where $P_i \in \C[x_{i+1}, \dots, x_n]$.  

\begin{corollary} \label{cor:linearization in triangular}
Let $\Gamma \subseteq \Aut(\C^n)$ be a finite group.
Assume that $\Gamma$ lies in the triangular group of $\Aut(\C^n)$.
Then $\Gamma$ is diagonalizable inside the triangular group. 
\end{corollary}

\begin{proof}
Apply Lemma \ref{lem:abstract linearization} by taking $G$ the triangular group, $L$ the group of diagonal matrices and $M$ the group of unipotent triangular automorphisms, that is to say with all $a_i = 1$.  
\end{proof}

\begin{proof}[Proof of Theorem \ref{thm:mainLin}]
Let $\Gamma$ be a finite subgroup of $\TSL$. 
The circumcenter $x$ of any orbit is a fixed point under the action of $\Gamma$. 
We claim that $\Gamma$ also fixes a vertex of $\Comp$. 
Indeed, let $\Sq$ be the support of $x$, that is, the cell of minimal dimension containing $x$.
If $\Sq$ is a vertex, there is nothing to prove. 
If $\Sq$ is an edge, then since its two vertices are not of the same type, $\Sq$ is fixed pointwise and its vertices also.
If $\Sq$ is a square, its two vertices of type 1 and 3 
are necessarily fixed by $\Gamma$ (but its two vertices of type 2 may be interchanged).

Up to conjugation, we may assume that $\Gamma$ fixes $[ \id]$, $[x_1,x_3]$ or $[x_1]$.

If $\Gamma$ fixes $[\id]$, this means that $\Gamma$ is included into $\OO_4$: There is nothing more to prove.

If $\Gamma$ fixes $[x_1,x_3]$, recall that $\Stab([x_1,x_3]) = \Er \rtimes \GL_2$ (Lemma \ref{lem:stab x1x3}). 
We conclude by Lemma \ref{lem:abstract linearization}, using the natural embedding $\Stab([x_1,x_3]) \to \Aut(\C^4)$.

Finally, assume that $\Gamma$ fixes $[x_1]$.
The group $\Stab([x_1])$ being the amalgamated product of its two subgroups $K_1$ and $H_2$ along their intersection (see Lemma~\ref{pro:second amalgamated structure of Stabx1}), we may assume, up to conjugation in $\Stab([x_1])$, that $\Gamma$ is included into $K_1$ or $H_2$ (e.g. \cite[I.4.3, Th. 8, p. 53]{S2}).

By forgetting the fourth coordinate, the group  $K_1$ may be identified with the  subgroup $\tilde{K_1}$ of $\Aut ( \A^3)$ whose elements are of the form
$$ (ax_1, bx_2 + ax_1 P(x_1), b^{-1} x_3 + ax_1 Q(x_1) ) \hspace{3mm} \text{or} \hspace{3mm}  (ax_1, b^{-1} x_3 + ax_1 Q(x_1),  bx_2 + ax_1 P(x_1)).$$
Then we can apply Lemma  \ref{lem:abstract linearization}, using the  embedding $\tilde{K_1} \to \Aut(\C^3)$ and the semi-direct product $\tilde{K_1} = M \rtimes L$, where
\begin{align*}
M &= \left\lbrace \left( x_1, x_2 + x_1 P(x_1),  x_3 + x_1 Q(x_1) \right) ;\; P,Q \in \C[x_1] \right\rbrace; \\
L &= \left\lbrace \left(ax_1, bx_2 , b^{-1} x_3\right) \text{ or }  \left(ax_1, b^{-1} x_3 ,  bx_2\right);\; a,b \in \C^* \right\rbrace. 
\end{align*}

Similarly, the group  $H_2$ may be identified with the  subgroup of triangular automorphisms of $\Aut(\C^3)$ whose elements are of the form
$$(x_1,x_3,x_2) \mapsto \left(ax_1, b^{-1}x_3 + x_1 Q(x_1), bx_2 + x_1 P(x_1,x_3)\right).$$
Then we can apply Corollary \ref{cor:linearization in triangular}.
\end{proof}

\subsection{Tits alternative}  \label{sec:Tits alternative}

A group satisfies the \textbf{Tits alternative} (resp. the \textbf{weak Tits alternative}) if each of its subgroups (resp. finitely generated subgroups) $H$ satisfies the following alternative: Either $H$ is virtually solvable (i.e. contains a solvable subgroup of finite index), or $H$ contains a free subgroup of rank 2.

It is known that $\Aut(\C^2)$ satisfies the Tits alternative (\cite{Ltits}), and that $\Bir(\p^2)$ satisfies the weak Tits alternative (\cite{Can}). 
One common ingredient to obtain the Tits alternative for $\TSL$ or for $\Bir(\p^2)$ is the following result (see \cite[Lemma 5.5]{Dinh}) asserting that groups satisfying the Tits alternative are stable by extension: 

\begin{lemma}  \label{lem:Tits alternative is stable by extension}
Assume that we have a short exact sequence of groups:
$$1 \to A \to B \to C \to 1,$$
where $A$ and $C$ are virtually solvable (resp. satisfy the Tits alternative), then $B$ is also virtually solvable (resp. also satisfies the Tits alternative).
\end{lemma}

We shall also use the following elementary lemma about behavior of solvability under taking Zariski closure. 

\begin{lemma}\label{lem:closure}
Let $A \supseteq B$ be  subgroups of $\SL_2$.
\begin{enumerate}
\item We have $[ \overline{A} : \overline{B} ] \leq [A:B]$;
\item We have $D\left( \overline{A} \right) \subseteq \overline{D(A)},$ where we denote $D(G)$ the derived subgroup of the group $G$;
\item If $A$ is solvable, then $\overline{A}$ is also;
\item If $A$ is virtually solvable, then $\overline{A}$ is also.
\end{enumerate}
\end{lemma}

\begin{proof}
(1) If $[A:B] = + \infty$, there is nothing to show. If $[A:B]$ is an integer $n$, there exist elements $a_1,\ldots,a_n$ of $A$ such that $A = \bigcup_i a_iB$. By taking the closure, we obtain  $\overline{A} = \bigcup_i a_i\overline{B}$ and the result follows.

(2) The preimage of the closed subset $\overline{D(A)} \subseteq \SL_2$ under the commutator morphism is closed in $\SL_2 \times \SL_2$.

(3) There exists a sequence of subgroups of $\SL_2$ such that
\[ A=A_0 \supseteq A_1 \supseteq \cdots \supseteq A_n = \{ 1 \} \quad \text{and} \quad D(A_i) \subseteq A_{i+1} \quad \text{for each $i$.} \]
By the last point, we immediately obtain
\[ \overline{A}=\overline{A}_0 \supseteq \overline{A}_1 \supseteq \cdots \supseteq \overline{A}_n = \{ 1 \} \quad \text{and} \quad D\left( \overline{A}_i \right)\subseteq \overline{A}_{i+1} \quad \text{for each $i$.} \]

(4) This is a direct consequence of points (1) and (3).
\end{proof}

We now want to apply a general theorem by Ballmann and {\'S}wi{\polhk{a}}tkowski.
In order to state their result we recall basic notions related to cubical complexes, which are cell complexes in which every cell is combinatorially equivalent to a cube.
We say that $X$ has dimension $d$ if the maximal dimension of cells is $d$.
A $d$-dimensional cubical complex X is \textbf{foldable} if there exists a combinatorial map of $X$ onto an $d$-cube which is injective on each cell of $X$.
We say that $X$ is \textbf{gallery connected} if any two top-dimensional $d$-cells are linked by a sequence of $d$-cells where any two consecutive cells have a $(d-1)$-cell in common.

\begin{theorem}[{\cite[Theorem 2]{BS}}]\label{thm:Ballmann Swiatkowski}
Let $X$ be a $d$-dimensional $\CAT(0)$, foldable, gallery connected cubical
 complex and $\Gamma \subseteq \Aut (X)$ a subgroup. Suppose that $\Gamma$ does not contain a free nonabelian subgroup acting freely on $X$. Then up to passing to a subgroup of finite index, there is a surjective homomorphism $h \colon \Gamma \to \Z^k$ for some $k \in \{0,\ldots,d \}$ such that the kernel $\Delta$ of $h$ consists precisely of the elliptic elements of $\Gamma$ and, furthermore, precisely one of the following three possibilities occurs:
\begin{enumerate}
\item $\Gamma$ fixes a point in $X$ (then $k=0$).
\item $k \geq 1$ and there is a $\Gamma$-invariant convex subset $E \subseteq X$ isometric to $k$-dimensional Euclidean space such that $\Delta$ fixes $E$ pointwise and such that $\Gamma / \Delta$ acts on $E$ as a cocompact lattice of translations. In particular, $\Gamma$ fixes each point of $E( \infty) \subseteq X(\infty)$.
\item $\Gamma$ fixes a point of $X ( \infty)$, but $\Delta$ does not fix a point in $X$. There is a sequence $(x_m)$ in $X$ which converges to a fixed point of $\Gamma$ in $X ( \infty)$ and such that the groups $\Delta_n := \Delta \cap \Stab (x_n)$ form a strictly increasing filtration of $\Delta$, i.e. $\Delta _n \subsetneqq \Delta _{n+1}$ and $\bigcup \Delta_n = \Delta$.
\end{enumerate}
\end{theorem}

In our situation, the result translates as

\begin{corollary} \label{cor:ballman}
Let $\Gamma \subseteq \TSL$ be a subgroup which does not contain a free subgroup of rank 2, and consider the derived group $D(\Gamma)$.
Then one of the following possibilities occurs:
\begin{enumerate}
\item $D(\Gamma)$ is elliptic.
\item There is a morphism $h \colon D(\Gamma) \to \Z$ such that the kernel of $h$ is elliptic or parabolic.
\item $D(\Gamma)$ is parabolic.
\end{enumerate}
\end{corollary}

\begin{proof}
By Lemma \ref{lem: 2a and 2b} the complex $\Comp$ admits four orbits of vertices under the action of $\STSL$, which are represented by the four vertices of the standard square.
This implies that $\Comp$ is foldable.
Thus $\Comp$ satisfies the hypothesis of Theorem~\ref{thm:Ballmann Swiatkowski} with $d=2$. 
Furthermore, since by Proposition~\ref{pro:nogrid} $\Comp$ does not contain a Euclidean plane, we must have $k=1$  in case (2). 
Now we review the proof of the theorem in order to see where it was necessary to pass to a subgroup of finite index.
The argument is to project the action of $D(\Gamma)$ on each factor, and to use the classical fact that a group that does not contain a free group of rank 2 and that acts on a tree is elliptic, parabolic or loxodromic \cite{PV}.
In the loxodromic case, in order to be sure that the pair of ends is pointwise fixed, in general we need to take a subgroup of order 2. 
But in our case $D(\Gamma)$ is a derived subgroup hence this condition is automatically satisfied.  
\end{proof}

Now we are essentially reduced to the study of elliptic and  parabolic subgroups in $\TSL$.
 
\begin{proposition}\label{pro:elliptic subgroups satisfy Tits alternative}
Let $\Delta \subseteq \TSL$ be an elliptic subgroup.
Then $\Delta$ satisfies the Tits alternative.
\end{proposition}

\begin{proof}
If the globally fixed vertex $v$ is of type 1, we may assume that $v =[x_1]$. 
The stabilizer $\Stab ( [x_1])$ of $v$ is equal to the set of
automorphisms $f= \smat{f_1}{f_2}{f_3}{f_4}$  such that $f_1=ax_1$ for some $a \in \C^*$. The natural morphism of groups:
$$\Stab ( [x_1]) \to \C^*, \hspace{2mm} \mat{ax_1}{f_2}{f_3}{f_4} \mapsto a$$
is surjective. By Corollary \ref{cor:linktype1}, its kernel is a subgroup of $\Aut_{\C(x_1)} \C(x_1)[x_2,x_3]$. By \cite{Ltits}, $\Aut_{\C} \C[x_2,x_3]$ satisfies the Tits alternative, but the proof would be analogous for $\Aut_{K} K[x_2,x_3]$ for any field $K$ of characteristic zero.
Therefore,  Lemma \ref{lem:Tits alternative is stable by extension} shows us  that  $\Stab ([x_1])$, hence also $\Delta$, satisfies the Tits alternative.

If the vertex $v$ is of type 2, we may assume that $v = [x_1,x_3]$.
The stabilizer $\Stab ( [x_1,x_3])$ of $v$ is equal to the set of
automorphisms $f= \smat{f_1}{f_2}{f_3}{f_4}$  such that $\Vect(f_1,f_3)=\Vect(x_1,x_3)$. 
By Lemma \ref{lem:stab x1x3}, the natural morphism
$$\Stab ( [x_1,x_3]) \to  \Aut( \Vect(x_1,x_3)  ) \simeq \GL_2, \hspace{2mm} \mat{f_1}{f_2}{f_3}{f_4} \mapsto (f_1,f_3)$$
is surjective, and its kernel is the group $\Er$. The  group $\GL_2$  is linear, hence satisfies the Tits alternative (\cite{delaHarpe}) and the group $\Er$ is abelian. Therefore,  by Lemma \ref{lem:Tits alternative is stable by extension} the group $\Stab ( [x_1,x_3])$ satisfies the Tits alternative. 

If the vertex $v$ is of type 3, we may assume that  $v = \sbmat{x_1}{x_2}{x_3}{x_4}$.
The stabilizer of $v$ is the orthogonal group $\OO_4$, which is linear hence satisfies the Tits alternative.
\end{proof}

\begin{proposition}\label{pro:parabolic subgroups}
Let $\Delta \subseteq \TSL$ be a parabolic subgroup.
Then $\Delta$ is virtually solvable.
\end{proposition}

\begin{proof}
The case of a parabolic subgroup $\Delta$ corresponds to Case (3) in Theorem \ref{thm:Ballmann Swiatkowski}, from which we keep the notations.
We may assume that all points $x_m$ are vertices of $\Comp$ (replace $x_m$ by one of the vertices of the cell containing $x_m$). 
For each $m$, consider the geodesic segment $S_m$ joining $x_m$ to $x_{m+1}$. Let $U_m$ be the union of the cells of $\Comp$ intersecting $S_m$.
Take $S'_m$ an edge-path geodesic segment of $\Comp$ joining $x_m$ to $x_{m+1}$ included into $U_m$, such that $S'_m \subseteq S'_{m+1}$ for all $m$. 
By considering the sequences of vertices on the successive $S'_m$, we obtain a sequence of vertices $y_i$, $i \ge 0$ such that:
\begin{itemize}
\item The sequence $x_m$ is a subsequence of $y_i$;
\item For each $i \ge 0$, $d(y_i, y_{i+1}) = 1$.
\end{itemize}

For each $m \geq 0$ we set
$$\Delta'_m = \Delta \cap \bigcap_{i \,\geq \, m} \Stab(y_i).$$

By construction the $\Delta'_m$ form an increasing filtration of $\Delta$. For $1 \leq j \leq 3$, let $X_j$ be the set of integers $i$ such that $y_i$ is a vertex of type $j$. One of the three following cases is satisfied:
\begin{enumerate}[a)]
\item $X_1$ and $X_3$ are infinite;
\item $X_1$ is infinite and $X_3$ is finite;
\item $X_1$ is finite and $X_3$ is infinite.
\end{enumerate}

In case a), there exists an infinite subset $A$ of $\N$ such that for all $a\in A$, the vertices $y_a, y_{a+1}, y_{a+2}$ are of type $1,2,3$ respectively. Note that the group  $\bigcap\limits_{a \,\leq \, i \, \leq \, a+2} \Stab(y_i)$ is conjugate to the group
$$S = \Stab ([x_1]) \cap \Stab ( [x_1,x_2]) \cap  \Stab ( [\id]),$$
which is the stabilizer of the standard square. 
Recall from Lemma \ref{lem:action on squares} that 
$$S=\left\{ \mat{ax_1}{b(x_2 + cx_1)}{b^{-1}(x_3 +d x_1)}{\dots}, \; a,b,c,d \in \C, \; ab \neq 0 \right\}$$
and so the second derived subgroup of $S$ is trivial: $D_2(S) = \{ 1 \}$. 
Therefore, $D_2 (\Delta'_a) =\{1 \}$ for each $a \in A$ and since $\Delta = \bigcup\limits_{a \in A}\Delta'_a$, we get $D_2 (\Delta) = 1$.\\

In case b), changing the first vertex we may assume that $X_3 = \emptyset$, that the vertices $y_{2i}$ of even indices are of type 2 and that the vertices $y_{2i+1}$ of odd indices are of type 1. Note that the group  $\bigcap\limits_{2a-1 \,\leq \, i \, \leq \, 2a+1} \Stab(y_i)$ is conjugate to the group
$${\widetilde \Er}= \Stab ([x_1]) \cap \Stab ( [x_1,x_3]) \cap  \Stab ([x_3]).$$
By Lemma \ref{lem:action on v3 next to v2} we have
$${\widetilde \Er} = \left\lbrace \mat{ax_1}{b^{-1} x_2+ax_1P(x_1,x_3)}{bx_3}{a^{-1} x_4+bx_3P(x_1,x_3)}; \;a,b \in \C^*, \, P \in \C[x_1,x_3] \right\rbrace$$
and thus $D_2\left( \widetilde \Er \right) = \{ 1 \}$. 
Therefore $\Delta'_{2a-1} =1$ and finally $D_2(\Delta) = 1$.\\

In case c), we may assume that $X_1 = \emptyset$, that the vertices $y_{2i}$ of even indices are of type 2 and that the vertices $y_{2i+1}$ of odd indices are of type 3. Note that the group  $\bigcap\limits_{2a \,\leq \, i \, \leq \, 2a+2} \Stab(y_i)$ is conjugate to the group
$$\Stab ([x_1,x_2]) \cap \Stab ( [\id]) \cap  \Stab ( [x_3,x_4]) \simeq \GL_2.$$
Up to passing again to the derived subgroup, we can assume that all $\Delta_n$ are conjugate to subgroups of $\SL_2$, where $\SL_2$ is identified to a subgroup of $\SO_4$ via the natural injection 
$$\SL_2 \to \SO_4, \;\mat{a}{b}{c}{d} \mapsto \mat{a}{b}{c}{d} \cdot \mat{x_1}{x_2}{x_3}{x_4}.$$
Since $\SL_2$ satisfies the Tits alternative, all $\Delta_n$, which by hypothesis do not contain free subgroups of rank 2, are virtually solvable.
By Lemma \ref{lem:closure}, the Zariski closure $\overline{\Delta}_n$ is again virtually solvable. 

If $\overline{\Delta}_n$ is finite for all $n$, since there is only a finite list of finite subgroups of $\SL_2$ which are not cyclic or binary dihedral, we conclude that all $\Delta_n$ are contained in binary dihedral groups hence solvable of index at most 3.

Now if $\dim \overline{\Delta}_n \ge 1$ for $n$ sufficiently large, then up to conjugacy the identity component $ ( \overline{\Delta}_n )^{\circ}$ of $\overline{\Delta}_n$ is equal to one of the following three groups:
\begin{multline*}
T= \left\{ \mat{\lambda}{0}{0}{\lambda^{-1}}, \; \lambda \in \C^* \right\}, \; A= \left\{ \mat{1}{\mu}{0}{1}, \; \mu \in \C \right\} \hspace{2mm}\\
\mbox{or} \; B= \left\{ \mat{\lambda}{\mu}{0}{\lambda^{-1}}, \; \lambda \in \C^*, \; \mu \in \C \right\}.
\end{multline*}

Therefore, $\overline{\Delta}_n$ is contained in the normalizer in $\SL_2$ of these groups. Since
\[\Nor_{\SL_2}(T)=\left\{  \mat{\lambda}{0}{0}{\lambda^{-1}}, \; \lambda \in \C^* \right\} \cup \left\{  \mat{0}{\lambda^{-1}}{\lambda}{0}, \; \lambda \in \C^* \right\}\]
and $ \Nor_{\SL_2}(A) = \Nor_{\SL_2}(B)=B$ are solvable of index 2, we conclude that $\Delta_n$ is solvable of index at most 2.

Finally in all cases $\Delta = \cup \Delta_n$ is solvable of index at most 3.
\end{proof}

We are now ready to prove Theorem \ref{thm:mainTits} from the introduction, that is, the Tits alternative for $\TSL$.

\begin{proof}[Proof of Theorem \ref{thm:mainTits}]
Let $\Gamma$ be a subgroup of $\TSL$, and assume that $\Gamma$ does not contain a free subgroup of rank $2$.  
We want to prove that $\Gamma$ is virtually solvable.
By Lemma \ref{lem:Tits alternative is stable by extension}, without loss in generality we can replace $\Gamma$ by its derived subgroup.
By Corollary \ref{cor:ballman} we have a short exact sequence
$$1 \to \Delta \to \Gamma \to \Z^k \to 1$$
with $k = 0$ or $1$.
By Lemma \ref{lem:Tits alternative is stable by extension}, it is enough to prove that $\Delta$ is virtually solvable.
When $\Delta$ is elliptic the result follows from
Proposition~\ref{pro:elliptic subgroups satisfy Tits alternative}, and when $\Delta$ is parabolic this is Proposition \ref{pro:parabolic subgroups}.
\end{proof} 

\section{Complements}
\label{sec:complements}

In this section we first provide examples of hyperbolic or hyperelliptic elements in $\TSL$, and also an example of parabolic subgroup.
Then we discuss several questions about the usual tame group of the affine space, the relation between $\Autq$ and $\Aut(\SL_2)$, and finally the property of infinite transitivity.

\subsection{Examples}

\subsubsection{Hyperbolic elements}

The following lemma allows us to produce some hyperbolic elements in $\TSL$, which are very similar to generalized H\'enon mapping on $\C^2$ from an algebraic point of view.

\begin{lemma} \label{lem:hyperbolic elements}
Let $P_1, \dots, P_r \in \C[x_2,x_4]$ be polynomials of degree at least 2, and $a_1, b_1, \dots, a_r,b_r \in \C^*$ be nonzero constants. Set 
$$g_i = \mat{b_i^{-1}x_2}{a_ix_1 + a_ix_2P_i(x_2,x_4)}{-a_i^{-1}x_4}{-b_ix_3-b_ix_4P_i(x_2,x_4)}.$$
Then the composition $g_r \circ \dots \circ g_1$ is a hyperbolic element of $\TSL$.
\end{lemma}

\begin{proof}
We have
$$g_i = \mat{b_i^{-1}x_2}{a_ix_1 + a_ix_2P(x_2,x_4)}{-a_i^{-1}x_4}{-b_ix_3-b_ix_4P(x_2,x_4)} = \mat{b_i^{-1}x_2}{a_ix_1}{-a_i^{-1}x_4}{-b_ix_3} \circ \mat{x_1 + x_2P_i(x_2,x_4)}{x_2}{x_3+x_4P_i(x_2,x_4)}{x_4}.$$
Since 
$$\mat{b_i^{-1}x_2}{a_ix_1}{-a_i^{-1}x_4}{-b_ix_3} \;\text{ and }\; \mat{x_1 + x_2P_i(x_2,x_4)}{x_2}{x_3+x_4P_i(x_2,x_4)}{x_4}$$
preserve respectively the edges between $[x_1,x_2]$ and $[id]$ and between $[x_2]$ and $[x_2,x_4]$, we get that $g_i$ preserves the hyperplane $\mathcal H$ associated with these two edges (see Figure \ref{fig:hyperplane}). 

Recall that $\mathcal H$ is a one-dimensional convex subcomplex of (the first barycentric subdivision of) $\Comp$, in particular $\mathcal H$ is a tree.
By \cite[II.6.2(4)]{BH}, since $\mathcal H$ is invariant under $g_i$, the translation length of $g_i$ on $\Comp$ is equal to the translation length of its restriction $g_i|_{\mathcal H}$, which is 2.
Indeed $\Stab(\mathcal H)$ is the amalgamated product of the stabilizers of  the edges between $[x_1,x_2]$ and $[id]$ and between $[x_2]$ and $[x_2,x_4]$, and $g_i$ is a word of length 2 in this product.
Similarly, $g_r \circ \dots \circ g_1 \in \Stab(\mathcal H)$ has length $2r$ in the amalgamated product, hence is hyperbolic with translation length equal to $2r$. 
\end{proof}

\begin{figure}[ht]
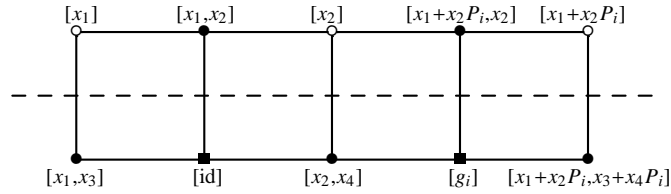

$$\mygraph{
!{<0cm,0cm>;<1.7cm,0cm>:<0cm,1.7cm>::}
!~-{@{-}@[|(2)]}
!{(-1.5,.5)}="deb"
!{(3.5,.5)}="fin"
!{(-1,1)}*{\circ}="x1"
!{(-0.96,1)}="x1right"
!{(-1,0.968)}="x1down"
!{(0,1)}*-{\bullet}="x1x2"
!{(1,1)}*{\circ}="x2"
!{(1.04,1)}="x2right"
!{(0.96,1)}="x2left"
!{(1,0.968)}="x2down"
!{(2,1)}*-{\bullet}="x1+x2"
!{(3,1)}*{\circ}="x1+"
!{(2.96,1)}="x1+left"
!{(3,0.968)}="x1+down"
!{(-1,0)}*-{\bullet}="x1x3"
!{(0,0)}*-{\blob}="id"
!{(1,0)}*-{\bullet}="x2x4"
!{(2,0)}*-{\blob}="g"
!{(3,0)}*-{\bullet}="x1+x3+"
"x1right"-^<{[x_1]}^>{[x_1,x_2]}"x1x2"-^>{[x_2]}"x2left"
"x2right"-^>{[x_1 + x_2P_i, x_2]}"x1+x2"-^>{[x_1 + x_2P_i]}"x1+left"
"x1down"-"x1x3"-_<{[x_1, x_3]}_>{[\id]}"id"-_>{[x_2,x_4]}"x2x4"-_>{[g_i]}"g"-_>{[x_1 + x_2P_i, x_3+x_4P_i]}"x1+x3+"
"x1x2"-"id"
"x2down"-"x2x4"
"x1+x2"-"g"
"x1+down"-"x1+x3+"
"deb"-@{--}"fin"
}$$
\caption{Part of the hyperplane associated with the edge $[x_2], [x_2,x_4]$.}
\label{fig:hyperplane}
\end{figure}

The previous examples induce hyperbolic isometries on the vertical tree $\T_V$, but they project as elliptic isometries on the factor $\T_H$.
Here is an example which is hyperbolic on both factors:

 \begin{figure}[ht]
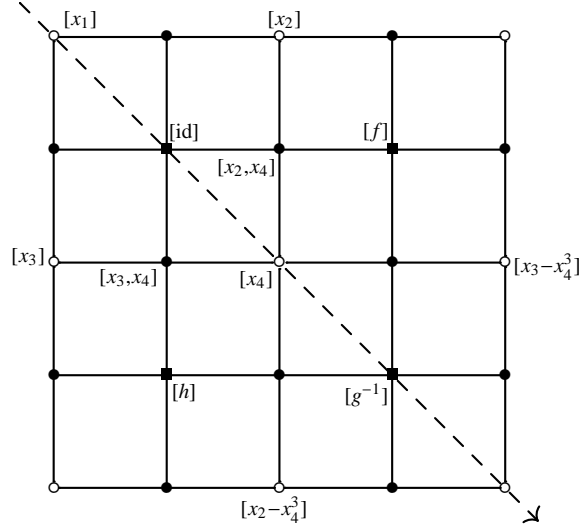

$$\mygraph{
!{<0cm,0cm>;<1.5cm,0cm>:<0cm,1.5cm>::}
!~-{@{-}@[|(2)]}
!{(-2.3,2.3)}="start"
!{(2.29,-2.29)}="end"
!{(-1,1)}*-{\blob}="x1"
!{(0,1)}*-{\bullet}="x1x2"
!{(-1,0)}*-{\bullet}="x1x3"
!{(1,1)}*-{\blob}="x2"
!{(1,0)}*-{\bullet}="x2x4"
!{(1,-1)}*-{\blob}="x4"
!{(0,-1)}*-{\bullet}="x3x4"
!{(-1,-1)}*-{\blob}="x3"
!{(0,0)}*{\circ}="x1x2x3x4"
!{(0,2)}*{\circ}="N"
!{(-.049,2)}="Nleft"
!{(.046,2)}="Nright"
!{(0,1.96)}="Ndown"
!{(0,-2)}*{\circ}="S"
!{(-.048,-2)}="Sleft"
!{(.045,-2)}="Sright"
!{(2,0)}*{\circ}="E"
!{(1.953,0)}="Eleft"
!{(2,-.04)}="Edown"
!{(-2,0)}*{\circ}="W"
!{(-1.953,0)}="Wright"
!{(-2,-.04)}="Wdown"
!{(2,2)}*{\circ}="NE"
!{(1.953,2)}="NEleft"
!{(2,1.96)}="NEdown"
!{(2,-2)}*{\circ}="SE"
!{(1.953,-2)}="SEleft"
!{(-2,2)}*{\circ}="NW"
!{(-1.953,2)}="NWright"
!{(-2,1.96)}="NWdown"
!{(-2,-2)}*{\circ}="SW"
!{(-1.953,-2)}="SWright"
!{(-1,2)}*-{\bullet}="NNW"
!{(1,2)}*-{\bullet}="NNE"
!{(2,1)}*-{\bullet}="NEE"
!{(2,-1)}*-{\bullet}="SEE"
!{(1,-2)}*-{\bullet}="SSE"
!{(-1,-2)}*-{\bullet}="SSW"
!{(-2,-1)}*-{\bullet}="SWW"
!{(-2,1)}*-{\bullet}="NWW"
"x1"-^<(.18){[\id]}_>(.7){[x_2,x_4]}"x1x2"-^>(.83){[f]}"x2"-"x2x4"-"x4"-^<(.25){[g^{-1}]}"x3x4"-^>(.82){[h]}"x3"-^>(.85){[x_3,x_4]}"x1x3"
"x1x3"-"x1"
"x1x2x3x4"-"x1x2" "x1x2x3x4"-"x2x4" "x1x2x3x4"-"x3x4"
"x1x2x3x4"-^<(.15){[x_4]}"x1x3"
"NWright"-^<(0.2){[x_1]}"NNW"-"Nleft"
"Nright"-^<{[x_2]}"NNE"-"NEleft"
"NEdown"-"NEE"-"E"
"Edown"-^<{[x_3-x_4^3]}"SEE"-"SE"
"SEleft"-"SSE"-"Sright"
"Sleft"-^<{[x_2-x_4^3]}"SSW"-"SWright"
"SW"-"SWW"-"Wdown"
"W"-^<{[x_3]}"NWW"-"NWdown"
"NWW"-"x1"-"NNW"
"NEE"-"x2"-"NNE"
"SWW"-"x3"-"SSW"
"SEE"-"x4"-"SSE"
"Ndown"-"x1x2" "Eleft"-"x2x4" "S"-"x3x4" "Wright"-"x1x3"
"start"-@{--}"x1x2x3x4"-@{-->}"end"
}$$
\caption{Geodesic through $[x_1],$ $g \cdot [x_1]=[x_4]$ and $g^2 \cdot [x_1]= g \cdot[x_4]$.}
\label{fig:hyperbolicweaklyreg}
\end{figure}

\begin{example}\label{exweakreg}
Consider the following automorphism $g$ of $\TSL$:
$$g=\mat{x_4+x_3x_1^2+x_2x_1^2+x_1^5}{x_2+x_1^3}{x_3+x_1^3}{x_1}.$$
Its inverse $g^{-1}$ is:
$$g^{-1}=\mat{x_4}{x_2-x_4^3}{x_3-x_4^3}{x_1-x_4^5-x_4^2(x_2-x_4^3)-(x_3-x_4^3)x_4^2}.$$
The automorphism $g$ is hyperbolic, as a consequence of Lemma \ref{lem:hyperbolic_elements}: If we compute the geodesic through $[x_1],$ $g  \cdot [x_1]$ and $g^2 \cdot [x_1]$ we find the segment $[x_1],[x_4],g \cdot [x_4]$ (see Figure \ref{fig:hyperbolicweaklyreg})  on which $g$ acts as a translation of length $2 \sqrt{2}.$ 
\end{example}

\subsubsection{Two classes of examples of hyperelliptic elements}  \label{sec:hyperelliptic}

Recall that an elliptic element of $\TSL$ is said to be hyperelliptic if $\Min(f)$ is unbounded.
In this section we gives some examples of hyperelliptic elements.

\begin{definition} \label{def:resonant}
We say that two numbers $a,b \in \C^*$ are \textbf{resonant} if they satisfy a relation $a^pb^q = 1$ for some $p,q \in \Z \smallsetminus \{0\}$.  
We say that a polynomial $R \in \C[x,y]$ is \textbf{resonant} in $a$ and $b$ if $R$ is not constant and $abR(ax, by) = R(x,y)$. 
\end{definition}

\begin{remark}
\begin{enumerate}
\item A polynomial $R$ is resonant in $a$ and $b$ if and only if it is resonant in $a^{-1}$ and $b^{-1}$.
On the other hand, the condition $R$ resonant in $a$ and $b$ is \textit{not} equivalent to $R$ resonant in $b$ and $a$. 
\item If $R = \sum r_{i,j}x^iy^j$, the condition $abR(ax, by) = R(x,y)$ is equivalent to the implication $r_{i,j} \neq 0 \Rightarrow a^{i+1}b^{j+1} = 1$.  
\item There exist some polynomials that are not resonant in $a$ and $b$ for any $(a,b) \in (\C^*)^2 \smallsetminus \{ (1,1)\}$. For instance $P(x,y) = x^2 + x^3 + y^2 + y^3$ is such a polynomial.
\end{enumerate}
\end{remark}

\begin{lemma} \label{lem:resonnant hyperelliptic}
If $a, b \in \C^*$ are resonant, then $f = \mat{ax_1}{b^{-1}x_2}{bx_3}{a^{-1}x_4}$ is hyperelliptic.
\end{lemma}

\begin{proof}
By Lemma \ref{lem: criterion of hyperellipticity}, to prove that  $f$ is hyperelliptic it is sufficient to show that $f$ commutes with some hyperbolic element.
By assumption there exist $p,q \in \Z \smallsetminus \{0\}$ such that $a^pb^q = 1$.
We can assume that $p,q$ have the same sign, by considering $\tau f \tau$ instead of $f$ if necessary, where $\tau$ is the transpose automorphism.
Moreover, up to replacing $f$ by $f^{-1}$, hence $a$ and $b$ by their inverses, we can assume $p,q \ge 1$.
We set 
$$g= \mat{-x_2}{-x_1 - x_2P(x_2,x_4)}{x_4}{x_3+x_4P(x_2,x_4)},$$ 
where $P \in \C[x_2,x_4]$ is a polynomial of degree at least $2$ that is resonant in $b$ and $a$. 
Denote  
$$\sigma = \mat{-x_3}{x_4}{-x_1}{x_2},  \quad \tilde f = \sigma f^{-1} \sigma = \mat{b^{-1}x_1}{ax_2}{a^{-1}x_3}{bx_4} \quad \text{ and } \tilde g = \sigma g \sigma.$$
We compute 
\begin{multline*}
g \circ f =  \mat{-b^{-1}x_2}{-ax_1 - b^{-1}x_2P\left(b^{-1}x_2,a^{-1}x_4\right)}{a^{-1}x_4}{bx_3+a^{-1}x_4P\left(b^{-1}x_2,a^{-1}x_4\right)} \\Â =  \mat{-b^{-1}x_2}{-ax_1 - ax_2P(x_2,x_4)}{a^{-1}x_4}{bx_3+bx_4P(x_2,x_4)} = \tilde f \circ g .
\end{multline*}
Conjugating this equality by the involution $\sigma$ we get
$\tilde g \circ \tilde{f}^{-1} = f^{-1} \circ \tilde g$, hence $f \circ \tilde g = \tilde g \circ \tilde f$.
Finally  $f$ commutes with  $\tilde g \circ g$ because 
$$(f \circ \tilde{g}) \circ g = (\tilde{g} \circ \tilde{f}) \circ g = \tilde{g} \circ (\tilde{f} \circ g) = \tilde{g} \circ (g  \circ f).$$ 
Then $\tilde{g} \circ g$ is hyperbolic by Lemma \ref{lem:hyperbolic elements} and therefore $f$ is hyperelliptic.
\end{proof}

\begin{lemma} \label{lem:elementary hyperelliptic}
If $a, b$ are roots of unity of the same order, then for any $P(x_1,x_3) \in \C[x_1,x_3]$ the elementary automorphism
$$f = \mat{a^{-1}x_1}{bx_2+bx_1P(x_1,x_3)}{b^{-1}x_3}{ax_4+ax_3P(x_1,x_3)}$$
is hyperelliptic.
\end{lemma}

\begin{proof}
There exist $m, n \ge 2$ such that $a^m = b$ and $b^n = a$.
We will use the observation that in $\Aut(\A^2_\C)$, with $\A^2_\C = \text{Spec}\, \C[x_1,x_3]$, the automorphisms $(x_3, x_1 + x_3^m) \circ (x_3, x_1 + x_3^n)$ and $(a^{-1}x_1, b^{-1}x_3)$ commute. 

By Lemma \ref{lem:hyperbolic elements}, the following automorphisms are hyperbolic, because their projections on $\T_H$ are hyperbolic:
$$g_1=\mat{x_3}{-x_4}{x_1+x_3^m}{-x_2-x_4x_3^{m-1}}, \quad
g_2=\mat{x_3}{-x_4}{x_1+x_3^n}{-x_2-x_4x_3^{n-1}} \; \text{and}  \; g = g_1 \circ g_2.$$
The projection $\pi_H(g)$ is a hyperbolic isometry, $\pi_H(f)$ is elliptic, and $\pi_H(g)$ and $\pi_H(f)$ commute.
By Lemma \ref{lem: criterion of hyperellipticity}, $\Min \pi_H(f)$ is unbounded.
We conclude by Lemma \ref{lem:hyperelliptic on projections}.  
\end{proof}

\begin{remark}
We believe that any hyperelliptic automorphism in $\TSL$ is conjugate to an automorphism of the form given in Lemmas \ref{lem:resonnant hyperelliptic} or \ref{lem:elementary hyperelliptic}.
However we were not able to get an easy proof of that fact.
\end{remark}

\subsubsection{An example of parabolic subgroup}

We give an example of parabolic subgroup in $\TSL$, where most elements have infinite order.
This is in contrast with the situation of $\Aut(\C^2)$, where a parabolic subgroup is always a torsion group (see \cite[Proposition 3.12]{Ltits}). 
Let
$$ H_n = \left\lbrace \mat{ax_1}{b^{-1}x_2}{bx_3}{a^{-1}x_4};\; a, b \in \C^*, (ab)^{2^n} = 1 \right\rbrace .$$
As in the proof of Lemma \ref{lem:resonnant hyperelliptic}, we set
\begin{multline*} 
\sigma = \mat{-x_3}{x_4}{-x_1}{x_2}, \;
g_n = \mat{-x_2}{-x_1-x_2P_n(x_2,x_4)}{x_4}{x_3+x_4P_n(x_2,x_4)},\\
\text{ and }\;
\tilde{g}_n = \sigma g_n \sigma = \mat{-x_2}{-x_1+x_2P_n(x_4,x_2)}{x_4}{x_3-x_4P_n(x_4,x_2)}
\end{multline*}
where $P_n(x,y) = (xy)^{2^n-1}$. 
Then we observe that for $j < k$, any element $h \in H_j$ commutes with $\tilde g_k \circ g_k$. 
On the other hand for any $k \ge 1$ and any $h \in  \left(\bigcup_{n\ge 0} H_n \right) \smallsetminus H_{k-1}$, ${\tilde g_k}^{-1} h \tilde g_k$ is a non linear elementary automorphism. 
We set
$$
\phi_n = \tilde g_n \circ g_n \circ\dots\circ \tilde g_1 \circ g_1, \;
\Delta_n = \phi_n^{-1} H_n \phi_n \; \text{ and }\; 
\Delta = \cup_{n \ge 0} \Delta_n.
$$
Then $\Delta$ is a parabolic subgroup of $\TSL$.
Indeed by Lemma \ref{lem:elliptic on projections} it is sufficient to prove that the isometry group  $\pi_V(\Delta)$ induced by $\Delta$ on the vertical tree $\T_V$ is parabolic. 
This is the case, since for each $n \ge 1$, $\phi_n^{-1} \cdot \pi_V [\id]$ is a fixed vertex for $\pi_V(\Delta_n)$, but not for $\pi_V(\Delta_{n+1})$, and $d(\pi_V [id],\phi_n^{-1} \cdot \pi_V[\id]) = 4n$ goes to infinity with $n$. 

\subsection{Further comments}

\subsubsection{Tame group of the affine space} \label{sec:tameKn}
 
In Section \ref{sec:general simplicial complex} we defined a simplicial complex associated with the tame group of $K^n$. 
We now make a few comments on this construction.
We make the convention to call \textbf{standard simplex} the simplex associated with the vertices $[x_1]$, $[x_1,x_2], \dots , [x_1, \dots, x_n]$.  

First observe that we could make the same formal construction as in \S \ref{sec:general simplicial complex} using the whole group $\Aut(K^n)$. 
But then it is not clear anymore that we obtain a connected complex.
More precisely, recall that if $X$ is a simplicial complex of dimension $n$, we say that $X$ is \textbf{gallery connected} if given any simplexes $S$, $S'$ of maximal dimension in $X$, there exists a sequence of simplexes of maximal dimension $S_1 = S, \dots, S_n = S'$ such that for any $i = 1, \dots, n-1$, the intersection $S_i \cap S_{i+1}$ is a face of dimension $n-1$ (see \cite[p. 55]{BS}).
Then the gallery connected component of the standard simplex of the complex associated with $\Aut(K^n)$ is precisely the complex associated to $\Tame(K^n)$.
It is probable that the whole complex is not connected, but it seems to be a difficult question.

We now focus on the case $K = \C$, $n = 3$.
In the same vein as the above discussion, observe that the Nagata automorphism 
$$N = (x_1 + 2x_2(x_2^2-x_1x_3) + x_3 (x_2^2-x_1x_3)^2,\, x_2 + x_3(x_2^2-x_1x_3),\, x_3)$$ 
defines a simplex  that shares the vertex $[x_3]$ with the standard simplex, but since $N$ is not tame these two simplexes are not in the same gallery connected component.
The question of the connectedness of the whole complex associated with $\Aut(\C^3)$ is equivalent to the question whether $\Aut(\C^3)$ is generated by the affine group and automorphisms preserving the variable $x_3$.

We denote by  $\Comp'$ the 2-dimensional simplicial complex associated with $\Tame(\C^3)$.
The standard simplex has vertices $[x_1]$, $[x_1,x_2]$ and $[\id]$, and the stabilizers of these vertices are respectively (here we use the notation of \S \ref{sec:general simplicial complex}):
\begin{align*}
\Stab [x_1] &= \left\lbrace (a x_1 + b, f, g);\; (f,g) \in \Tame_{\C[x_1]}(\Spec \C[x_2,x_3]) \right\rbrace \\
\Stab [x_1,x_2] &= \left\lbrace (ax_1 + b x_2 + c, a'x_1 + b'x_2 + c', dx_3 + P(x_1,x_2)) \right\rbrace \\ 
\Stab [x_1,x_2,x_3] &= A_3.
\end{align*}
By construction the group $\Tame(\C^3)$ acts on the complex $\Comp'$ with fundamental domain the standard simplex.
To say that $\Tame(\C^3)$ is the amalgamated product of the three stabilizers above along their pairwise intersection is equivalent to the simple connectedness of the complex. 
This is precisely the content of the main theorem of \cite{wright}, where the subgroups are denoted by $\tilde H_1$, $H_2$ and $H_3$. 
Observe that the proof of Wright relies on the understanding of the relations in the tame group and so ultimately on the Shestakov-Umirbaev theory: This is similar to our proof of Proposition   
\ref{pro:1connected}, which relies on an adaptation of the Shestakov-Umirbaev theory to the case of a quadric 3-fold.

Note that the naive thought according to which $\TSL$ would be the amalgamated product of the four types of elementary groups is false. 
Indeed, if $P,Q$ are non-constant polynomials of $\C [x_1]$, the two following elements belong to different factors and they commute (this is similar to a remark made by J. Alev  a long time ago about $\Tame(\C^3)$, see \cite{Alev}):
\[  \mat{x_1}{x_2 + x_1 P(x_1)}{x_3}{x_4 + x_3 P(x_1)}, \; \mat{x_1}{x_2 }{x_3+x_1 Q(x_1)}{x_4 + x_2 Q(x_1)}. \]
On the other hand, it follows from our study in Section \ref{sec:structures} (see also Figure \ref{fig:nesting}) that $\STSL$ is the amalgamated product of the four stabilizers of each vertex of the standard square along their pairwise intersections: In view of the result of Wright, this is another evidence that the groups $\Tame(\C^3)$ and $\TSL$ are qualitatively quite similar.

As mentioned at the end of \cite{wright}, there are basic open questions about the complex $\Comp'$: The contractibility of $\Comp'$ is not clear, or even whether it is unbounded or not.
In view of what we proved about the complex $\Comp$ associated with $\TSL$, a natural question would be to ask if $\Comp'$ is $\CAT(0)$.
We believe that this is not the case (for any choice of local Euclidean structure on $\Comp'$, that is, for any choice of angles for the standard triangle), but this would have to be carefully checked by computing the link structure for vertices of $\Comp'$. 
On the other hand, it seems possible that the complex $\Comp'$ is hyperbolic.
Of course this last question is relevant  only if $\Comp'$ is unbounded, but we believe this to be true.

  
\subsubsection{The restriction morphism} \label{sec:rho}

Recall that we have natural morphisms of restriction:
\[ \pi \colon \Tame _q (\C^4) \to \Tame(\SL_2)  \quad \text{and} \quad  \rho \colon \Aut _q (\C^4) \to \Aut(\SL_2).\]

We have proved in Proposition \ref{pro:TQ=TSL} that $\pi$ is an isomorphism.
On the other hand, we have
$$\rho \left( \mat{x_1}{x_2 + x_1(q-1)}{x_3}{x_4 + x_3(q-1)} \right) = \id_{\SL_2},$$
so that $\rho$ is not injective.

If follows from the next remark that the automorphism 
$$\mat{x_1}{x_2 + x_1(q-1)}{x_3}{x_4 + x_3(q-1)}$$
of $\Aut _q (\C^4)$ does not belong to $\Tame _q (\C^4)$.

\begin{remark}
Any automorphism $f= \smat{x_1}{f_2}{x_3}{f_4}$ of $\Tame_q(\C^4)$ is of the form 
$$f= \mat{x_1}{x_2+x_1 P(x_1,x_3)}{x_3}{x_4+ x_3 P (x_1,x_3)}.$$
This follows from Theorem \ref{thm:main}, that is, from the existence of elementary reduction. 
Indeed, if a non linear automorphism $f= \smat{x_1}{f_2}{x_3}{f_4}$ belongs to $\Tame_q(\C^4)$, by Lemma \ref{lem:degree of each component drops} it necessarily admits an elementary reduction of the form 
$$\mat{x_1}{f_2+x_1 P_1(x_1,x_3)}{x_3}{f_4 +x_3 P_1(x_1,x_3)},$$
which in turn admits an elementary reduction of the same form. We can  continue until we obtain a linear automorphism and this proves the result.
\end{remark}

Note that any automorphism $f=\smat{f_1}{f_2}{f_3}{f_4}$ in $\Aut_q(\C^4)$ such that $f_1=x_1$ and $f_3=x_3$ is necessarily of the form  $f=\smat{x_1}{x_2+x_1P}{x_3}{x_4 + x_3 P}$, where $P \in \C[x_1,x_3,q]$. Indeed, since $x_1f_4-x_3f_2 = q$,  there exists some polynomial $P$ in $\C[x_1,x_2,x_3,x_4]$ such that $f_2 =x_2 + x_1P$ and $f_4= x_4 + x_3 P$. 
The Jacobian condition $\det (\frac{\partial f_i}{\partial x_j})_{i,j}=1$ is equivalent to $\delta P= 0$, where $\delta$ is the locally nilpotent derivation of $\C[x_1,x_2,x_3,x_4]$ given by $\delta=x_1 \partial_{x_2} + x_3 \partial_{x_4}$. 
One could easily check that $\Ker  \delta = \C[x_1,x_3,q]$. Conversely, for any element $P$ of $\C[x_1,x_3,q]$, it is clear that $f=\smat{x_1}{x_2+x_1P}{x_3}{x_4 + x_3 P}$ is an element of $\Aut_q(\C^4)$ whose inverse is  $f^{-1}=\smat{x_1}{x_2-x_1P}{x_3}{x_4-x_3P}$.

If we take $P(x_1,x_3,q)=q$, we obtain the famous Anick's automorphism. Since $f_3$ actually depends on $x_4$, Corollary \ref{cor:linktype1} above directly implies that this automorphism does not belong to $\TSL$.
However in restriction to $\SL_2$ the Anick's automorphism coincides with the linear (hence tame) automorphism $\smat{x_1}{x_2+x_1}{x_3}{x_4 + x_3}$.
On the other hand there exist  automorphisms in $\Aut_q(\C^4)$ whose restriction to the quadric $q=1$ does not coincide with the restriction of any automorphism in $\TSL$: see \cite[\S 5]{LV} where it is proved that the following automorphism is a concrete example:
$$\mat{x_1}{x_2}{x_3}{x_4} \mapsto \mat{x_1 -x_2(x_1+x_4)}{x_2}{x_3 + (x_1-x_4)(x_1+x_4) -x_2(x_1+x_4)^2}{x_4 + x_2(x_1+x_4)}.$$

Observe that for the Anick's automorphism the degrees of the components are not the same when considered as elements of $\C[x_1,x_2,x_3,x_4]$ or as elements of $\C[\SL_2]$.
On the other hand it seems possible that in the case of an automorphism $f = \smat{f_1}{f_2}{f_3}{f_4} \in \TSL$, equalities $\deg f_i = \deg_{\C^4} f_i$ always hold for each component $f_i$.
This is an interesting question, that we have not been able to solve. 
Let us formulate it precisely.
For any element $\overline{p} \in {\mathcal O} (\SL_2) := \C [x_1,x_2,x_3,x_4] / \langle q-1 \rangle$, set
\[ \deg \overline{p} = \min \{ \deg r, \; r \in \overline{p} \}.\]
Note that $\deg \overline{p} = \deg p$ if and only if $p = 0$ or $q$ does not divide the leading part $\hom{p}$ of $p$ (see \cite[\S 2.5]{LV}).

\begin{question} \label{question of degree}
If $p$ is the component of an element of $\TQ$, do we have $\deg \overline{p} = \deg p$?
\end{question}

Note that a positive answer to Question \ref{question of degree} would immediately imply Proposition \ref{pro:TQ=TSL}. Indeed, if $f=\smat{f_1}{f_2}{f_3}{f_4} \in \Ker \pi$, there exist polynomials $g_i$ such that $f_i = x_i + (q-1) g_i$. But if $\deg f_i = \deg \overline{f_i}$, we get $g_i =0$, so that $f_i =x_i$ and $f = \id$.

Another natural but probably difficult question about the morphism $\rho$ is the following: 

\begin{question} \label{question of surjectivity}
Is the map $\rho \colon \Aut _q (\C^4) \to \Aut(\SL_2)$ surjective? 
\end{question}

\subsubsection{Infinite transitivity}

As a final remark we check that $\STSL$ acts infinitely transitively on the quadric $\SL_2$, as a consequence of the results in \cite{AFKKZ}.

Consider the locally nilpotent derivation $\partial = x_1 \partial_{x_2} + x_3 \partial_{x_4}$ of the coordinate ring $\C[\SL_2] =\C [x_1,x_2,x_3,x_4] /  (q-1)$. We have $\Ker \partial = \C [x_1,x_3]$ and for any element $P$ of $\C [x_1,x_3]$, we have
\[ \exp (P \partial) = \mat{x_1}{x_2 +x_1P}{x_3}{x_4+x_2P} \in \STSL.\]
Therefore, the set ${\mathcal N}$  of locally nilpotent derivations on $\SL_2$ that are conjugate in $\STSL$ to the above derivations is saturated in the sense of \cite[Definition 2.1]{AFKKZ}. Furthermore, one could easily show that $\STSL$ is generated by ${\mathcal N}$. Indeed, it is clear that any elementary automorphism is the exponential of an element of ${\mathcal N}$. We leave as an exercise for the reader to check that $\SO_4$ is included into the group generated by ${\mathcal N}$. Finally, since $\STSL$ contains the group $\SL_2$, it acts transitively on $\SL_2$, and we conclude by  \cite[Theorem 2.2]{AFKKZ}.

\section*{Annex}
\addtocounter{section}{1} 
\renewcommand{\thesection}{A}
\setcounter{subsection}{0}
\setcounter{theorem}{0}

In this annex we prove that on both groups  $\TSL$ and $\TQ$ there exists a good notion of elementary reduction, in the spirit of Shestakov-Umirbaev and Kuroda theories.
In the case of $\TSL$ this was done in \cite{LV}.
The purpose of this annex is twofold: We propose a simplified version of the argument in the case of $\TSL$, and we establish a similar result for the group $\TQ$.

\subsection{Main result}
 
In the sequel $G$ denotes either the group $\TQ$ or the group $\TSL$, since most of the statements hold without any change in both settings.  

Recall that we define the \textbf{degree} of a monomial of $\C [x_1,x_2,x_3,x_4]$ by
\[\deg x_1^ix_2^jx_3^kx_4^l = (i,j,k,l)
\left(
  \begin{matrix}
    2  & 1 & 1 & 0\\ 
    1 & 2 & 0 & 1 \\ 
    1 & 0 & 2 & 1 \\ 
    0 & 1 & 1 & 2
  \end{matrix}
\right) = (2i+j+k,i+ 2j +l,i+2k+l,j+k+2l) \in \N^4.\]
Then, by using the graded lexicographic order on $\N^4$, we define the degree of any nonzero element of $\C [x_1,x_2,x_3,x_4]$: We first compare the sums of the coefficients and, in case of a tie, apply the lexicographic order. For example, we have
\[ \deg (x_1+x_2+x_3+x_4) =(2,1,1,0), \quad \deg (x_1x_2 + x_3^2) = (3,3,1,1), \]
\[\deg x_1x_4=\deg x_2x_3 = \deg q =(2,2,2,2).\] 
By convention, we set $\deg 0 = -\infty$, with $-\infty$ smaller than any element of $\N^4$.
The leading part  of a polynomial
\[p=\sum_{i, j,  k, l}p_{i,j,k,l} \; x_1^ix_2^jx_3^kx_4^l \; \in\C[x_1,x_2,x_3,x_4]\]
is denoted $\hom{p}$. Hence, we have
\[ \hom{p}=\sum_{\deg x_1^ix_2^jx_3^kx_4^l \; = \; \deg p }\hspace{-6mm} p_{i,j,k,l} \; x_1^ix_2^jx_3^kx_4^l.\]
Remark that $\hom{p}$ is not in general a monomial. 
For instance, we have $\hom{q} = q$. 
We define the \textbf{degree of an automorphism} $f= \smat{f_1}{f_2}{f_3}{f_4}$ to be 
\[\deg f = \max_i \deg f_i \in \N^4.\]
We have similar definitions in the case of $\TSL$, where the degree on $\C [\SL_2]$, also noted $\deg$, is defined by considering minimum over all representatives.
 
An \textbf{elementary automorphism} is an element of $G$ of the form
\[e = u \mat{x_1}{x_2 + x_1P(x_1,x_3)}{x_3}{x_4+x_3P(x_1,x_3)} u^{-1}\]
where $u \in \V_4$, $P \in \C[x_1,x_3]$.
We say that $f \in G$ admits an \textbf{elementary reduction} if there exists an elementary automorphism $e$ such that $\deg e\circ f < \deg f$.
We denote by $\a$ the set of elements of $G$ that admit a sequence of elementary reductions to an element of $\OO_4$. The main result of this annex is then:

\begin{theorem} \label{thm:main}
Any non-linear element of $G$ admits an elementary reduction, that is, we have the equality $G = \a$.
\end{theorem}

\subsection{Lower bounds}\label{sec:para}

The following result is a close analogue of \cite[Lemma 3.3(i)]{Ku:main} and is taken from \cite[\S 3]{LV}.

\begin{minoration}\label{mino:TSL}
Let $f_1,f_2 \in\C[\SL_2]$ be algebraically independent and let $R(f_1,f_2)$ be an element of $\C[f_1,f_2]$. Assume that $R(f_1,f_2) \not\in\C[f_2]$ and $\hom{f_1}\not\in\C[\hom{f_2}]$. Then
\[  \deg (f_2R(f_1,f_2))>\deg f_1.\]
\end{minoration}

In this section we establish the following analogous lower bound in the context of $G = \TQ$. 

\begin{minoration}\label{mino:TQ}
Let $(f_1,f_2) \in\C[x_1,x_2,x_3,x_4]^2$ be part of an automorphism of $\C^4$ and let $R(f_1,f_2)$ be an element of $\C[f_1,f_2]$. Assume that $R(f_1,f_2) \not\in\C[f_2]$ and $\hom{f_1}\not\in\C[\hom{f_2}]$. Then
\[  \deg (f_2R(f_1,f_2))>\deg f_1.\]
\end{minoration}

We say that $(f_1, f_2) \in\C[x_1,x_2,x_3,x_4]^2$ is \textbf{part of an automorphism} of $\C^4$, if there exists $(f_3,f_4) \in\C[x_1,x_2,x_3,x_4]^2$ such that $(f_1,f_2,f_3,f_4)$ is an automorphism of $\C^4$.

We follow the proof of Lower bound \ref{mino:TSL} given in \cite[\S 3]{LV}. The only non-trivial modification lies in Lemma \ref{lem: not all zero} below, but for the convenience of the reader we give the full detail of the arguments.

\subsubsection{Generic degree}

Given $f_1, f_2\in\C[x_1,x_2,x_3,x_4] \smallsetminus \{0\}$, consider 
$$R = \sum R_{i,j} X_1^iX_2^j\in\C[X_1,X_2]$$
 a non-zero polynomial in two variables. 
Generically (on the coefficients $R_{i,j}$ of $R$), $\deg R(f_1,f_2)$ coincides with $\ged R$ where $\ged$ (standing for \textbf{generic degree}) is the weighted degree on  $\C[X_1,X_2]$ defined by 
\[\ged X_i =\deg f_i\in\N^4,\]
again with the graded lexicographic order. 
Namely we have
$$ R(f_1,f_2)  =  \Rgen(f_1,f_2)+LDT(f_1,f_2)$$
where
$$  \Rgen(f_1,f_2) = \sum_{\ged X_1^iX_2^j \: = \: \ged R}R_{i,j}f_1^if_2^j$$
is the leading part of $R$ with respect to the generic degree and $LDT$ represents the Lower (generic) Degree Terms. 
One has 
$$\deg LDT(f_1,f_2)<\deg \Rgen(f_1,f_2)=\ged R =\deg R(f_1,f_2)$$ 
\textit{unless} $\Rgen(\hom{f_1},\hom{f_2})=0$, in which case the degree falls: $\deg R(f_1,f_2)<\ged R$. 

Let us focus on the condition $\Rgen(\hom{f_1},\hom{f_2})=0$. Of course this can happen only if $\hom{f_1}$ and $\hom{f_2}$ are algebraically dependent. Remark that the ideal 
$$I=\{S\in\C[X_1,X_2];\; S(\hom{f_1},\hom{f_2})=0\}$$
must then be principal, prime and generated by a $\ged$-homogeneous polynomial. 
The only possibility is that $I=(X_1^{s_1}-\lambda X_2^{s_2})$ where $\lambda\in\C^*$, $s_1\deg f_1=s_2\deg f_2$ and $s_1,s_2$ are coprime. To sum up, in the case where $\hom{f_1}$ and $\hom{f_2}$ are algebraically dependent one has
\begin{equation}
\deg R(f_1,f_2)<\ged R \;\Leftrightarrow\;  \Rgen(\hom{f_1},\hom{f_2})=0  \;\Leftrightarrow\; \Rgen\in (H) \, \label{rel}
\end{equation}
where $H=X_1^{s_1}-\lambda X_2^{s_2}$.

\subsubsection{Pseudo-Jacobians}

If $f_1,f_2,f_3,f_4$ are polynomials in $\C[x_1,x_2,x_3,x_4]$, we denote by  $\Jac(f_1,f_2,f_3,f_4)$ the Jacobian determinant, i.e. the determinant of the Jacobian $4 \times 4$- matrix $(\frac{\partial f_i}{\partial x_j})$.
Then we define the \textbf{pseudo-Jacobian} of $f_1,f_2,f_3$ by the formula
\[\jj(f_1,f_2,f_3):=\Jac(q,f_1,f_2,f_3).\]

\begin{lemma} \label{lem:pseudojac}
Assume $f_1, f_2, f_3 \in \C[x_1,x_2,x_3,x_4]$. Then
\[\deg\jj(f_1,f_2,f_3) \leq  \deg f_1+\deg f_2+\deg f_3-(2,2,2,2).\]
\end{lemma}

\begin{proof}
An easy computation shows the following inequality:

\[\deg \Jac(f_1,f_2,f_3,f_4)  \leq \sum_i \deg f_i- \sum_{i}\deg x_i=
\sum_i \deg f_i -(4,4,4,4).\]

Recalling the definitions of $\jj$ and $\deg$ we obtain:
\begin{multline*}
\deg\jj(f_1,f_2,f_3) = \deg \Jac (q,f_1,f_2,f_3) \\ 
\leq \deg q+ \sum_i \deg f_i -(4,4,4,4) = \sum_i \deg f_i -(2,2,2,2).\qedhere
\end{multline*}
\end{proof}

We shall essentially use those pseudo-Jacobians with $f_1=x_1$, $x_2$, $x_3$ or $x_4$. Therefore we introduce the notation 
$\jj_k(\cdot,\cdot):=\jj(x_k,\cdot,\cdot)$ for all $k=1,2,3,4$. The inequality from Lemma \ref{lem:pseudojac} gives
\[\deg\jj_k(f_1,f_2) \leq \deg f_1+\deg f_2+\deg x_k-(2,2,2,2)\]
from which we deduce
\begin{equation} \label{ineqjb}
\deg\jj_k(f_1,f_2) <  \deg f_1+\deg f_2,\;\forall \, k=1,2,3,4.
\end{equation}

We shall also need the following observation. 

\begin{lemma}\label{lem: not all zero}
If $(f_1,f_2)$ is part of an automorphism of $\C^4$, then the elements $\jj_k(f_1,f_2)$, $k =1, \dots, 4$, are not simultaneously zero, i.e.
$\max_{k}\deg\jj_k(f_1,f_2)\neq -\infty$ or, equivalently, \[\max_{k}\deg\jj_k(f_1,f_2)\in\N^4.\]
\end{lemma}

\begin{proof}
Assume that  $\jj(x_k,f_1,f_2)=0$ for each $k$.
This means that the elements $q,f_1,f_2$ are algebraically dependent. 
But, since $(f_1,f_2)$ is part of an automorphism of $\C^4$, the ring $\C[f_1,f_2]$ is algebraically closed in $\C[x_1,x_2,x_3,x_4]$ (indeed, there exists an automorphism of the algebra $\C[x_1,x_2,x_3,x_4]$ sending $\C[f_1,f_2]$ to $\C [x_1,x_2]$). 
Therefore, there exists a polynomial $R$ such that $q = R(f_1,f_2)$.
We now prove that this is impossible. 
Indeed, we may assume that $f_1$ and $f_2$ do not have constant terms. Let $l_1$ and $l_2$ be their linear parts. Write $R=\sum_{i,j}R_{i,j}X^iY^j$. 
It is clear that $R_{0,0} = 0$ (look at the constant term) and that $R_{1,0} =R_{0,1} =0$ (look at the linear part and use the fact that $l_1$, $l_2$ are linearly independent). 
Therefore, looking at the quadratic part, we get
\[ q=R_{2,0}\, l_1^2+ R_{1,1} \, l_1l_2 + R_{0,2}\, l_2^2.\]
We get a contradiction since the rank of the quadratic form $q$ is $4$ and the rank of the quadratic form on the right is at most $2$. \end{proof}

\subsubsection{The parachute}

In this paragraph $(f_1, f_2) \in \C[x_1,x_2,x_3,x_4]^2$ is part of an automorphism of $\C^4$, and we set $d_i:= \deg f_i \in \N^4$. We define the \textbf{parachute} of $f_1, f_2$ to be  
$$ \nabla(f_1,f_2) = d_1 + d_2 - \max_{k} \deg\jj_k(f_1,f_2).$$ 
By Lemma \ref{lem: not all zero}, we get $ \nabla(f_1,f_2)\leq d_1+d_2 $.

\begin{lemma}\label{lem:ineq1}
Assume $\deg\frac{\partial^n R}{\partial X_2^n}(f_1,f_2)$ coincides with the generic degree $\ged \frac{\partial^n R}{\partial X_2^n}$. 
Then 
$$d_2 \cdot\deg_{X_2}R-n\nabla(f_1,f_2)  < \deg R(f_1,f_2).$$
\end{lemma}

\begin{proof}
As already remarked $\Jac$, $\jj$ and now $\jj_k$ as well are $\C$-derivations in each of their entries. We may then apply the chain rule on $\jj_k(f_1,\cdot)$ evaluated in $R(f_1,f_2)$:
$$ \frac{\partial R}{\partial X_2}(f_1,f_2)\jj_k(f_1,f_2)
=  \jj_k(f_1,R(f_1,f_2)).$$
Now taking the degree and applying inequality (\ref{ineqjb}) (with $R(f_1,f_2)$ instead of $f_2$), we obtain
$$\deg\frac{\partial R}{\partial X_2}(f_1,f_2)+\deg\jj_k(f_1,f_2)  <  d_1+\deg R(f_1,f_2).$$
We deduce
$$\deg\frac{\partial R}{\partial X_2}(f_1,f_2)+d_2-(\underbrace{d_1+d_2-\max_{k}\deg\jj_k(f_1,f_2)}_{=\nabla(f_1,f_2)})  < \deg R(f_1,f_2).$$
By induction, for any $n \ge 1$ we have
\begin{align*}
\deg\frac{\partial^n R}{\partial X_2^n}(f_1,f_2)+nd_2-n\nabla(f_1,f_2) & < \deg R(f_1,f_2) \; .
\end{align*}
Now if the integer $n$ is as given in the statement one gets:
$$
\deg\frac{\partial^n R}{\partial X_2^n}(f_1,f_2)=\ged \frac{\partial^n R}{\partial X_2^n}\geq d_2\cdot\deg_{X_2}\frac{\partial^n R}{\partial X_2^n} =d_2\cdot(\deg_{X_2}R-n)=d_2\cdot\deg_{X_2}R-d_2n
$$ 
which, together with the previous inequality, gives the result.
\end{proof}

\begin{lemma}\label{lem:goodn}
Let $H$ be the generating relation between $\hom{f_1}$ and $\hom{f_2}$ as in the equivalence (\ref{rel}) and let $n\in\N$ be such that $\Rgen\in(H^n)\setminus(H^{n+1})$. Then $n$ fulfills the assumption of Lemma \ref{lem:ineq1}, i.e.
\[\deg\frac{\partial^n R}{\partial X_2^n}(f_1,f_2)=\ged \frac{\partial^n R}{\partial X_2^n}\; .\]
\end{lemma}

\begin{proof}
It suffices to remark that ${\left( \frac{\partial R}{\partial X_2} \right) }_\text{\tiny \hspace{-1mm} gen}=\frac{\partial \Rgen}{\partial X_2}$ and that $\Rgen\in(H^n)\smallsetminus(H^{n+1})$ implies $\frac{\partial \Rgen}{\partial X_2}\in(H^{n-1})\smallsetminus(H^{n})$. One concludes by induction.
\end{proof}

Remark that, by definition of $n$ in Lemma \ref{lem:goodn} above, we have:
\[\deg_{X_2} R\geq\deg_{X_2}\Rgen\geq ns_2.\]
Together with Lemma \ref{lem:ineq1} and recalling that $s_1d_1=s_2d_2$, this gives:
\begin{equation}\label{last}
d_1ns_1 -n\nabla(f_1,f_2)  < \deg R(f_1,f_2). 
\end{equation}\\

\subsubsection{Proof of Lower bound \ref{mino:TQ}}~\par

Let $n$ be as in Lemma \ref{lem:goodn}. 
If $n=0$, then $\deg R(f_1,f_2)=\ged R\geq \deg f_1$ by the assumption  $R(f_1,f_2) \not\in\C[f_2]$ and then $\deg (f_2R(f_1,f_2))\geq\deg f_2+\deg f_1>\deg f_1$ as wanted. 

If $n\geq 1$ then, by (\ref{last}),
$$ d_1s_1-\nabla(f_1,f_2) < \deg R(f_1,f_2) $$
and, since $\nabla(f_1,f_2) \leq d_1 + d_2$, 
$$  d_1s_1-d_1-d_2 < \deg R(f_1,f_2).$$
We obtain
$$d_1(s_1-1) <  \deg R(f_1,f_2)+d_2=\deg(f_2R(f_1,f_2)).$$
The assumption $\hom{f_1}\not\in\C[\hom{f_2}]$ forbids $s_1$ to be equal to one, hence we get the desired lower bound.\qed


\subsection{Proof of the main result} \label{sec:proof}
In this section, we prove Theorem \ref{thm:main}. We need the two following easy lemmas.

\begin{lemma} \label{lem:degree of each component drops}
Let $f = \smat{f_1}{f_2}{f_3}{f_4} \in G$.
If $e \in \El$ and $e\circ f =  \smat{f'_1}{f_2}{f'_3}{f_4}$, then
$$\deg e \circ f  \sphericalangle \deg f \Longleftrightarrow \deg f'_1 \sphericalangle \deg f_1  \Longleftrightarrow   \deg f'_3 \sphericalangle \deg f_3$$
for any relation $\sphericalangle$ among $<$, $>$, $\le$, $\ge$ and $=$.
\end{lemma}

\begin{proof}
We have 
$$e=\mat{x_1 +x_2P(x_2,x_4)}{x_2}{x_3+x_4 P(x_2,x_4)}{x_4}$$
where $P$ is non-constant. 
We first prove  the equivalence for $\sphericalangle$ equal to $<$.
One has $f_1f_4-f_2f_3=q$ and the polynomials $f_i$ are not linear hence the leading parts must cancel one another: 
$\hom{f_1}\hom{f_4}-\hom{f_2}\hom{f_3}=0$. It follows: $\deg f_1+\deg f_4=\deg f_2+\deg f_3$.
Similarly
$\deg f_1'+\deg f_4=\deg f_2+\deg f_3'$.
So we obtain
$$\deg f_1 - \deg f_1' = \deg f_3 - \deg f_3'.$$

Assume $\deg e \circ f < \deg f$.
Thus $\deg f = \max(\deg f_1, \deg f_3)$, hence 
$$\max(\deg f_1', \deg f_3') \le \deg e\circ f < \deg f = \max(\deg f_1, \deg f_3),$$ which implies $\deg f_1' < \deg f_1$ and $\deg f_3' < \deg f_3$.

Conversely if one of the inequalities $\deg f_1' < \deg f_1$ or $\deg f_3' < \deg f_3$ is satisfied then both are satisfied, and this implies $\deg f_2 < \deg  f_2 P(f_2,f_4) = \deg f_1$ and similarly $\deg f_4 < \deg f_3$. 
Hence $\deg e \circ f < \deg f$.

We have proved the equivalence for $\sphericalangle$ equal to $<$. 
Since $f = e^{-1} \circ (e \circ f)$, we also obtain the equivalence for $\sphericalangle$ equal to $>$.
The equivalences for the three remaining symbols $=, \le, \ge$ follow.
\end{proof}

\begin{lemma} \label{lem: easy decomposition of elements of G}
Any element of $G$ can be written under the form
\[f= e_{\ell}\circ e_{\ell -1} \circ \cdots \circ e_1 \circ a,\]
where the elements $e_i$ are elementary and $a$ belongs to $\OO_4$.
\end{lemma}

\begin{proof}
Observe that any element of $\SO_4$ is a composition of (linear) elementary automorphisms. 
Since both $\STSL$ and $\STQ$ are generated by $\SO_4$ and the elementary automorphisms, it follows that any element of these two groups may be written as
\[f=e_{\ell}\circ e_{\ell -1} \circ \cdots \circ e_1,\]
where the automorphisms $e_i$ are elementary. 
The result follows.
\end{proof}

Since the set $\a$ obviously contains $\OO_4$, the following proposition joined to Lemma \ref{lem: easy decomposition of elements of G} directly implies  Theorem \ref{thm:main}.
 
\begin{proposition}\label{pro:main}
If $f \in \a$ and $e$ is an elementary automorphism, then $e\circ f \in \a$. 
\end{proposition}

In the rest of this section we prove the proposition by induction on $d := \deg f \in \N^4$.

If $d= (2,1,1,0)$, that is to say if $f \in \OO_4$, then either $\deg e \circ f = d$ and again $e \circ f \in \OO_4 \subset \a$, or $\deg e\circ f > d$ and $e \circ f$ admits an obvious elementary reduction to an element of $\OO_4$, by composing by $e^{-1}$.

Now we assume $d > (2,1,1,0)$, we set $\a_{<\, d}:= \{g \in\a;\; \deg g<d \}$ and we assume the following:

\begin{IH}\nonumber\label{IH}
If $g\in \a_{< \,d}$ and if $e$ is elementary, then $e \circ g \in \a$.
\end{IH}

We pick $f \in \a$ such that $\deg f = d$, an elementary automorphism $e$, and we must prove that $e\circ f \in \a$.

If $\deg e\circ f > \deg f$, this is clear, so  we now assume that $\deg e\circ f \le \deg f$.

Since $f \in \a$, there exists an elementary automorphism $e'$ such that $\deg e' \circ f <d$ and $e' \circ f \in \a$, i.e. $e' \circ f \in \a_{< \, d}$.

\begin{caselist} \label{list:3cases}

Up to conjugacy by an element of $\V_4$, we may assume that:
\[e' = \mat{x_1+x_2P(x_2,x_4)}{x_2}{x_3+x_4P(x_2,x_4)}{x_4}\]
and that one of the three following assertions is satisfied:
\begin{enumerate}
\item $e \in \El$, i.e. $e=\mat{x_1+x_2Q(x_2,x_4)}{x_2}{x_3+x_4Q(x_2,x_4)}{x_4}$ for some polynomial $Q$;
\item $e \in \Er$, i.e. $e=\mat{x_1}{x_2+x_1Q(x_1,x_3)}{x_3}{x_4+x_3Q(x_1,x_3)}$ for some polynomial $Q$;
\item $e \in \Et$, i.e. $e=\mat{x_1+x_3Q(x_3,x_4)}{x_2+x_4Q(x_3,x_4)}{x_3}{x_4}$ for some polynomial $Q$.
\end{enumerate}
Indeed, the fourth case where $e$ would belong to $\Eb$ is conjugate to the third one.
\end{caselist}

The first two cases are easy to handle.\\

\underline{Case (1).} $e \in \El$.

Since $e' \circ f \in \a_{< \, d}$ and $e \circ e'^{-1} \in \El$, the Induction Hypothesis directly shows us that $ (e \circ e'^{-1}) \circ (e' \circ f) =e \circ f$ belongs to   $\a$.\\

\underline{Case (2).} $e \in \Er$.

\[\mbox{We have  }e' \circ f=\mat{f_1+f_2P(f_2,f_4)}{f_2}{f_3+f_4P(f_2,f_4)}{f_4}\text{ and }e\circ f=\mat{f_1}{f_2+f_1Q(f_1,f_3)}{f_3}{f_4+f_3Q(f_1,f_3)}.\]
By Lemma \ref{lem:degmax} (1), the polynomial $P(f_2,f_4)$ is non-constant, since otherwise we would get $\deg e' \circ f= \deg f$. By Lemma \ref{lem:degree of each component drops}, the inequality $\deg e' \circ f < \deg f$ is equivalent to $\deg (f_1 + f_2 P(f_2,f_4)) < \deg f_1$, so that $\deg f_1 = \deg ( f_2 P(f_2,f_4)) > \deg f_2$. But then, $\deg ( f_2 + f_1 Q(f_1,f_3)) > \deg f_2$, so that Lemma \ref{lem:degree of each component drops} gives us $\deg e \circ f > \deg f$, a contradiction.\\

\underline{Case (3).} $e \in \Et$.

We are in the setting of the following lemma, where Lower bound   \ref{mino:TSL}-\ref{mino:TQ} makes reference either to Lower bound   \ref{mino:TSL} when $G= \TSL$ or to Lower bound   \ref{mino:TQ} when $G= \TQ$.

\begin{lemma} \label{lem:minoration does not apply to both}
Let $f \in G$, and assume that
\[e' \circ f=\mat{f_1+f_2P(f_2,f_4)}{f_2}{f_3+f_4P(f_2,f_4)}{f_4} \quad \text{and} \quad e\circ f=\mat{f_1+f_3Q(f_3,f_4)}{f_2+f_4Q(f_3,f_4)}{f_3}{f_4},\]
with  $\deg e' \circ f < \deg f$ and  $\deg e \circ f \le \deg f$.
Then Lower bound  \ref{mino:TSL}-\ref{mino:TQ} does not apply to either $P(f_2,f_4)$ or $Q(f_3,f_4)$. 
\end{lemma}

\begin{proof}
If Lower bound  \ref{mino:TSL}-\ref{mino:TQ} applies to both $P(f_2,f_4)$ and $Q(f_3,f_4)$, we would obtain the following contradictory sequence of inequalities:
\begin{align*}
\deg f_2 &< \deg (f_4P(f_2,f_4)) && (\text{Lower bound  \ref{mino:TSL}-\ref{mino:TQ} applied to } P); \\
\deg (f_4P(f_2,f_4)) &= \deg f_3 &&  (\deg e' \circ f < \deg f); \\
\deg f_3 &< \deg (f_4Q(f_3,f_4)) && (\text{Lower bound  \ref{mino:TSL}-\ref{mino:TQ} applied to } Q); \\
\deg (f_4Q(f_3,f_4)) &\le \deg f_2 && (\deg e \circ f \le \deg f). \qedhere
\end{align*}
\end{proof}

We conclude the proof of Proposition \ref{pro:main} with the following lemma.

\begin{lemma}
\label{lem:conditions}
If Lower bound  \ref{mino:TSL}-\ref{mino:TQ} does not apply to either $P(f_2,f_4)$ or $Q(f_3,f_4)$, i.e. if one of the four following assertions is satisfied
\[\mbox{(i) } Q(f_3,f_4) \in \C[f_4]; \; \mbox{(ii) } \hom{f_2} \in \C[\hom{f_4}] ;\; \mbox{(iii) } P(f_2,f_4) \in \C[f_4]; \; \mbox{(iv) } \hom{f_3} \in \C[\hom{f_4}],\]
then $e \circ f \in \a$.
\end{lemma}

\begin{proof}
(i) Assume $Q(f_3,f_4)=Q(f_4)\in\C[f_4]$.

Since $e' \circ f \in \a_{< \, d}$ and $e$ is elementary, the Induction Hypothesis  gives us $e \circ e'\circ f \in \a$.

Note that $e\circ {e'}^{-1} \circ e^{-1}$ belongs to $\El$. Therefore, it is enough to show that  $e \circ e'\circ f \in \a_{< \, d}$. Indeed, a new implication of the induction hypothesis will then prove that $(e\circ e'^{-1} \circ e^{-1}) \circ (e \circ e'\circ f) = e \circ f$ belongs to $\a$.

However, we have $\deg e \circ f \leq \deg f$, so that by applying two times Lemma \ref{lem:degree of each component drops}, we successively get $\deg (f_2 + f_4 Q(f_4)) \leq \deg f_2$ and then 
$\deg e \circ e' \circ f \leq \deg e'\circ f$. Since $\deg e' \circ f < \deg f$, we are done.\\

(ii) Assume $\hom{f_2} \in \C[\hom{f_4}]$.

Then there exists $\tilde Q(f_4)\in\C[f_4]$ such that $\deg (f_2+f_4\tilde Q(f_4))<\deg f_2$.
We take 
$$\tilde e =\mat{x_1+x_3\tilde Q(x_4)}{x_2+x_4\tilde Q(x_4)}{x_3}{x_4},$$
and we have  $\tilde e\circ f\in\a$ by case (i). 
Thus $\tilde e\circ f\in\a_{< \, d}$.
Since $e\circ \tilde e^{-1}\in \Et$,  the Induction Hypothesis shows us that $(e\circ \tilde e^{-1}) \circ (\tilde e\circ f) =e \circ f$ belongs to $\a$.\\

(iii) Assume $P(f_2,f_4)=P(f_4) \in \C[f_4]$.

Note that $e' \circ e\circ {e'}^{-1}$ belongs to $\Et$.
By the Induction Hypothesis, we get $(e' \circ e\circ {e'}^{-1}) \circ (e' \circ f) =e' \circ e \circ f \in \a$. 
If we can prove $\deg e' \circ e \circ f < \deg f$ then we can use the Induction Hypothesis again to obtain that  ${e'}^{-1} \circ (e' \circ e \circ f) =e \circ f \in \a$.

We argue as in case (i). We have $\deg e' \circ f < \deg f$, so that by applying two times Lemma \ref{lem:degree of each component drops}, we successively get $\deg (f_3 + f_4 P(f_4)) < \deg f_3$ and then $\deg e' \circ e \circ f <\deg e\circ f$. Since $\deg e \circ f \leq \deg f$, we are done.\\

(iv) Finally assume $\hom{f_3} \in \C[\hom{f_4}]$.

There exists $\tilde P(f_4)\in\C[f_4]$ such that $\deg (f_3+f_4\tilde P(f_4)) <\deg f_3$.
We take 
$$\tilde e =\mat{x_1+x_2\tilde P(x_4)}{x_2}{x_3+x_4\tilde P(x_4)}{x_4},$$
and we have  $\tilde e\circ f\in\a$ by the easy first case of List of Cases \ref{list:3cases}. 
Thus $\tilde e\circ f\in\a_{< \, d}$.
Therefore, we may replace $e'$ by $\tilde e$ and then we conclude by case (iii). 
\end{proof}

\bibliographystyle{myalpha.bst} 
\bibliography{biblio}

\end{document}